\documentclass[oneside]{article}
\usepackage[a4paper,left=30mm,right=30mm,top=30mm,bottom=30mm]{geometry}
\usepackage{fancyhdr}
\title{\LARGE{\bfseries{Homogenization of degenerate coupled transport processes in porous media with memory terms}}}
\date{}
\author{
\large{
Michal Bene\v s\footnote{Department of Mathematics,
Faculty of Civil Engineering, Czech Technical University in Prague,
Th\'{a}kurova 7, 166 29 Prague 6, Czech Republic,
E-mail: michal.benes@cvut.cz}\;
and
Igor Pa\v{z}anin\footnote{Department of Mathematics, Faculty of Science, University of Zagreb,
Bijeni\v{c}ka 30, 10000 Zagreb, Croatia,
E-mail: pazanin@math.hr}
}
}

\usepackage{amssymb}
\usepackage{amsfonts}
\usepackage{graphicx}%[dvips]
\usepackage{amsmath}
\usepackage{latexsym}
\usepackage{amssymb}
\usepackage{pstricks}
\usepackage{amsmath}
\usepackage{graphics}
\usepackage{mathrsfs}
\usepackage{fancyhdr}
\usepackage{times}

\newtheorem{theorem}{Theorem}[section]
\newtheorem{definition}[theorem]{Definition}
\newtheorem{lemma}[theorem]{Lemma}

\newtheorem{remark}[theorem]{Remark}

\newenvironment{proof}[1][Proof.]{\begin{trivlist}
\item[\hskip \labelsep {\bfseries #1}]}{\end{trivlist}}

 \DeclareMathOperator{\ess}{ess}

\newcommand{\bfn}{\mbox{\boldmath{$n$}}}
\newcommand{\bfe}{\mbox{\boldmath{$e$}}}

\newcommand{\bfg}{\mbox{\boldmath{$g$}}}
\newcommand{\bfq}{\mbox{\boldmath{$q$}}}

\newcommand{\bfu}{\mbox{\boldmath{$u$}}}
\newcommand{\bfw}{\mbox{\boldmath{$w$}}}

\newcommand{\bfv}{\mbox{\boldmath{$v$}}}

\newcommand{\bfJ}{\mbox{\boldmath{$J$}}}

\newcommand{\bfI}{\mbox{\boldmath{$I$}}}

\newcommand{\bfA}{\mbox{\boldmath{$A$}}}

\newcommand{\bfLambda}{\mbox{\boldmath{$\Lambda$}}}

\newcounter{constants}
\begin{document}

\maketitle

\begin{abstract}
In this paper
we establish a homogenization result for a doubly nonlinear parabolic system
arising from the hygro-thermo-chemical processes in porous media taking into account
memory phenomena. We present a meso-scale model of the composite (heterogeneous) material
where each component
is considered as a porous system and the voids
of the skeleton are partially saturated with liquid water.
It is shown that the solution of the meso-scale problem
is two-scale convergent to that of the upscaled problem as the spatial
parameter goes to zero.
\end{abstract}

%%%%%%%%%%%%%%%%%%%%%%%%%%%%%%%%%%%%%%%%%%%%%%%%%%%%%%%%%%%%%%%%%%%%%%%%%%%%%%%%%%%%%%%%
\section{Introduction}

Mathematical modeling of coupled transport processes in
heterogeneous porous media is of great interest in engineering
practice (civil engineering, environmental engineering, nuclear engineering applications).
Coupled heat and moisture flow through multiphase porous material
is associated with systems of strongly coupled nonlinear PDEs of the form \cite{bear,Pinder}
\begin{align}
\partial_t\varrho_{\alpha}
+
\nabla\cdot  \bfJ_{\alpha}
& =
{s}_{\alpha}
&&
\textmd{(conservation of mass)},
\label{eq:par1}
\\
\partial_t  e_{\alpha}
+
\nabla \cdot  \bfq_{\alpha}
& =
\mathcal{Q}_{\alpha}
+
\mathcal{E}_{\alpha}
-
H_{\alpha} {s}_{\alpha}
&&
\textmd{(conservation of energy)}.
\label{eq:par2}
\end{align}
Here,
$\varrho_{\alpha}$ represents the averaged mass density of the $\alpha$-phase (e.g. solid,
 liquid water, oil,  gas, etc.),
$\bfJ_{\alpha}$ is the mass flux and  ${s}_{\alpha}$ stands for a production term.
Further,
$e_{\alpha}$ is the total internal
energy of the $\alpha$-phase,
$\bfq_{\alpha}$ is the heat flux,
$\mathcal{Q}_{\alpha}$ stands for the volumetric heat source,
$\mathcal{E}_{\alpha}$ represents the term
expressing energy exchange with the other phases
and the symbol $H_{\alpha}$  stands for the specific enthalpy
of the $\alpha$-phase.

Because of the nonlinear structure, complexity and multi-scale nature of the problem \eqref{eq:par1}--\eqref{eq:par2},
this system must be solved numerically using suitable computational methods.
Nevertheless, the complexity of the microscopic structure of heterogeneous multiphase materials
makes detailed numerical simulations very expensive.
Therefore, the effective (homogenized) material coefficients are used instead
of taking into account properties of individual phases.
The composition of the heterogeneous material may vary substantially and
the experimental determination of effective properties is
rather complicated or expensive due to lack of enough accurate data.
Therefore,
recently there has been an explosive growth of interest into the development
of mathematical homogenization methods
to derive effective material properties of heterogeneous media directly from their microstructure.
The present paper is devoted to the periodic homogenization of
a system of degenerate partial differential equations
with memory phenomena using the two-scale homogenization technique.
The model covers a large range of problems, in particular,
modelling  of hygro-thermo-chemical processes in concrete at early ages taking into account
hydration phenomena. In the meso-scale analysis below,
fresh concrete is treated as a composite
material where each component (cement paste, aggregates)
is considered as a porous medium where the voids
of the skeleton (both, cement paste and aggregates) are partially saturated with liquid water.

The paper is organized as follows.
In Section~\ref{sec:meso_scale_model},
we describe the meso-scale model for transport processes in early age concrete
taking into account hydration phenomena
and review homogenization results for degenerate parabolic problems.
In Section~\ref{sec:preliminaries},
we introduce some notation and describe various function spaces
and precisely specify our assumptions on data and coefficient functions
under which the main result of the paper is proven.
In Section~\ref{sec:main_result},
we provide the weak formulation of the meso-scale problem,
present the existence theorem
(for the proof see Appendix~\ref{proof_1} and \cite{BenesPazanin2018})
and state the main result of the paper, see Theorem~\ref{thm:main_result}.
The main result is proven in Section~\ref{sec:proof_main}.
We first establish a priori estimates for solutions of the meso-scale problem,
uniform with respect to $\varepsilon$,
see Subsection~\ref{subsec:a_priori_estimates}.
Having obtained a priori estimates,
in Subsection~\ref{subsec:passage_limit},
we collect the definition and basic results for the weak two-scale convergence and
pass to the homogenization limit  $\varepsilon \rightarrow 0$,
where
$\varepsilon$ is the characteristic length representing the small scale variability of concrete.
Transport coefficients and material properties depend no longer of material heterogeneities.
Then, the proof of the main result is completed in Subsection~\ref{subsec:homogenized_model}
eliminating the microscopic variable from the upscaled system and
decoupling the cell problems from the two-scale homogenized problem.

%%%%%%%%%%%%%%%%%%%%%%%%%%%%%%%%%%%%%%%%%%%%%%%%%%%%%%%%%%%%%%%%%%%%%%%%%%%%%%%%%%%%%
\section{The heterogeneous meso-scale model}
\label{sec:meso_scale_model}
Let $\Omega$ be a bounded domain in $\mathbb{R}^2$
with Lipschitz boundary $\partial \Omega$.
$\bfn$~denotes the outer unit normal vector to $\partial\Omega$.
Let $T \in (0,\infty)$ be fixed throughout the paper,
$\Omega_T := \Omega\times (0,T)$ and $\partial\Omega_T := \partial\Omega\times (0,T)$.
Let
$\mathcal{Y}=(0,1)^2$ be a periodicity cell. % of unit volume.
We consider a porous structure consisting of two distinct flow regions
periodically distributed in a domain $\Omega$
with period $\varepsilon \mathcal{Y}$.
Let $\mathcal{Y}$ be split on two complementary parts
$\mathcal{Y}_{a}$ and $\mathcal{Y}_{c}$. We denote by $\chi_{a}(y)$
and  $\chi_{c}(y)$ the corresponding characteristic functions  of
$\mathcal{Y}_{a}$ and $\mathcal{Y}_{c}$, respectively,
extended $\mathcal{Y}$-periodically to all of $\mathbb{R} \times \mathbb{R}$.
Hence, the domain $\Omega$ is decomposed into two sub-domains,
$\Omega_{a}^{\varepsilon}$ (aggregates) and $\Omega_{c}^{\varepsilon}$ (cement paste),
which are defined as
$$
\Omega_{a}^{\varepsilon}
\equiv
\left\{
x \in \Omega;
\quad
\chi_{a}({x/\varepsilon}) = 1
\right\}
\qquad  \textmd{ and }  \qquad
\Omega_{c}^{\varepsilon}
\equiv
\left\{
x \in \Omega;
\quad
\chi_{c}({x/\varepsilon}) = 1
\right\}.
$$
The characteristic functions $\chi_{a}$ and $\chi_{c}$
are used as multipliers to denote the zero-extension of various
functions.
From the geometrical point of view,
$\varepsilon$ is the characteristic length representing the small scale variability of concrete,
see Figure~\ref{figure_periodic_medium}.

\bigskip

\begin{figure}[h]
\centering %\captionstyle{flushleft}
\includegraphics[width=10cm]{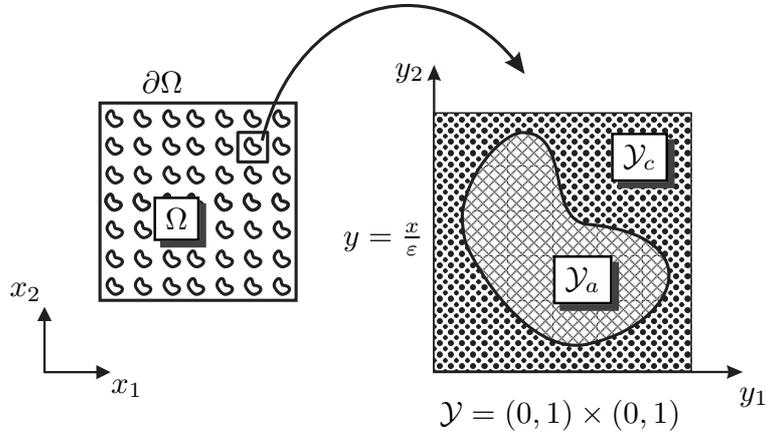}
\caption{2D example of periodic medium.}
\label{figure_periodic_medium}
\end{figure}

In this work we assume the particular form of \eqref{eq:par1}--\eqref{eq:par2}
for the two-phase (solid skeleton and liquid water)
heterogeneous porous system (cement paste and aggregates)
and study the homogenization
of the system of partial differential equations in $\Omega_{T}$
indexed by the scale parameter $\varepsilon$, namely
\begin{eqnarray}
\partial_t    {b}({x/\varepsilon},{p}^{\varepsilon},{r}^{\varepsilon})
-
\nabla\cdot
\left[
{a}({x/\varepsilon},{p}^{\varepsilon},\vartheta^{\varepsilon},{r}^{\varepsilon})
\nabla {p}^{\varepsilon}
\right]
&=&
\alpha_1  \chi_{c}({x/\varepsilon})  f({p}^{\varepsilon},\vartheta^{\varepsilon},{r}^{\varepsilon})
,
\label{eq1a}
\\
\partial_t
\left[
c_{w}
{b}({x/\varepsilon},{p}^{\varepsilon},{r}^{\varepsilon})
\vartheta^{\varepsilon}
+
\sigma({x/\varepsilon},{r}^{\varepsilon}) \vartheta^{\varepsilon}
\right]
\nonumber
\\
-
\nabla\cdot
\left[
{\lambda}({x/\varepsilon},{p}^{\varepsilon},\vartheta^{\varepsilon},{r}^{\varepsilon})
\nabla \vartheta^{\varepsilon}
+
c_{w}
\vartheta^{\varepsilon}
{a}({x/\varepsilon},{p}^{\varepsilon},\vartheta^{\varepsilon},{r}^{\varepsilon})
\nabla {p}^{\varepsilon}
\right]
&=&
\alpha_2  \chi_{c}({x/\varepsilon})   f({p}^{\varepsilon},\vartheta^{\varepsilon},{r}^{\varepsilon})
%\;
%{\rm in} \; {\Omega_T}
,
\label{eq1b}
\end{eqnarray}
coupled with an integral condition
\begin{equation}\label{eq1c}
{r}^{\varepsilon}(x,t)
=
\int_0^t
f({p}^{\varepsilon}(x,s),\vartheta^{\varepsilon}(x,s),{r}^{\varepsilon}(x,s))
\,
{\rm d}s
\end{equation}
and boundary and initial conditions
\begin{align}
-{a}({x/\varepsilon},{p}^{\varepsilon},\vartheta^{\varepsilon},{r}^{\varepsilon})
\nabla {p}^{\varepsilon}
\cdot
\bfn
&=
\beta_{e}
(
{p}^{\varepsilon}
-
{p}_{\infty}
)
\quad
&&
{\rm in} \; \partial\Omega_{T},
\label{eq1e}
\\
-
{\lambda}({x/\varepsilon},{p}^{\varepsilon},\vartheta^{\varepsilon},{r}^{\varepsilon})
\nabla \vartheta^{\varepsilon} \cdot \bfn
&=
\alpha_{e} (\vartheta^{\varepsilon} - \vartheta_{\infty})
\quad
&&
{\rm in} \; \partial\Omega_{T},
\label{eq1d}
\\
{p}^{\varepsilon}(0)
&=
{p}_{0}
\quad
&&
{\rm in} \;  \Omega,
\label{eq1f}
\\
\vartheta^{\varepsilon}(0)
&=
\vartheta_0
\quad
&&
{\rm in} \;  \Omega.
\label{eq1g}
\end{align}
In the system of equations above,
to shorten mathematical formulations, we have used the following simplified
notations:
\begin{eqnarray}
&&
{b}^{\varepsilon}({x},{p}^{\varepsilon},{r}^{\varepsilon})
=
{b}({x/\varepsilon},{p}^{\varepsilon},{r}^{\varepsilon})
:=
\varrho_{w}
\left[
\chi_{c}({x/\varepsilon})
\phi_{c}({r}^{\varepsilon})
+
\chi_{a}({x/\varepsilon})
\phi_{a}
\right]
\mathcal{S}({p}^{\varepsilon}),
\label{notation_form_b}
\\
&&
\sigma^{\varepsilon}({x},{r}^{\varepsilon})
=
\sigma({x/\varepsilon},{r}^{\varepsilon})
:=
\chi_{c}({x/\varepsilon})
\varrho_{sc} {c}_{sc}(1 - \phi_{c}({r}^{\varepsilon}) )
+
\chi_{a}({x/\varepsilon})
\varrho_{sa} {c}_{sa}(1 - \phi_{a} ),
\label{notation_form_sigma}
\\
&&
{a}^{\varepsilon}({x},{p}^{\varepsilon},\vartheta^{\varepsilon},{r}^{\varepsilon})
=
{a}({x/\varepsilon},{p}^{\varepsilon},\vartheta^{\varepsilon},{r}^{\varepsilon})
:=
\varrho_{w}
\frac{{k}_{R}(\mathcal{S}({p}^{\varepsilon}))}{\mu(\vartheta^{\varepsilon})}
\left[
\chi_{c}({x/\varepsilon})
{k}_{c}({r}^{\varepsilon})
+
\chi_{a}({x/\varepsilon})
{k}_{a}
\right],
\label{vector_flux_water}
\\
&&
{\lambda}^{\varepsilon}({x},{p}^{\varepsilon},\vartheta^{\varepsilon},{r}^{\varepsilon})
=
{\lambda}({x/\varepsilon},{p}^{\varepsilon},\vartheta^{\varepsilon},{r}^{\varepsilon})
:=
\chi_{c}({x/\varepsilon})
\lambda_{c}({p}^{\varepsilon},\vartheta^{\varepsilon},{r}^{\varepsilon})
+
\chi_{a}({x/\varepsilon})
\lambda_{a}({p}^{\varepsilon},\vartheta^{\varepsilon}).
\label{vector_flux_heat}
\end{eqnarray}
From the physical point of view,
equations \eqref{eq1a} and \eqref{eq1b}, respectively, represent the mass balance of moisture (liquid water)
and the heat equation for the porous system.
The equation \eqref{eq1c} represents an integral condition with an additional unknown \emph{memory function}
${r}^{\varepsilon} : {\Omega_T}\rightarrow \mathbb{R}$.
%different branches of physics,
%in hydrology  (contaminant transport with adsorption) \cite{Kacur1999,Kacur1999b}
%or
Such a type of equations arises in the theory of heat conduction
when inner heat sources are of special types,
in particular, in so-called problems of hydratational heat during cement hydration.
In this case, the intensity of inner sources
of heat in cement paste (and corresponding sinks of liquid water)
depends also on the amount of heat already developed.
Typical example of problem \eqref{eq1a}--\eqref{eq1c}, modeling heat and moisture transport
in early age concrete, is given in~\cite{BenesPazanin2018}.

Now, let us describe the notation in \eqref{eq1a}--\eqref{vector_flux_heat} in more details.
Next to the memory function ${r}^{\varepsilon}$,
${p}^{\varepsilon} : {\Omega_T}\rightarrow \mathbb{R}$
and
$\vartheta^{\varepsilon} : {\Omega_T}\rightarrow \mathbb{R}$
denote the unknown water pressure and temperature of the porous system, respectively.
In the model, we assume different hydraulic and thermal characteristics in $\Omega_{a}$
and $\Omega_{c}$, respectively.
In particular,
$\phi_{\alpha}$, $\alpha \approx a,c$, is the porosity
and
$\mathcal{S}$ represents the degree of saturation with liquid water.
Further,
${k}_{\alpha}$ is the intrinsic permeability of the fluid,
${k}_{R}$ represents the relative hydraulic conductivity,
$\lambda_{\alpha}$ is the thermal conductivity function
and
$\mu$ is the temperature dependent kinematic viscosity of the fluid.
Material constants are as follows:
$\varrho_w$ is the density of liquid water and
${c}_w$ represents the isobaric heat capacity of water.
Moreover, $\varrho_{s\alpha}$ and ${c}_{s\alpha}$,
respectively, are the mass density and the isobaric heat capacity of the solid.
$\alpha_{e}$ designates the film coefficient for the heat transfer,
$\beta_{e}$ represents the surface emissivity of water,
$\vartheta_{\infty}$ is the temperature of the environment
and
${p}_{\infty}$ is a fictitious water pressure related to the ambient conditions
(the relative humidity, gas pressure and temperature).
Finally,
${p}_{0} : \Omega \rightarrow \mathbb{R}$
and ƒ
$\vartheta_0 : \Omega\rightarrow \mathbb{R}$
are the initial distributions of the water pressure and temperature.
Recall that the characteristic functions ${\chi}_{a}$ and ${\chi}_{c}$
are $\mathcal{Y}$-periodic functions on $\mathbb{R}^2$.
As the spatial parameter $\varepsilon$ gets smaller,
the characteristic functions ${\chi}_{a}$ and ${\chi}_{c}$
oscillate more rapidly.
We solve the periodic homogenization problem and investigate the behavior of the solution
in the limit (as $\varepsilon \rightarrow 0$).

The system \eqref{eq1a}--\eqref{eq1b}
can be written in terms of operators ${A}$, ${B}$, ${F}$
in general form
\begin{equation}\label{eq:par3}
\partial_t {B}(x/\varepsilon,\bfu)
-
\nabla\cdot {A}(x/\varepsilon,\bfu,\nabla \bfu)
=
{F}(x/\varepsilon,\bfu),
\end{equation}
where $\bfu$ is the unknown vector of state variables.
The existence, uniqueness and regularity of the solution to the equation \eqref{eq:par3}
have been studied (assuming the operator $B$ in the parabolic
part to be subgradient) by several authors, see e.g. \cite{AltLuckhaus1983,FiloKacur1995,jungel2000,otto}.
A theoretical analysis of special type of \eqref{eq:par3} was also performed in
\cite{BenesKrupicka2015} and \cite{BenesKrupicka2016}
in the case of degenerate coupled flows in geomaterials and
\cite{BenesZeman} in the particular non-degenerate case with applications to nonlinear heat and
moisture transport in multi-layer porous structures.

Homogenization of \eqref{eq:par3} has been studied e.g. in \cite{cao2013,cao2014,jian2000,Nandakumaran2001}.
In \cite{cao2013}, the authors studied the homogenization of the
scalar problem assuming ${B}(y,\bfu)=u$, however, including
more general nonlinear elliptic operator of the form
${A}(y,\bfu,\nabla \bfu) = {D}(y,u,\nabla u) + {K}(y,u)$.
The operator ${D}(y,u,\nabla u)$ was assumed to be periodic in $y={x/\varepsilon}$
and degenerate in $\nabla u$.
In \cite{cao2014}, the authors studied
the homogenization of multi-scale degenerate problem
(arising in the motion of saturated-unsaturated water flow in porous media)
like \eqref{eq:par3} with
$
{A}(y,\bfu,\nabla \bfu)
=
-
\left[
a({x/\varepsilon})
\nabla u^{\varepsilon}
+
\bfg({x/\varepsilon},u^{\varepsilon})\right]
$
and
${B}(x/\varepsilon,\bfu) = {B}(u)$.
The authors derived the homogenized equation and present
results on the first order corrector.
In \cite{Nandakumaran2001}, the authors studied
the homogenization of nonlinear parabolic equations like
\eqref{eq1a} with mixed boundary conditions under
restrictive assumptions on $B$ and with non-degenerate and monotone
elliptic part.
In \cite{jian2000}, the author studied the asymptotic behaviour of the problem
with ${B}(y,\bfu)=|u|^{\nu} {\rm sign}\, u$, $\nu>0$, including
fast or slow diffusion through a nonhomogeneous medium.
Homogenization of a coupled system of diffusion-convection
equations in a domain with periodic microstructure, modeling the flow
of isothermal immiscible compressible fluids through porous media,
was theoretically studied e.g. in \cite{amaziane2010a, amaziane2013a}.

%\cite{Clark}

%In \cite{Krehel2014}, the authors applied two-scale
%convergence techniques to obtain the macro model of thermal-diffusion
%reaction problems (including the Dufour ad Soret effects)
%modeling e.g. behavior of concrete at high temperatures, drug delivery
%in biology tissues etc.

In the present paper,
our aim is to extend our previous result from \cite{BenesPazanin2017b},
where we established a homogenization result for a fully nonlinear
degenerate parabolic system (with mixed homogeneous Dirichlet-Neumann boundary conditions)
arising  from the heat and moisture flow through partially saturated porous media.
In particular,
following  \cite{AltLuckhaus1983},
the degeneracy in the elliptic part of the mass balance equation (liquid water)
has been transformed to the parabolic term
using the so called Kirchhoff transformation
\begin{displaymath}%
v := \kappa(p) = \int\limits_{0}^{p} k_{R}(\mathcal{S}(s))
{\rm d}s.
\end{displaymath}
However, due to the degeneracy of the problem
($k_{R}$ is assumed to be bounded below by zero)
we are not able to ensure that the inverse function $p=\kappa^{-1}(v)$ solves the original problem.
Therefore we consider directly the original problem  \eqref{eq1a}--\eqref{eq1b}
(in terms of the pressure and temperature)
and omit degeneracies using $L^{\infty}$-estimates for the solution.
Moreover, in the present paper, we assume physically more relevant Newton type boundary conditions
and incorporate memory function in source terms of the governing equations.

%%%%%%%%%%%%%%%%%%%%%%%%%%%%%%%%%%%%%%%%%%%%%%%%%%%%%%%%%%%%%%%%%%%%%%%%%%%%%%%%
\section{Preliminaries}\label{sec:preliminaries}
\subsection{Some functions spaces}
Throughout the paper, we will always use positive constants $C$,
$c$, $c_1$, $c_2$, $\dots$, which are not specified and which may
differ from line to line.
${\bfI_d}$ denotes the identity matrix of size $2\times 2$.
We suppose
$p,q,p'\in [1,\infty]$, $p'$ denotes the conjugate exponent to $p$,
$p>1$, ${1}/{p} + {1}/{p'} = 1$.
$L^p(\Omega)$ denotes the usual
Lebesgue space equipped with the norm $\|\cdot\|_{L^p(\Omega)}$ and
$W^{k,p}(\Omega)$, $k\geq 0$ ($k$ need not to be an integer, see
\cite{KufFucJoh1977}), denotes the usual
Sobolev-Slobodecki space with the norm $\|\cdot\|_{W^{k,p}(\Omega)}$.
Recall that
$W^{0,p}(\Omega)=L^p(\Omega)$.
By  ${X}'$ we denote the space of all continuous, linear forms on Banach space ${X}$
and by $\langle \cdot,\cdot \rangle$ we denote the
duality between ${X}$ and ${X}'$. As usual, $(\cdot,\cdot)$
stands for the inner product on $L^2(\Omega)$.
By $L^p(0,T;X)$ we denote the usual Bochner space (see \cite{AdamsFournier1992}).
%In the paper we shall use the following embedding theorems
%(see \cite{AdamsFournier1992,KufFucJoh1977}):
%%
%\begin{equation}
%\label{embedding_theorems}
%\left\{
%\begin{array}{lll}
%W^{k,p}(\Omega) \hookrightarrow L^q(\Omega),
%&
%\|\varphi\|_{L^q(\Omega)} \leq c \,\|\varphi\|_{W^{k,p}(\Omega)},
%&
%1\leq q <+\infty, \, kp=2,
%\\
%W^{k,p}(\Omega) \hookrightarrow L^q(\Omega),
%&
%\|\varphi\|_{L^q(\Omega)} \leq c \,\|\varphi\|_{W^{k,p}(\Omega)},
%&
%1\leq q \leq 2\,p/(2-kp),
%\\
%& & kp<2,
%\\
%W^{k,p}(\Omega) \hookrightarrow L^{\infty}(\Omega),
%&
%\|\varphi\|_{L^{\infty}(\Omega)} \leq c \,\|\varphi\|_{W^{k,p}(\Omega)},
%&
%kp>2
%\end{array}\right.
%\end{equation}
%for every $\varphi \in W^{k,p}(\Omega)$.
%
%
Let $W^{1,2}_{per}(\mathcal{\mathcal{Y}})$ be the space of elements of
$W^{1,2}(\mathcal{\mathcal{Y}})$ having the same trace on opposite face of $\mathcal{\mathcal{Y}}$.
$L^{p}_{per}(\mathcal{\mathcal{Y}})$ is the subspace of $L^p(\mathcal{Y})$
of $\mathcal{Y}$-periodic functions $\varphi$, i.e. $\varphi(x+k\bfe_i)=\varphi(x)$
a.e. on $\mathbb{R}^2$, for all $k \in \mathbb{Z}$ and $i\in \left\{1,2\right\}$,
where $\left\{\bfe_1,\bfe_2 \right\}$ is the canonical basis of $\mathbb{R}^2$.
By $C_{per}({\mathcal{Y}})$ we denote the subspace of $C(\mathbb{R}^{2})$ of $\mathcal{Y}$-periodic functions.
By $C^{\infty}_{per}({\mathcal{Y}})$ we denote the subspace of
$C^{\infty}(\overline{\mathcal{Y}})$ of $\mathcal{Y}$-periodic functions.
$L^{p}(\Omega;X)$ stands for the set of measurable functions $\varphi: x \in \Omega \rightarrow
\varphi(x) \in X$ such that $\|\varphi(x)\|_{X} \in L^p(\Omega)$.
$L^{2}_{per}(\mathcal{Y};C(\overline{\Omega}))$ represents the space of measurable functions
$\varphi: y \in \mathcal{Y} \rightarrow
\varphi(y) \in C(\overline{\Omega})$ such that $\|\varphi(y)\|_{C(\overline{\Omega})} \in L^2_{per}(\mathcal{Y})$.

%%%%%%%%%%%%%%%%%%%%%%%%%%%%%%%%%%%%%%%%%%%%%%%%%%%%%%%%%%%%%%%%%%%%%%%%%%%%%%%%%%%%%%%%
\subsection{Structure and data properties}

Before we state the main result of our paper presented in the next section,
let us introduce the assumptions on coefficient functions in
\eqref{eq1a}--\eqref{vector_flux_heat}:

\begin{itemize}

\item[(i)]
$\mathcal{S} \in C^{1}(\mathbb{R})$ is a positive and monotone function such that
\begin{equation}
0 < \mathcal{S}(\xi) \leq {C}_{s} < +\infty,
\quad
 0 < \mathcal{S}'(\xi)  \leq {S}_{L}
 \qquad
\forall \xi \in \mathbb{R} \quad ({C}_{s}, {S}_{L} = {\rm const}).
\label{con11a}
\end{equation}

\item[(ii)]
${k}_{c}$, ${k}_{R}$, $\mu$, $\lambda_{a}$ and $\lambda_{c}$ are continuous functions
satisfying
\begin{align}
&
0 < k_1 \leq {k}_{c}(\xi) \leq k_2 < +\infty \quad (k_1,k_2 = {\rm const})
&&
\forall \xi \in \mathbb{R},
\label{cond:intr_perm}
\\
&
k_{R} \in C([0,{C}_{s}]), \;
\left(k_{R}(\xi_1)-k_{R}(\xi_2)\right)(\xi_1 - \xi_2) > 0
&&
\forall \xi_1,\xi_2 \in [0,{C}_{s}], \; \xi_1 \neq \xi_2,
\label{cond:rel_perm_I}
%\\
%&
%&&
%\xi_1 \neq \xi_2,
%\nonumber
\\
&
0 <  k_{R}(\xi)
&&
\forall \xi \in [0,{C}_{s}],
\label{cond:rel_perm_II}
\\
&
0 <  \mu_1 \leq  \mu(\xi) \leq  \mu_2 < +\infty \quad ( \mu_1, \mu_2 = {\rm const})
&&
\forall \xi \in \mathbb{R},
\label{cond:viscosity}
\\
&
0  <  \lambda_{1} < \lambda_{c}(\xi_1,\xi_2,\xi_3)  <  \lambda_{2} < +\infty
 \quad (\lambda_{1},\lambda_{2} = {\rm const})
&&
\forall \xi_1,\xi_2,\xi_3 \in \mathbb{R},
\label{con12c}
\\
&
0  <  \lambda_{1} < \lambda_{a}(\xi_1,\xi_2)  <  \lambda_{2} < +\infty
 \quad (\lambda_{1},\lambda_{2} = {\rm const})
&&
\forall \xi_1,\xi_2 \in \mathbb{R}.
\label{con12d}
\end{align}

\item[(iii)] The function $\phi_{c}$ is Lipschitz continuous, i.e.
there exists a constant $C_{\phi}>0$ such that
\begin{align}\label{lipschitz_phi}
| \phi_{c}(\xi_1) - \phi_{c}(\xi_2) |
&
\leq
C_{\phi}
|\xi_1 - \xi_2|
&&
\forall \xi_1,\xi_2 \in \mathbb{R}
\end{align}
and
\begin{equation}\label{bound_phi}
0 < \phi_1 \leq \phi_{c}(\xi) \leq \phi_2 < +\infty
\quad
\forall \xi \in \mathbb{R}
\quad
(\phi_1,\phi_2 = {\rm const}).
\end{equation}

\item[(iv)]
$f$ is Lipschitz continuous in all respective variables and
\begin{equation}\label{assum:bound_f}
0  \leq f(\xi_1, \xi_2, \xi_3)
\leq
C_f
\quad
\forall \xi_1, \xi_2, \xi_3 \in \mathbb{R}
\qquad (C_f = {\rm const}).
\end{equation}

Further, we assume
\begin{equation}\label{assum:zero_hydration_f}
f(\xi_1, \xi_2, \xi_3) =0
\quad
\forall \xi_{1} \leq {p}_{\infty}.
\end{equation}

\item[(v)] (Material constants. Boundary and initial data)
$\alpha_{e}$, $\beta_{e}$, $\alpha_1$, $\alpha_2$, ${p}_{\infty}$ and $\vartheta_{\infty}$
are assumed to be given constants, such that
\begin{equation}\label{material_constants}
\alpha_{e}>0, \; \beta_{e}>0, \; \alpha_1<0, \; \alpha_2>0, \; {p}_{\infty} <0, \; \vartheta_{\infty}>0
\end{equation}
and
\begin{equation}\label{material_constants_condition_drying}
\alpha_1 +  \varrho_w  C_{\phi} {C}_{s} <0.
\end{equation}

Finally, we assume ${\vartheta}_{0} \in W^{1,2}\cap L^{\infty}(\Omega)$
and
${p}_{0} \in L^{\infty}(\Omega)$
such that
\begin{equation}\label{cond:lower_bound_pressure}
{p}_{\infty} < {p}_0(\cdot)  \leq 0   \qquad \textmd{ a.e. in } \Omega.
\end{equation}

\end{itemize}
Throughout the paper the hypotheses (i)--(v) will be assumed.

%%%%%%%%%%%%%%%%%%%%%%%%%%%%%%%%%%%%%%%%%%%%%%%%%%%%%%%%%%%%%%%%%%%%%%%%%%
\section{Main result}
\label{sec:main_result}

The existence of the so called weak solution
to the problem \eqref{eq1a}--\eqref{eq1g}
has been recently proven in \cite{BenesPazanin2018}
considering mixed Dirichlet-Neumann boundary conditions.
The existence proof for the problem with the Newton boundary conditions
is briefly sketched in Appendix~\ref{proof_1}.
The main result of the present paper is devoted to the homogenization
of the system \eqref{eq1a}--\eqref{eq1g} based on the two-scale convergence theory,
see Theorem~\ref{thm:main_result} below.
%To the best of our knowledge, mathematical homogenization of
%degenerate coupled-transport problems in partially saturated
%porous media remains largely unexplored and remains open due to the nonlinear coupling and
%degeneracies of the system.

\bigskip

We first formulate the problem \eqref{eq1a}--\eqref{eq1g} in a weak sense.

\begin{definition}\label{def:weak_form}

Let $\varepsilon > 0$ be given.
We say that a triple
$[{p}^{\varepsilon},\vartheta^{\varepsilon},{r}^{\varepsilon}]$,
such that
\begin{align*}
&
{p}^{\varepsilon}
\in
L^2(0,T;{{W}^{1,2}(\Omega)}) \cap L^{\infty}({\Omega_T})   \cap L^{\infty}(\partial\Omega_T),
\\
&
\vartheta^{\varepsilon}
\in
L^2(0,T;{{W}^{1,2}(\Omega)})  \cap L^{\infty}({\Omega_T}),
\\
&
{r}^{\varepsilon} \in C([0,T];L^{\infty}(\Omega)) \cap  {L}^{2}(0,T;{W}^{1,2}(\Omega)),
\end{align*}
is a weak solution of the system
\eqref{eq1a}--\eqref{eq1g} iff
\begin{itemize}

\item[(i)]
$\partial_t  {b}({\cdot},{r}^{\varepsilon},{p}^{\varepsilon})
\in L^2(0,T;{{{{W}^{1,2}(\Omega)}'}})$,
$\partial_t \left[
c_{w}{b}({\cdot},{r}^{\varepsilon},{p}^{\varepsilon}) \vartheta^{\varepsilon}
+
\sigma({\cdot},{r}^{\varepsilon}) \vartheta^{\varepsilon}
\right]
\in L^2(0,T;{{{{W}^{1,2}(\Omega)}'}})$
and
\begin{equation}\label{eq:defn_weak_01}
\int_0^T
\langle
\partial_t {b}({x/\varepsilon},{p}^{\varepsilon},{r}^{\varepsilon}) ,\varphi
\rangle
{\rm d}t
+
\int_{\Omega_{T}}
\left[
 {b}({x/\varepsilon},{p}^{\varepsilon},{r}^{\varepsilon}) - {b}({x/\varepsilon},0,{p}_{0})
\right]
\partial_t \varphi
{\rm d}x {\rm d}t
=
0
\end{equation}
for every test function $\varphi \in L^2(0,T;{{W}^{1,2}(\Omega)})\cap W^{1,1}(0,T;L^{\infty}(\Omega))$ with
$\varphi(T)=0$ and
\begin{multline}\label{eq:defn_weak_02}
\int_0^T  \langle
\partial_t
\left[
c_{w}
{b}({x/\varepsilon},{p}^{\varepsilon},{r}^{\varepsilon}) \vartheta^{\varepsilon}
+
\sigma({x/\varepsilon},{r}^{\varepsilon}) \vartheta^{\varepsilon}
\right],\psi
\rangle {\rm d}t
\\
+
\int_{\Omega_T}
\left[
\left(
c_{w}
{b}({x/\varepsilon},{p}^{\varepsilon},{r}^{\varepsilon})
+
\sigma({x/\varepsilon},{r}^{\varepsilon})
\right)
\vartheta^{\varepsilon}
-
\left(
c_{w}
{b}({x/\varepsilon},0,{p}_{0})
+
\sigma({x/\varepsilon},{0})
\right)
\vartheta_0
\right]
\partial_t \psi
{\rm d}x {\rm d}t
=
0
\end{multline}
for every test function $\psi \in L^2(0,T; {{W}^{1,2}(\Omega)})\cap W^{1,1}(0,T;L^{\infty}(\Omega))$ with
$\psi(T)=0$;

\item[(ii)]
the triple $[{p}^{\varepsilon},\vartheta^{\varepsilon},{r}^{\varepsilon}]$ satisfies the following system:
\begin{multline}\label{eq:defn_weak_03}
\int_0^T
\langle
\partial_t {b}({x/\varepsilon},{p}^{\varepsilon},{r}^{\varepsilon})  ,  \varphi
\rangle
{\rm d}t
+
\int_{\Omega_T}
{a}({x/\varepsilon},{p}^{\varepsilon},\vartheta^{\varepsilon},{r}^{\varepsilon})
\nabla {p}^{\varepsilon}
\cdot \nabla\varphi
{\rm d}x {\rm d}t
+
\int_{\partial\Omega_T}
\beta_{e}
{p}^{\varepsilon} \varphi
\,
{\rm d}{S} {\rm d}t
\\
=
\int_{\Omega_T}
\alpha_1   \chi_{c}({x/\varepsilon})  f({p}^{\varepsilon},\vartheta^{\varepsilon},{r}^{\varepsilon})
\varphi
\;
{\rm d}x {\rm d}t
+
\int_{\partial\Omega_T}
\beta_{e}
{p}_{\infty} \varphi
\,
{\rm d}{S} {\rm d}t
\end{multline}
for all test functions $\varphi \in L^2(0,T; {W}^{1,2}(\Omega))$,
further,
\begin{multline}\label{eq:defn_weak_04}
\int_{0}^T
\langle
\partial_t
\left[
c_{w} {b}({x/\varepsilon},{p}^{\varepsilon},{r}^{\varepsilon}) \vartheta^{\varepsilon}
+
\sigma({x/\varepsilon},{r}^{\varepsilon}) \vartheta^{\varepsilon}
\right],
\psi
\rangle
{\rm d}t
+
\int_{\partial\Omega_T}
\alpha_{e} \vartheta^{\varepsilon} \psi
\,
{\rm d}{S} {\rm d}t
\\
+
c_{w}
\int_{\partial\Omega_T}
\beta_{e} \vartheta^{\varepsilon} (
{p}^{\varepsilon}
-
{p}_{\infty}
) \psi
\,
{\rm d}{S} {\rm d}t
\\
+
\int_{\Omega_T}
{\lambda}({x/\varepsilon},{p}^{\varepsilon},\vartheta^{\varepsilon},{r}^{\varepsilon})
\nabla \vartheta^{\varepsilon}
\cdot
\nabla\psi
\,
{\rm d}x {\rm d}t
+
\int_{\Omega_T}
c_{w} \vartheta^{\varepsilon}
{a}({x/\varepsilon},{p}^{\varepsilon},\vartheta^{\varepsilon},{r}^{\varepsilon})
\nabla {p}^{\varepsilon}
\cdot
\nabla\psi
\,
{\rm d}x {\rm d}t
\\
=
\int_{\Omega_T}
\alpha_2    \chi_{c}({x/\varepsilon}) f({p}^{\varepsilon},\vartheta^{\varepsilon},{r}^{\varepsilon})
\psi
\,
{\rm d}x {\rm d}t
+
\int_{\partial\Omega_T}
\alpha_{e} \vartheta_{\infty}  \psi
\,
{\rm d}{S} {\rm d}t
\end{multline}
for all test functions
$\psi \in L^2(0,T;{W}^{1,2}(\Omega))$
and, finally,
\begin{equation}\label{eq:memory_weak_form}
{r}^{\varepsilon}(x,t)
=
\int_0^t
f({p}^{\varepsilon}(x,s),\vartheta^{\varepsilon}(x,s),{r}^{\varepsilon}(x,s))
\,
{\rm d}s
\quad
\textmd{   for all } t \in [0,T].
\end{equation}

\end{itemize}

\end{definition}

%----------------------------------------------------------------------
\begin{theorem}\label{thm:existence_theorem_epsilon}
Let the assumptions {\rm (i)--(v)} be satisfied. Then there exists at
least one weak solution of the
system \eqref{eq1a}--\eqref{eq1g}. Moreover,
\begin{equation}\label{max_principle}
\ess \sup_{\Omega_T} |\vartheta^{\varepsilon}| \leq   C,
\end{equation}
where the constant  $C$ is independent of $\varepsilon$
and
\begin{equation}\label{bound_p}
0 \geq {p}^{\varepsilon} \geq {p}_{\infty}
\qquad
\textmd{almost everywhere in  } \Omega_{T} \textmd{ and }  \partial\Omega_{T}.
\end{equation}
\end{theorem}

\begin{proof}
The proof is similar to the proof of \cite[Theorem~3.2]{BenesPazanin2018}.
In particular,
in \cite{BenesPazanin2018}, we considered the problem with
homogeneous Dirichlet-Neumann
boundary conditions. Similar arguments apply to the case of
Newton boundary conditions.
For the convenience of the reader, in Appendix~\ref{proof_1},
we present the basic steps and emphasize the differences.
\hfill $\square$
\end{proof}

The main result of the paper is summarized in the following theorem.
\begin{theorem}[Main result]
\label{thm:main_result}
Let $\varepsilon>0$.
Suppose that the data satisfy the assumptions {\rm (i)--(v)}
and
$[{p}^{\varepsilon},\vartheta^{\varepsilon},{r}^{\varepsilon}]$
is the weak solution of \eqref{eq1a}--\eqref{eq1g}.
Then there exist the pairs $[{p},{p}_{1}]$ and $[\vartheta,{\vartheta}_{1}]$
and the function $r$, such that
\begin{align*}
&
{p} \in   L^2(0,T;{{W}^{1,2}(\Omega)}) \cap L^{\infty}(\Omega_T),
\quad
{p}_{1} \in L^2(\Omega_T;W^{1,2}_{per}({\mathcal{Y}})),
\\
&
\vartheta \in   L^2( 0,T; W^{1,2}(\Omega)) \cap L^{\infty}(\Omega_T),
\quad
{\vartheta_1} \in L^2(\Omega_T;W^{1,2}_{per}({\mathcal{Y}})),
\\
&
{r} \in   L^{\infty}(\Omega_T)
\end{align*}
and  a sequence
$[{p}^{\varepsilon_j},\vartheta^{\varepsilon_j},{r}^{\varepsilon_j}]$
such that
$\lim_{j\rightarrow+\infty}\varepsilon_j = 0^+$
and
%such that we have %(as $\varepsilon \rightarrow 0$)
\begin{eqnarray}
{p}^{\varepsilon_j}  \rightharpoonup  {p}
\;\textmd{ and }
\;
\vartheta^{\varepsilon_j} \rightharpoonup  \vartheta
&&
\textrm{weakly in } L^2(0,T;{W^{1,2}(\Omega)}),
\label{conv_21}
\\
{p}^{\varepsilon_j}  \rightharpoonup  {p}
\;\textmd{ and }
\;
\vartheta^{\varepsilon_j} \rightharpoonup  \vartheta
&&
\textrm{weakly star in } L^{\infty}(\Omega_T),
\label{conv_22}
\\
{p}^{\varepsilon_j}  \rightarrow {p}
\;\textmd{ and }\;
\vartheta^{\varepsilon_j} \rightarrow \vartheta
&&
\textrm{almost everywhere in } {\Omega_T},
\label{conv_23}
\\
{p}^{\varepsilon_j}  \rightarrow {p}
\;\textmd{ and }\;
\vartheta^{\varepsilon_j} \rightarrow \vartheta
&&
\textrm{almost everywhere in } {\partial\Omega_{T} },
\label{conv_23bb}
\\
\nabla {p}^{\varepsilon_j}  \rightarrow  \nabla_x {p} + \nabla_y {p}_{1}
\;\textmd{ and }\;
\nabla \vartheta^{\varepsilon_j} \rightarrow  \nabla_x \vartheta + \nabla_y {\vartheta_1}
&&
\textrm{in the two-scale sense},
\label{conv_24}
\\
{r}^{\varepsilon_j}  \rightharpoonup  {r}
&&
\textrm{weakly star in } L^{\infty}(\Omega_T),
\label{conv_25}
\\
{r}^{\varepsilon_j}  \rightarrow {r}
&&
\textrm{almost everywhere in } {\Omega_T}.
\label{conv_26}
\end{eqnarray}
Further, ${p}$, $\vartheta$ and ${r}$ satisfy (in a weak sense)
the following coupled homogenized problem in ${\Omega_T}$
\begin{equation}\label{eq:hom_01}
\partial_t {b}^{*}({p},{r})
-
\nabla\cdot
\left[
\bfA^*({p},\vartheta,{r})\nabla {p}
\right]
=
\alpha_1  \chi_{c}^*   {f}({p},\vartheta,{r})
\end{equation}
and
\begin{equation}\label{eq:hom_02}
\partial_t
\left[
\left(
c_{w}   {b}^{*}({p},{r}) + \sigma^*(r)
\right)
\vartheta
\right]
-
\nabla \cdot
\left[
\bfLambda^*({p},\vartheta,{r}) \nabla \vartheta
+
c_{w}\vartheta \bfA^*({p},\vartheta,{r}) \nabla {p}
\right]
=
\alpha_2   \chi_{c}^*   {f}({p},\vartheta,{r}),
\end{equation}
with the integral condition
\begin{equation}
{r}(x,t)
=
\int_0^t
{f}({p}(x,s),\vartheta(x,s),{r}(x,s))
{\rm d}s
\end{equation}
and with the boundary and initial conditions
\begin{align}
- \bfA^*({p},\vartheta,{r})\nabla {p} \cdot \bfn
&
= \beta_{e}
(
{p}
-
{p}_{\infty}
)
\quad
&&{\rm in} \;   \partial\Omega_{T},
\label{eq:hom_03}
\\
-
\bfLambda^*({p},\vartheta,{r}) \nabla \vartheta
\cdot \bfn
&
= \alpha_{e} (\vartheta - \vartheta_{\infty})
\quad
&&{\rm in} \;  \partial\Omega_{T},
\label{eq:hom_04}
\\
{p}(x,0)
= {p}_0(x),
\quad
\vartheta(x,0)
& = \vartheta_0(x)
\quad
&&{\rm in} \;  \Omega.
\label{eq:hom_05}
\end{align}
Here,
the homogenized coefficient functions are defined as
\begin{align}
&
\chi_{c}^*
=
\int_\mathcal{Y}
\chi_{c}({y})
{\rm d}y,
\\
&
{b}^{*}(\xi,\zeta)
=
\varrho_{w}
\int_\mathcal{Y}
\left[
\chi_{c}({y})
\phi_{c}(\zeta)
\mathcal{S}(\xi)
+
\chi_{a}({y})
\phi_{a}
\mathcal{S}(\xi)
\right]
{\rm d}y,
\\
&
\bfA^*(\xi,\eta,\zeta)
=
\int_{\mathcal{Y}}
{a}({y},\xi,\eta,\zeta)
\left( {\bfI_d} + \nabla_y \bfw \right)
{\rm d}y
,
\\
&
\sigma^*(\zeta)
=
\int_\mathcal{Y}
\left[
\chi_{c}({y})
\varrho_{sc} {c}_{sc}(1 - \phi_{c}(\zeta) )
+
\chi_{a}({y})
\varrho_{sa} {c}_{sa}(1 - \phi_{a} )
\right]
{\rm d}y,
\\
&
\bfLambda^*(\xi,\eta,\zeta)
=
\int_{\mathcal{Y}}
{\lambda}({y},\xi,\eta,\zeta)
\left(
{\bfI_d} + \nabla_y \bfv
\right)
{\rm d}y,
\end{align}
where
${\bfI_d}$ denotes the identity matrix of size ${2}\times {2}$.
Further, $\bfw = (w_1,w_2)$,
$\nabla_y\bfw$ is the matrix
$(\nabla_y\bfw)_{ij} = \partial {w}_{j} / \partial {y}_{i}$,
$i,j=1,2$.
Similarly,
$\bfv = ({v}_{1},{v}_{2})$ and
$\nabla_y \bfv$ is the matrix
$(\nabla_y\bfv)_{ij} = \partial {v}_{j} / \partial {y}_{i}$,
$i,j=1,2$.
Further,
${w}_{i}$ and ${v}_{i} \in {W^{1,2}_{per}(\mathcal{Y})}$
are periodic solutions of the following cell problems, respectively,
\begin{align}
&
-
\nabla_y \cdot
\left(
{a}({y},\xi,\eta,\zeta) (\bfe_i  + \nabla_y w_{i})
\right)
=
0
\quad \textrm{ in } \mathcal{Y}
\quad  \textmd{ and }
\quad
\int_\mathcal{Y}
{w}_{i}
\,{\rm d}y
=
0,
\qquad
i=1,2,
\label{eq:cell_problem_w}
\\
&
-\nabla_y \cdot
\left(
\lambda({y},\xi,\eta,\zeta) (\bfe_i  + \nabla_y {v}_{i})
\right)
=
0
\quad \textrm{  in } \mathcal{Y}
\quad  \textmd{ and }
\quad
\int_\mathcal{Y} {v}_{i}
\,{\rm d}y
=
0,
\qquad
i=1,2.
\label{eq:cell_problem_v}
\end{align}
\end{theorem}

%------------------------------------------------------------------------
\paragraph{Schedule of the proof of the main result.}
In the proof of the main result we use the two-scale convergence theory.
The existence of the weak solution to \eqref{eq1a}--\eqref{eq1g}
may be proved in a similar same way as \cite[Theorem~3.2]{BenesPazanin2018},
see Appendix~\ref{proof_1}.
In the next section we derive necessary a priori estimates
with respect to the scale parameter $\varepsilon$
to show that there exists a sequence
$[{p}^{\varepsilon_j},{\vartheta}^{\varepsilon_j},{r}^{\varepsilon_j}]$,
whose limit is the solution
of the upscaled problem (in a weak sense).
The derivation of the
estimates is a nontrivial procedure due to degeneracies in the transport coefficients,
see (i) and (ii).
The crucial step is to obtain the $L^{\infty}$-estimates
$\| {p}^{\varepsilon} \|_{L^{\infty}(\Omega_{T})} \leq C$.
$\| {\vartheta}^{\varepsilon} \|_{L^{\infty}(\Omega_{T})} \leq C$
and
$\| {r}^{\varepsilon} \|_{L^{\infty}(\Omega_{T})} \leq C$.
Then, we get the usual elliptic a priori estimates
for the solution  $[{p}^{\varepsilon},\vartheta^{\varepsilon},{r}^{\varepsilon}]$ of the
two-scale problem \eqref{eq1a}--\eqref{eq1g}.
In particular, we shall prove the estimates
\begin{align*}%\label{est_u_01}
&
\| {p}^{\varepsilon} \|_{L^2(0,T;{W^{1,2}})} \leq C,
\qquad
\int_0^{T-\tau}
\left(
\mathcal{S}\left({p}^{\varepsilon}(t+\tau)\right)
-
\mathcal{S}\left({p}^{\varepsilon}(t)\right),{p}^{\varepsilon}(t+\tau) - {p}^{\varepsilon}(t)
\right)
\,{\rm d}t
\leq
C \tau,
\\
&
\| {\vartheta}^{\varepsilon} \|_{L^2(0,T;{W^{1,2}})} \leq C,
\qquad
\int_0^{T-\tau}
|{\vartheta}^{\varepsilon}(t+\tau) - {\vartheta}^{\varepsilon}(t)|^{2}
\,{\rm d}t
\leq
C \tau,
\\
&
\| {r}^{\varepsilon} \|_{L^2(0,T;{W^{1,2}})} \leq C,
\qquad
\int_0^{T-\tau}
| {r}^{\varepsilon}(t+\tau) - {r}^{\varepsilon}(t) |^{2}
\,{\rm d}t
\leq
C \tau.
\end{align*}
Based on these estimates
we can pass to the
limit as the space period $\varepsilon$ vanishes to get \eqref{conv_21}--\eqref{conv_26}
(through a selected subsequence).
Next, using the two-scale convergence theory \cite{Allaire1992,cioranescu,Nguetseng}
we may identify the limits as the solution of
the macro-scale formulation \eqref{eq:hom_01}--\eqref{eq:hom_05} (in a weak sense).
%The rest of the proof runs analogously as in \cite[Sections~5 and 6]{benespazaninJMAA2017}.

%%%%%%%%%%%%%%%%%%%%%%%%%%%%%%%%%%%%%%%%%%%%%%%%%%%%%%%%%%%%%%%%%%%%%
\section{Proof of the main result}
\label{sec:proof_main}
%%%%%%%%%%%%%%%%%%%%%%%%%%%%%%%%%%%%%%%%%%%%%%%%%%%%%%%%%%%%%%%%%%%%%
\subsection{A-priori estimates.}
\label{subsec:a_priori_estimates}
Here we present some a-priori bounds for the solution of
\eqref{eq:defn_weak_01}--\eqref{eq:memory_weak_form}.
First, we derive uniform estimates for the pressure ${p}^{\varepsilon}$.
Define the function $\Theta: \mathbb{R} \rightarrow \mathbb{R}$ given by the equation
\begin{equation}\label{est:press_02}
\Theta(\xi) = \int_{0}^{\xi}   \mathcal{S}'(z)z  {{\rm d}z},
\qquad
\xi \in \mathbb{R}.
\end{equation}
It is easy to check that
\begin{equation}\label{est:press_03}
\Theta(\xi_1) - \Theta(\xi_2) \leq [\mathcal{S}(\xi_1) - \mathcal{S}(\xi_2)]\xi_1
\qquad
\forall \xi_1,\xi_2 \in \mathbb{R}.
\end{equation}
Now, let $\tau>0$. Note that from \eqref{eq:memory_weak_form} and using \eqref{lipschitz_phi} we deduce
\begin{equation}\label{ineq:porosity_lipschitz}
|\phi_{c}({r}^{\varepsilon}(x,s))
-
\phi_{c}({r}^{\varepsilon}(x,s-\tau))|
\leq
{C}_{\phi}
\int_{s-\tau}^{s}
{f}({p}^{\varepsilon}(x,\xi),\vartheta^{\varepsilon}(x,\xi),{r}^{\varepsilon}(x,\xi))
\,
{\rm d}\xi.
\end{equation}
Define
${p}^{\varepsilon}(t)={p}_{0}$
and
${r}^{\varepsilon}(t) = 0$ for $-\tau < t < 0$. Letting
$\tau \rightarrow 0$ in the estimate
\begin{align*}%\label{}
&
\frac{1}{\tau}
\int_{0}^{t}
\langle
b\left({x/\varepsilon},{p}^{\varepsilon}(s),{r}^{\varepsilon}(s)\right)
-
b\left({x/\varepsilon},{p}^{\varepsilon}(s-\tau),{r}^{\varepsilon}(s-\tau)\right),
{p}^{\varepsilon}(s)
\rangle
\,{\rm d}s
\nonumber
\\
=
\;
&
\frac{1}{\tau}
\int_{0}^{t}\int_{\Omega}
\left[
b\left({x/\varepsilon},{p}^{\varepsilon}(s),{r}^{\varepsilon}(s)\right)
-
b\left({x/\varepsilon},{p}^{\varepsilon}(s-\tau),{r}^{\varepsilon}(s-\tau)\right)
\right]
{p}^{\varepsilon}(s)
\,{\rm d}x{\rm d}s
\nonumber
\\
=
\;
&
\frac{\varrho_{w}}{\tau}
\int_{0}^{t}\int_{\Omega}
\left[
\chi_{c}({x/\varepsilon})
\phi_{c}({r}^{\varepsilon}(s))
\mathcal{S}({p}^{\varepsilon}(s))
-
\chi_{c}({x/\varepsilon})
\phi_{c}({r}^{\varepsilon}(s-\tau))
\mathcal{S}({p}^{\varepsilon}(s-\tau))
\right]
{p}^{\varepsilon}(s)
\,{\rm d}x{\rm d}s
\nonumber
\\
&
+
\frac{\varrho_{w}}{\tau}
\int_{0}^{t}\int_{\Omega}
\left[
\chi_{a}({x/\varepsilon})
\phi_{a}
\mathcal{S}({p}^{\varepsilon}(s))
-
\chi_{a}({x/\varepsilon})
\phi_{a}
\mathcal{S}({p}^{\varepsilon}(s-\tau))
\right]
{p}^{\varepsilon}(s)
\,{\rm d}x{\rm d}s
\nonumber
\\
=
\;
&
\frac{\varrho_{w}}{\tau}
\int_{0}^{t}\int_{\Omega}
\chi_{c}({x/\varepsilon})
\left(
\phi_{c}({r}^{\varepsilon}(s))
-
\phi_{c}({r}^{\varepsilon}(s-\tau))
\right)
\mathcal{S}({p}^{\varepsilon}(s))
{p}^{\varepsilon}(s)
\,{\rm d}x{\rm d}s
\nonumber
\\
&
+
\frac{\varrho_{w}}{\tau}
\int_{0}^{t}
\int_{\Omega}
\chi_{c}({x/\varepsilon})
\phi_{c}({r}^{\varepsilon}(s-\tau))
\left(
\mathcal{S}({p}^{\varepsilon}(s))
-
\mathcal{S}({p}^{\varepsilon}(s-\tau))
\right)
{p}^{\varepsilon}(s)
\,{\rm d}x{\rm d}s
\nonumber
\\
&
+
\frac{\varrho_{w}}{\tau}
\int_{0}^{t}\int_{\Omega}
\chi_{a}({x/\varepsilon})
\phi_{a}
\left(
\mathcal{S}({p}^{\varepsilon}(s))
-
\mathcal{S}({p}^{\varepsilon}(s-\tau))
\right)
{p}^{\varepsilon}(s)
\,{\rm d}x{\rm d}s
\nonumber
\\
\geq
\;
&
\frac{\varrho_{w}}{\tau}
\int_{0}^{t}\int_{\Omega}
\chi_{c}({x/\varepsilon})
\left(
\phi_{c}({r}^{\varepsilon}(s))
-
\phi_{c}({r}^{\varepsilon}(s-\tau))
\right)
\mathcal{S}({p}^{\varepsilon}(s))
{p}^{\varepsilon}(s)
\,{\rm d}x{\rm d}s
\nonumber
\\
&
+
\frac{\varrho_{w}}{\tau}
\int_{0}^{t}
\int_{\Omega}
\chi_{c}({x/\varepsilon})
\left[
\phi_{c}({r}^{\varepsilon}(s))
{\Theta}({p}^{\varepsilon}(s))
-
\phi_{c}({r}^{\varepsilon}(s-\tau))
{\Theta}({p}^{\varepsilon}(s-\tau))
\right]
\,{\rm d}x{\rm d}s
\nonumber
\\
&
-
\frac{\varrho_{w}}{\tau}
\int_{0}^{t}
\int_{\Omega}
\chi_{c}({x/\varepsilon})
\left(
\phi_{c}({r}^{\varepsilon}(s))
-
\phi_{c}({r}^{\varepsilon}(s-\tau))
\right)
{\Theta}({p}^{\varepsilon}(s))
\,{\rm d}x{\rm d}s
\nonumber
\\
&
+
\frac{\varrho_{w}}{\tau}
\int_{0}^{t}\int_{\Omega}
\chi_{a}({x/\varepsilon})
\phi_{a}
\left(
{\Theta}({p}^{\varepsilon}(s))
-
{\Theta}({p}^{\varepsilon}(s-\tau))
\right)
\,{\rm d}x{\rm d}s
\nonumber
\\
\geq
\;
&
\frac{\varrho_{w}}{\tau}
\int_{0}^{t}\int_{\Omega}
\chi_{c}({x/\varepsilon})
\left(
\phi_{c}({r}^{\varepsilon}(s))
-
\phi_{c}({r}^{\varepsilon}(s-\tau))
\right)
\mathcal{S}({p}^{\varepsilon}(s))
{p}^{\varepsilon}(s)
\,{\rm d}x{\rm d}s
\nonumber
\\
&
-
\frac{\varrho_{w}}{\tau}
\int_{0}^{t}
\int_{\Omega}
\chi_{c}({x/\varepsilon})
\left(
\phi_{c}({r}^{\varepsilon}(s))
-
\phi_{c}({r}^{\varepsilon}(s-\tau))
\right)
{\Theta}({p}^{\varepsilon}(s))
\,{\rm d}x{\rm d}s
\nonumber
\\
&
+
\frac{\varrho_{w}}{\tau}
\int_{t-\tau}^{t}
\int_{\Omega}
\chi_{c}({x/\varepsilon})
\phi_{c}({r}^{\varepsilon}(s))
{\Theta}({p}^{\varepsilon}(s))
\,{\rm d}x{\rm d}s
-
\frac{\varrho_{w}}{\tau}
\int_{-\tau}^{0}
\int_{\Omega}
\chi_{c}({x/\varepsilon})
\phi_{c}({r}^{\varepsilon}(s))
{\Theta}({p}^{\varepsilon}(s))
\,{\rm d}x{\rm d}s
\nonumber
\\
&
+
\frac{\varrho_{w}}{\tau}
\int_{t-\tau}^{t}
\int_{\Omega}
\chi_{a}({x/\varepsilon})
\phi_{a}
{\Theta}({p}^{\varepsilon}(s))
\,{\rm d}x{\rm d}s
-
\frac{\varrho_{w}}{\tau}
\int_{-\tau}^{0}
\int_{\Omega}
\chi_{a}({x/\varepsilon})
\phi_{a}
{\Theta}({p}^{\varepsilon}(s))
\,{\rm d}x{\rm d}s,
\end{align*}
we obtain
\begin{align}\label{eq:120}
&
\varrho_{w}
\int_{\Omega}
\left(
\chi_{c}({x/\varepsilon})
\phi_{c}({r}^{\varepsilon}(t))
{\Theta}({p}^{\varepsilon}(t))
+
\chi_{a}({x/\varepsilon})
\phi_{a}
{\Theta}({p}^{\varepsilon}(t))
\right)
{\rm d}{x}
\nonumber
\\
&
-
\varrho_{w}
\int_{\Omega}
\left(
\chi_{c}({x/\varepsilon})
\phi_{c}({0})
{\Theta}({p}^{\varepsilon}_{0})
+
\chi_{a}({x/\varepsilon})
\phi_{a}
{\Theta}({p}^{\varepsilon}_{0})
\right)
{\rm d}{x}
\nonumber
\\
\leq
&
\;
\int_0^T
\langle
\partial_t {b}({x/\varepsilon},{p}^{\varepsilon},{r}^{\varepsilon})  , {p}^{\varepsilon}
\rangle
{\rm d}t
+
C_{\phi}\varrho_{w}
\int_{\Omega_{t}}
\chi_{c}({x/\varepsilon})
f({p}^{\varepsilon},\vartheta^{\varepsilon},{r}^{\varepsilon})
|{\Theta}({p}^{\varepsilon}) - \mathcal{S}({p}^{\varepsilon}) {p}^{\varepsilon} |
\;
{\rm d}x {\rm d}s.
\end{align}
Setting $\chi_{(0,t)}{p}^{\varepsilon}$ as a test function in
\eqref{eq:defn_weak_03}
($\chi_{(0,t)}$ is the characteristic function of $(0,t)$)
and using \eqref{eq:120}
we arrive at
\begin{align}\label{est:120}
&
\varrho_{w}
\int_{\Omega}
\left[
\chi_{c}({x/\varepsilon})
\phi_{c}({r}^{\varepsilon}(t))
+
\chi_{a}({x/\varepsilon})
\phi_{a}
\right]
{\Theta}({p}^{\varepsilon}(t))
{\rm d}{x}
\nonumber
\\
&
+
\int_{\Omega_t}
{a}({x/\varepsilon},{p}^{\varepsilon},\vartheta^{\varepsilon},{r}^{\varepsilon})
|\nabla {p}^{\varepsilon}|^2
{\rm d}x{\rm d}s
%\nonumber
%\\
%&
+
\frac{\beta_{e}}{2}
\int_{\partial\Omega_{t}}
| {p}^{\varepsilon} |^{2}
\,
{\rm d}{S} {\rm d}t
\nonumber
\\
\leq
\;
&
\varrho_{w}
\int_{\Omega}
\left[
\chi_{c}({x/\varepsilon})
\phi_{c}({0})
+
\chi_{a}({x/\varepsilon})
\phi_{a}
\right]
{\Theta}({p}^{\varepsilon}_{0})
{\rm d}{x}
+
\frac{\beta_{e}}{2}
\int_{\partial\Omega_{t}}
|{p}_{\infty}|^{2}
\,
{\rm d}{S} {\rm d}t
\nonumber
\\
&
+
C_{\phi}
\varrho_{w}
\int_{\Omega_{t}}
\chi_{c}({x/\varepsilon})
f({p}^{\varepsilon},\vartheta^{\varepsilon},{r}^{\varepsilon})
|{\Theta}({p}^{\varepsilon}) - \mathcal{S}({p}^{\varepsilon}) {p}^{\varepsilon} |
\;
{\rm d}x {\rm d}s
\nonumber
\\
&
+
|\alpha_1|
\int_{\Omega_{t}}
\chi_{c}({x/\varepsilon})
f({p}^{\varepsilon},\vartheta^{\varepsilon},{r}^{\varepsilon})
|{p}^{\varepsilon}|
\;
{\rm d}x {\rm d}s
\end{align}
for almost all $t \in (0,T)$.
Using \eqref{con11a}, \eqref{bound_phi}, \eqref{assum:bound_f} and \eqref{bound_p}
and applying the Young's inequality
to last two terms on the right-hand side in
\eqref{est:120}
we get
\begin{equation}\label{}
\int_{\Omega}
{\Theta}({p}^{\varepsilon}(t))
{\rm d}{x}
+
\int_{0}^{t}
\|  {p}^{\varepsilon}\|^2_{W^{1,2}(\Omega)}
{\rm d}s
\\
\leq
{C}_{1}
+
{C}_{2}
\int_{0}^{t}
\int_{\Omega}
\chi_{c}
{\Theta}({p}^{\varepsilon})
\;
{\rm d}{x}{\rm d}s
\end{equation}
for almost all $t \in (0,T)$.
We now apply   the Gronwall's inequality to conclude that
\begin{equation}\label{est:energy_u}
\sup_{0 \leq t \leq T}
\int_{\Omega}
{\Theta}({p}^{\varepsilon}(t))
{\rm d}{x}
+
\int_0^T
\|{p}^{\varepsilon}\|^2_{W^{1,2}(\Omega)}
{\rm d}t
\leq
C.
\end{equation}
Next we proceed analogously as in \cite{FiloKacur1995}.
Using $\chi_{(t,t+\tau)}w$, $w\in {{W}^{1,2}(\Omega)}$, as a test function in
\eqref{eq:defn_weak_03} we can write
\begin{multline}\label{eq:101}
\langle
{b}({x/\varepsilon},{p}^{\varepsilon}(t+\tau),{r}^{\varepsilon}(t+\tau))
-
{b}({x/\varepsilon},{p}^{\varepsilon}(t),{r}^{\varepsilon}(t)),
w
\rangle
\\
+
\int_t^{t+\tau}
\int_{\Omega}
{a}({x/\varepsilon},{p}^{\varepsilon},\vartheta^{\varepsilon},{r}^{\varepsilon})
\nabla {p}^{\varepsilon}    \cdot   \nabla w
\,{\rm d}x{\rm d}s
+
\int_{t}^{t+\tau}
\int_{\partial\Omega}
\beta_{e}
{p}^{\varepsilon}  {w}
\,
{\rm d}{S} {\rm d}s
\\
=
\alpha_1
\int_t^{t+\tau}
\int_{\Omega}
\chi_{c}({x/\varepsilon})
f({p}^{\varepsilon},\vartheta^{\varepsilon},{r}^{\varepsilon})
w
\;
{\rm d}x {\rm d}s
+
\int_{t}^{t+\tau}
\int_{\partial\Omega}
\beta_{e}
 {p}_{\infty}  {w}
\,
{\rm d}{S} {\rm d}s.
\end{multline}
Now we set
$w = {p}^{\varepsilon}(t+\tau) - {p}^{\varepsilon}(t)$ and integrate \eqref{eq:101} with respect to $t$
over $(0,T-\tau)$ to obtain
\begin{multline*}
\int_0^{T-\tau}
\langle
{b}({x/\varepsilon},{p}^{\varepsilon}(t+\tau),{r}^{\varepsilon}(t+\tau))
-
{b}({x/\varepsilon},{p}^{\varepsilon}(t),{r}^{\varepsilon}(t)),
{p}^{\varepsilon}(t+\tau) - {p}^{\varepsilon}(t)
\rangle
{\rm d}t
\\
\leq
C \tau
\int_0^{T}
\left(
\| {p}^{\varepsilon} \|^2_{W^{1,2}(\Omega)}
+
1
\right)
{\rm d}t
\end{multline*}
and
\begin{align*}
&
\int_0^{T-\tau}
\int_{\Omega}
[
{b}({x/\varepsilon},{p}^{\varepsilon}(t+\tau),{r}^{\varepsilon}(t+\tau))
-
{b}({x/\varepsilon},{p}^{\varepsilon}(t),{r}^{\varepsilon}(t+\tau))
]
({p}^{\varepsilon}(t+\tau) - {p}^{\varepsilon}(t))
{\rm d}x{\rm d}t
\\
\leq
\;
&
\int_0^{T-\tau}
\int_{\Omega}
[
{b}({x/\varepsilon},{p}^{\varepsilon}(t),{r}^{\varepsilon}(t))
-
{b}({x/\varepsilon},{p}^{\varepsilon}(t),{r}^{\varepsilon}(t+\tau))
]
({p}^{\varepsilon}(t+\tau) - {p}^{\varepsilon}(t))
{\rm d}x{\rm d}t
\\
&
+
C \tau
\int_0^{T}
\left(
\| {p}^{\varepsilon} \|^2_{W^{1,2}(\Omega)}
+
1
\right)
{\rm d}t.
\end{align*}
In view of \eqref{notation_form_b}, we can write
\begin{align*}
&
\varrho_{w}
\int_0^{T-\tau}
\int_{\Omega}
\chi_{c}({x/\varepsilon})
\phi_{c}({r}^{\varepsilon}(t+\tau))
[\mathcal{S}({p}^{\varepsilon}(t+\tau))
-
\mathcal{S}({p}^{\varepsilon}(t))
]
({p}^{\varepsilon}(t+\tau) - {p}^{\varepsilon}(t))
{\rm d}x{\rm d}t
\\
&
+
\varrho_{w}
\int_0^{T-\tau}
\int_{\Omega}
\chi_{a}({x/\varepsilon})
\phi_{a}
[\mathcal{S}({p}^{\varepsilon}(t+\tau))
-
\mathcal{S}({p}^{\varepsilon}(t))
]
({p}^{\varepsilon}(t+\tau) - {p}^{\varepsilon}(t))
{\rm d}x{\rm d}t
\\
\leq
\;
&
\varrho_{w}
\int_0^{T-\tau}
\int_{\Omega}
\chi_{c}({x/\varepsilon})
[
\phi_{c}({r}^{\varepsilon}(t))
-
\phi_{c}({r}^{\varepsilon}(t+\tau))
]
\mathcal{S}({p}^{\varepsilon}(t))
({p}^{\varepsilon}(t+\tau) - {p}^{\varepsilon}(t))
{\rm d}x{\rm d}t
\\
&
+
C \tau
\int_0^{T}
\left(
\| {p}^{\varepsilon} \|^2_{W^{1,2}(\Omega)}
+
1
\right)
{\rm d}t.
\end{align*}
Now, owing to \eqref{bound_phi}, \eqref{assum:bound_f}, \eqref{ineq:porosity_lipschitz} and \eqref{est:energy_u}
we deduce
\begin{equation}\label{est:compact_u_00}
\int_0^{T-\tau}
\int_{ {\Omega} }
[\mathcal{S}({p}^{\varepsilon}(t+\tau))
-
\mathcal{S}({p}^{\varepsilon}(t))
]
({p}^{\varepsilon}(t+\tau) - {p}^{\varepsilon}(t))
{\rm d}x{\rm d}t
\leq
C \tau.
\end{equation}
Note that $C$ is independent of $\varepsilon$.
From \eqref{est:energy_u} it follows that
\begin{equation}\label{conv:u000}
{p}^{\varepsilon}  \rightharpoonup  {p}
\qquad
\textrm{weakly in } L^2(0,T;W^{1,2}(\Omega))
\end{equation}
(through a subsequence which we shall
again denote by ``$\left\{\varepsilon\right\}$'').
Using the
compactness arguments of \cite[Lemma 1.9]{AltLuckhaus1983}
(see also \cite[Eqs. (2.10)--(2.12)]{FiloKacur1995})
and the estimates \eqref{est:energy_u} and \eqref{est:compact_u_00}
we get
\begin{align}
&
\mathcal{S}({p}^{\varepsilon})
\rightarrow
\mathcal{S}({p})
\textmd{ in }L^1({\Omega}_{T})
\;
\textmd{  and   almost everywhere on } \, {\Omega}_{T}.
\label{eq555a}
\end{align}
Since $\mathcal{S}$
is a strictly monotone function (recall \eqref{con11a}),
it follows from \eqref{eq555a}
that, see \cite[Proposition 3.35]{Kacur1990a},
\begin{equation}
\label{conv:u00}
{p}^{\varepsilon}  \rightarrow  {p}
\qquad  \textrm{ almost everywhere on } {\Omega_T}.
%\quad
%
%\quad
%
\end{equation}

\bigskip

We next derive uniform estimates for ${\vartheta}^{\varepsilon}$.
First, let $\tau>0$ and define
${p}^{\varepsilon}(t)={p}_{0}$
and
${{\vartheta}}^{\varepsilon}(t)={\vartheta}_0$
for $-\tau < t < 0$.
Letting
$\tau \rightarrow 0$ in the estimate
\begin{align*}%\label{}
&
-
\frac{c_{w}}{\tau}
\int_{0}^{t}
\int_{\Omega}
\left[
b\left({x/\varepsilon},{p}^{\varepsilon}(s),{r}^{\varepsilon}(s)\right)
-
b\left({x/\varepsilon},{p}^{\varepsilon}(s-\tau),{r}^{\varepsilon}(s-\tau)\right)
\right]
{\vartheta}^{\varepsilon}(s)^2
\,{\rm d}x{\rm d}s
\nonumber
\\
&
+
\frac{2 c_{w} }{\tau}
\int_{0}^{t}\int_{\Omega}
\left[
b\left({x/\varepsilon},{p}^{\varepsilon}(s),{r}^{\varepsilon}(s)\right)
{\vartheta}^{\varepsilon}(s)
-
b\left({x/\varepsilon},{p}^{\varepsilon}(s-\tau),{r}^{\varepsilon}(s-\tau)\right)
{\vartheta}^{\varepsilon}(s-\tau)
\right]
{\vartheta}^{\varepsilon}(s)
\,{\rm d}x{\rm d}s
\nonumber
\\
&
+
\frac{2}{\tau}
\int_{0}^{t}\int_{\Omega}
\left[
\sigma({x/\varepsilon},{r}^{\varepsilon}(s)) {\vartheta}^{\varepsilon}(s)
-
\sigma({x/\varepsilon},{r}^{\varepsilon}(s-\tau)) {\vartheta}^{\varepsilon}(s-\tau)
\right]
{\vartheta}^{\varepsilon}(s)
\,{\rm d}x{\rm d}s
\nonumber
\\
=\;
&
-
\frac{ c_{w} }{\tau}
\int_{0}^{t}
\langle
b\left({x/\varepsilon},{p}^{\varepsilon}(s),{r}^{\varepsilon}(s)\right)
-
b\left({x/\varepsilon},{p}^{\varepsilon}(s-\tau),{r}^{\varepsilon}(s-\tau)\right)
,
{\vartheta}^{\varepsilon}(s)^2
\rangle
\,{\rm d}s
\nonumber
\\
&
+
\frac{2 c_{w} }{\tau}
\int_{0}^{t}
\langle
b\left({x/\varepsilon},{p}^{\varepsilon}(s),{r}^{\varepsilon}(s)\right)    {\vartheta}^{\varepsilon}(s)
-
b\left({x/\varepsilon},{p}^{\varepsilon}(s-\tau),{r}^{\varepsilon}(s-\tau)\right)    {\vartheta}^{\varepsilon}(s-\tau),
{\vartheta}^{\varepsilon}(s)
\rangle
\,{\rm d}s
\nonumber
\\
&
+
\frac{2}{\tau}
\int_{0}^{t}
\langle
\sigma({x/\varepsilon},{r}^{\varepsilon}(s))   {\vartheta}^{\varepsilon}(s)
-
\sigma({x/\varepsilon},{r}^{\varepsilon}(s-\tau))   {\vartheta}^{\varepsilon}(s-\tau),
{\vartheta}^{\varepsilon}(s)
\rangle
\,{\rm d}s
\nonumber
\\
\geq
\;
&
\frac{1}{\tau}
\int_{t-\tau}^{t}
\int_{\Omega}
{\vartheta}^{\varepsilon}(s)^2
\left[
c_{w} b\left({x/\varepsilon},{p}^{\varepsilon}(s),{r}^{\varepsilon}(s)\right)
+
\sigma({x/\varepsilon},{r}^{\varepsilon}(s))
\right]
\, {\rm d}x  {\rm d}s
\nonumber
\\
&
-
\frac{1}{\tau}
\int_{-\tau}^{0}
\int_{\Omega}
{\vartheta}^{\varepsilon}(s)^2
\left[
c_{w} b\left({x/\varepsilon},{p}^{\varepsilon}(s),{r}^{\varepsilon}(s)\right)
+
\sigma({x/\varepsilon},{r}^{\varepsilon}(s))
\right]
\, {\rm d}x  {\rm d}s
\end{align*}
we obtain
\begin{align}\label{eq:int_by_parts_2}
&
\int_{\Omega}
{\vartheta}^{\varepsilon}(t)^2
\left[
c_{w} b\left({x/\varepsilon},{p}^{\varepsilon}(t),{r}^{\varepsilon}(t)\right)
+
\sigma({x/\varepsilon},{r}^{\varepsilon}(t))
\right]
{\rm d}x
-
\int_{\Omega}
{\vartheta}_0^2
\left[
c_{w} b\left({x/\varepsilon},{p}_{0},{0}\right)
+
\sigma({x/\varepsilon},{0})
\right]
{\rm d}{x}
\nonumber
\\
\leq
\;
&
-
\int_0^t
\langle
c_{w}
\partial_s
b\left( {x/\varepsilon},{p}^{\varepsilon}(s),{r}^{\varepsilon}(s) \right) ,
{\vartheta}^{\varepsilon}(s)^2
\rangle
{\rm d}s
\nonumber
\\
&
+
\int_{0}^t
\langle
\partial_s
\left(
\left[
c_{w}  b\left({x/\varepsilon},{p}^{\varepsilon}(s),{r}^{\varepsilon}(s)\right)
+
\sigma({x/\varepsilon},{r}^{\varepsilon}(s))
\right]
{\vartheta}^{\varepsilon}(s)
\right),
2{\vartheta}^{\varepsilon}(s)
\rangle
\,{\rm d}s
\end{align}
for almost all $t \in (0,T)$.
Now,
using $\varphi = c_{w}  [{\vartheta}^{\varepsilon}]^2$
as a test function in \eqref{eq:defn_weak_03}
and
$\psi = 2{\vartheta}^{\varepsilon}$ in \eqref{eq:defn_weak_04},
subtracting both equations
and using \eqref{eq:int_by_parts_2} we arrive at
\begin{align*}%\label{eq65ee}
&
\int_{\Omega}
{\vartheta}^{\varepsilon}(t)^2
\left[
c_{w}  b\left({x/\varepsilon},{p}^{\varepsilon}(t),{r}^{\varepsilon}(t)\right)
+
\sigma({x/\varepsilon},{r}^{\varepsilon}(t))
\right]
{\rm d}x
\\
&
+
\int_{\Omega_t}
2
{\lambda}({x/\varepsilon},{p}^{\varepsilon},\vartheta^{\varepsilon},{r}^{\varepsilon})
|\nabla{\vartheta}^{\varepsilon}|^2
{\rm d}x{\rm d}s
\\
&
+
2
\int_{\partial\Omega_{t}}
\alpha_{e} [\vartheta^{\varepsilon}]^{2}
\,
{\rm d}{S} {\rm d}t
+
c_{w}
\int_{\partial\Omega_{t}}
\beta_{e}
[\vartheta^{\varepsilon}]^{2}
(
{p}^{\varepsilon}
-
{p}_{\infty}
)
\,
{\rm d}{S} {\rm d}t
\\
\leq
&
\;
\int_{\Omega}
{\vartheta}_0^2
\left[
c_{w} b\left({x/\varepsilon},{p}_{0},{0}\right)
+
\sigma({x/\varepsilon},{0})
\right]
{\rm d}{x}
+
2
\int_{\partial\Omega_{t}}
\alpha_{e} \vartheta_{\infty}  \vartheta^{\varepsilon}
\,
{\rm d}{S} {\rm d}t
\\
&
+
\int_{\Omega_t}
\left(
2 \alpha_{2} {\vartheta}^{\varepsilon}
-
\alpha_{1} c_{w} [{\vartheta}^{\varepsilon}]^2
\right)
\chi_{c}({x/\varepsilon})
f({p}^{\varepsilon},\vartheta^{\varepsilon},{r}^{\varepsilon})
{\rm d}x{\rm d}s
\end{align*}
for almost all $t \in (0,T)$.
Hence, using \eqref{con12c}, \eqref{con12d} and \eqref{bound_p} we arrive at the inequality
\begin{align*}
%\label{eq65cee}
&
\int_{\Omega}
{\vartheta}^{\varepsilon}(t)^2
\left[
c_{w}   b\left({x/\varepsilon},{p}^{\varepsilon}(t),{r}^{\varepsilon}(t)\right)
+
\sigma({x/\varepsilon},{r}^{\varepsilon}(t))
\right]
{\rm d}x
\\
&
+
2 \lambda_{1} \int_{\Omega_t}   |\nabla{\vartheta}^{\varepsilon}|^2 {\rm d}x{\rm d}s
+
\int_{\partial\Omega_{t}}
\alpha_{e} [\vartheta^{\varepsilon} ]^{2}
\,
{\rm d}{S} {\rm d}t
\\
\leq
&
\;
\int_{\Omega}
{\vartheta}_0^2
\left[
c_{w}  b\left({x/\varepsilon},{p}_{0},{0}\right)
+
\sigma({x/\varepsilon},{0})
\right]
{\rm d}{x}
+
\int_{\partial\Omega_{t}}
\alpha_{e} [\vartheta_{\infty}]^{2}
\,
{\rm d}{S} {\rm d}t
\\
&
+
\int_{\Omega_t}
\left(
2 \alpha_{2} {\vartheta}^{\varepsilon}
-
\alpha_{1} c_{w} [{\vartheta}^{\varepsilon}]^2
\right)
\chi_{c}({x/\varepsilon})
f({p}^{\varepsilon},\vartheta^{\varepsilon},{r}^{\varepsilon})
{\rm d}x{\rm d}s.
\end{align*}
Therefore we conclude, applying the Gronwall's inequality, that
\begin{equation}\label{est:energy_theta}
\sup_{0 \leq t \leq T}
\|{\vartheta}^{\varepsilon}(t)\|^2_{L^2(\Omega)}
+
\int_0^T
\|{\vartheta}^{\varepsilon}\|^2_{W^{1,2}(\Omega)}
{\rm d}t
\leq
C,
\end{equation}
which immediately yields
\begin{eqnarray}
\| {\vartheta}^{\varepsilon} \|_{L^2(0,T;W^{1,2}(\Omega))} \leq C.
\label{est1_theta_a}
\end{eqnarray}
It follows that ${\vartheta}^{\varepsilon}$ converges weakly in $L^2(0,T;W^{1,2}(\Omega))$
to ${\vartheta}$ (along  a selected subsequence).

Further,
using $\chi_{(t,t+\tau)}w$, $w\in {{W}^{1,2}(\Omega)}$,
as a test function in
\eqref{eq:defn_weak_04} we obtain
\begin{align}\label{eq:comp_theta_02}
&
\langle
c_{w}
{b}\left(
{x/\varepsilon},{p}^{\varepsilon}(t+\tau),{r}^{\varepsilon}(t+\tau)
\right)
{\vartheta}^{\varepsilon}(t+\tau)
-
c_{w}  b\left({x/\varepsilon},{p}^{\varepsilon}(t),{r}^{\varepsilon}(t)\right)
{\vartheta}^{\varepsilon}(t),
w
\rangle
\nonumber
\\
&
+
\langle
\sigma({x/\varepsilon},{r}^{\varepsilon}(t+\tau))
{\vartheta}^{\varepsilon}(t+\tau)
-
\sigma({x/\varepsilon},{r}^{\varepsilon}(t))
{\vartheta}^{\varepsilon}(t),
w
\rangle
\nonumber
\\
&
+
\int_{t}^{t+\tau}
\int_{\Omega}
{\lambda}({x/\varepsilon},{p}^{\varepsilon},\vartheta^{\varepsilon},{r}^{\varepsilon})
\nabla {\vartheta}^{\varepsilon}
\cdot
\nabla w
\,{\rm d}x {\rm d}s
\nonumber
\\
&
+
\int_{t}^{t+\tau}
\int_{\Omega}
c_{w} \vartheta^{\varepsilon}
{a}({x/\varepsilon},{p}^{\varepsilon},\vartheta^{\varepsilon},{r}^{\varepsilon})
\nabla {p}^{\varepsilon}
\cdot \nabla   w
\,{\rm d}x {\rm d}s
\nonumber
\\
&
+
\int_{t}^{t+\tau}
\int_{\partial\Omega}
\alpha_{e} \vartheta^{\varepsilon}  w
\,
{\rm d}{S} {\rm d}s
+
c_{w}
\int_{t}^{t+\tau}
\int_{\partial\Omega}
\beta_{e} \vartheta^{\varepsilon} (
{p}^{\varepsilon}
-
{p}_{\infty}
)
{w}
\,
{\rm d}{S} {\rm d}s
\nonumber
\\
=
&
\;
\int_{t}^{t+\tau}
\int_{\Omega}
\alpha_2  \chi_{c}({x/\varepsilon})  f({p}^{\varepsilon},\vartheta^{\varepsilon},{r}^{\varepsilon})    {w}
\,{\rm d}x {\rm d}s
+
\int_{t}^{t+\tau}
\int_{\partial\Omega}
\alpha_{e} \vartheta_{\infty}  {w}
\,
{\rm d}{S} {\rm d}s
.
\end{align}

Now we set
$w = {\vartheta}^{\varepsilon}(t+\tau) - {\vartheta}^{\varepsilon}(t)$
and integrate \eqref{eq:comp_theta_02}  with respect to $t$
over $(0,T-\tau)$ to obtain
\begin{align*}%\label{}
&
\int_{0}^{T-\tau}
\left(
c_{w}
{b}\left(
{x/\varepsilon},{p}^{\varepsilon}(t+\tau),{r}^{\varepsilon}(t+\tau)
\right)
\left[
{\vartheta}^{\varepsilon}(t+\tau)
-
{\vartheta}^{\varepsilon}(t)
\right],
{\vartheta}^{\varepsilon}(t+\tau) - {\vartheta}^{\varepsilon}(t)
\right)
\,{\rm d}t
\nonumber
\\
&
+
\int_{0}^{T-\tau}
\left(
\left[
c_{w}
{b}\left(
{x/\varepsilon},{p}^{\varepsilon}(t+\tau),{r}^{\varepsilon}(t+\tau)
\right)
-
c_{w}  b\left({x/\varepsilon},{p}^{\varepsilon}(t),{r}^{\varepsilon}(t)\right)
\right]
{\vartheta}^{\varepsilon}(t),
{\vartheta}^{\varepsilon}(t+\tau) - {\vartheta}^{\varepsilon}(t)
\right)
\,{\rm d}t
\nonumber
\\
&
+
\int_{0}^{T-\tau}
\int_{\Omega}
\sigma({x/\varepsilon},{r}^{\varepsilon}(t+\tau))
|{\vartheta}^{\varepsilon}(t+\tau)
-
{\vartheta}^{\varepsilon}(t)|^2
\,{\rm d}x{\rm d}t
\nonumber
\\
&
+
\int_{0}^{T-\tau}
\int_{\Omega}
\left[
\sigma({x/\varepsilon},{r}^{\varepsilon}(t+\tau))
-
\sigma({x/\varepsilon},{r}^{\varepsilon}(t))
\right]
{\vartheta}^{\varepsilon}(t)
(
{\vartheta}^{\varepsilon}(t+\tau)
-
{\vartheta}^{\varepsilon}(t)
)
\,{\rm d}x{\rm d}t
\nonumber
\\
&
+
\int_{0}^{T-\tau}
\int_{t}^{t+\tau}
\int_{\Omega}
{\lambda}({x/\varepsilon},{p}^{\varepsilon},\vartheta^{\varepsilon},{r}^{\varepsilon})
\nabla {\vartheta}^{\varepsilon}(s)
\cdot
\nabla ({\vartheta}^{\varepsilon}(t+\tau) - {\vartheta}^{\varepsilon}(t) )
\,{\rm d}x {\rm d}s
{\rm d}t
\nonumber
\\
&
+
\int_{0}^{T-\tau}
\int_{t}^{t+\tau}
\int_{\Omega}
{c}_{w}
{\vartheta}^{\varepsilon}(s)
{a}({x/\varepsilon},{p}^{\varepsilon},\vartheta^{\varepsilon},{r}^{\varepsilon})
\nabla {p}^{\varepsilon}
\cdot\nabla  ({\vartheta}^{\varepsilon}(t+\tau) - {\vartheta}^{\varepsilon}(t) )
\,{\rm d}x {\rm d}s {\rm d}t
\nonumber
\\
&
+
\int_{0}^{T-\tau}
\int_{t}^{t+\tau}
\int_{\partial\Omega}
\alpha_{e} \vartheta^{\varepsilon}(s)  ({\vartheta}^{\varepsilon}(t+\tau) - {\vartheta}^{\varepsilon}(t) )
\,
{\rm d}{S} {\rm d}s {\rm d}t
\nonumber
\\
&
+
c_{w}
\int_{0}^{T-\tau}
\int_{t}^{t+\tau}
\int_{\partial\Omega}
\beta_{e}
\vartheta^{\varepsilon}(s)
(
{p}^{\varepsilon}(s)
-
{p}_{\infty}
)
({\vartheta}^{\varepsilon}(t+\tau) - {\vartheta}^{\varepsilon}(t) )
\,
{\rm d}{S} {\rm d}s {\rm d}t
\nonumber
\\
&
=
\int_{0}^{T-\tau}
\int_{t}^{t+\tau}
\int_{\Omega}
\alpha_2 \chi_{c}({x/\varepsilon}) f({p}^{\varepsilon},\vartheta^{\varepsilon},{r}^{\varepsilon})
({\vartheta}^{\varepsilon}(t+\tau) - {\vartheta}^{\varepsilon}(t) )
\,{\rm d}x {\rm d}s {\rm d}t
\nonumber
\\
&
+
\int_{0}^{T-\tau}
\int_{t}^{t+\tau}
\int_{\partial\Omega}
\alpha_{e} \vartheta_{\infty}(s)  ( {\vartheta}^{\varepsilon}(t+\tau) - {\vartheta}^{\varepsilon}(t) )
\,
{\rm d}{S} {\rm d}s {\rm d}t
.
\end{align*}
Hence we deduce
\begin{align}\label{est:comp_theta_05}
&
\quad
c_1
\int_{0}^{T-\tau}
\int_{\Omega}
|{\vartheta}^{\varepsilon}(t+\tau)
-
{\vartheta}^{\varepsilon}(t)|^2
\,{\rm d}x{\rm d}t
\nonumber
\\
&
\leq
\int_{0}^{T-\tau}
\left(
\left[
c_{w}
{b}\left(
{x/\varepsilon},{p}^{\varepsilon}(t+\tau),{r}^{\varepsilon}(t+\tau)
\right)
-
c_{w}  b\left({x/\varepsilon},{p}^{\varepsilon}(t),{r}^{\varepsilon}(t)\right)
\right]
{\vartheta}^{\varepsilon}(t),
{\vartheta}^{\varepsilon}(t+\tau) - {\vartheta}^{\varepsilon}(t)
\right)
\,{\rm d}t
\nonumber
\\
&
\quad
+
\int_{0}^{T-\tau}
\int_{\Omega}
\left[
\sigma({x/\varepsilon},{r}^{\varepsilon}(t+\tau))
-
\sigma({x/\varepsilon},{r}^{\varepsilon}(t))
\right]
{\vartheta}^{\varepsilon}(t)
(
{\vartheta}^{\varepsilon}(t+\tau)
-
{\vartheta}^{\varepsilon}(t)
)
\,{\rm d}x{\rm d}t
\nonumber
\\
&\quad
+
C \tau
\int_{0}^{T}
\left(
\|{\vartheta}^{\varepsilon}\|^2_{W^{1,2}(\Omega)}
+
\|{\vartheta}^{\varepsilon}\|_{L^{\infty}(\Omega)}
(\|{p}^{\varepsilon}\|_{W^{1,2}(\Omega)}+1)
\|{\vartheta}^{\varepsilon}\|_{W^{1,2}(\Omega)}
\right)
{\rm d}t.
\end{align}
The first integral on the right-hand side in \eqref{est:comp_theta_05}
can be further estimated using \eqref{con11a}
(Lipschitz continuity of $b$ in the second variable,
i.e. Lipschitz continuity of $\mathcal{S}$, see \eqref{notation_form_b}),
\eqref{max_principle} and the Young's inequality
to get
\begin{align}\label{est:comp_theta_06}
&
\int_{0}^{T-\tau}
\left(
\left[
c_{w}
{b}\left(
{x/\varepsilon},{p}^{\varepsilon}(t+\tau),{r}^{\varepsilon}(t+\tau)
\right)
-
c_{w}  b\left({x/\varepsilon},{p}^{\varepsilon}(t),{r}^{\varepsilon}(t)\right)
\right]
{\vartheta}^{\varepsilon}(t),
{\vartheta}^{\varepsilon}(t+\tau) - {\vartheta}^{\varepsilon}(t)
\right)
{\rm d}t
\nonumber
\\
\leq
&\;
c_{w}    C(\delta)
\|{\vartheta}^{\varepsilon}\|^2_{L^{\infty}({\Omega_T})}
\int_{0}^{T-\tau}
\int_{\Omega}
|
{b}\left(
{x/\varepsilon},{p}^{\varepsilon}(t+\tau),{r}^{\varepsilon}(t+\tau)
\right)
-
{b}\left(
{x/\varepsilon},{p}^{\varepsilon}(t),{r}^{\varepsilon}(t)
\right)
|^2
{\rm d}x{\rm d}t
\nonumber
\\
&
+
\delta
\int_{0}^{T-\tau}
\int_{\Omega}
|{\vartheta}^{\varepsilon}(t+\tau) - {\vartheta}^{\varepsilon}(t)|^2
{\rm d}x{\rm d}t
\nonumber
\\
\leq
&\;
C(\delta)
\int_{0}^{T-\tau}
\int_{\Omega}
\left(
{b}\left(
{x/\varepsilon},{p}^{\varepsilon}(t+\tau),{r}^{\varepsilon}(t+\tau)
\right)
-
{b}\left(
{x/\varepsilon},{p}^{\varepsilon}(t),{r}^{\varepsilon}(t)
\right),
{p}^{\varepsilon}(t+\tau) - {p}^{\varepsilon}(t)
\right)
{\rm d}x{\rm d}t
\nonumber
\\
&
+
\delta
\int_{0}^{T-\tau}
\int_{\Omega}
|{\vartheta}^{\varepsilon}(t+\tau) - {\vartheta}^{\varepsilon}(t)|^2
{\rm d}x{\rm d}t.
\end{align}
Now, choosing $\delta$ sufficiently small,
combining \eqref{est:comp_theta_05}, \eqref{est:comp_theta_06}
and using \eqref{max_principle},
\eqref{eq:a_priori_r_02},
\eqref{est:compact_u_00} and \eqref{est1_theta_a}
we get
\begin{equation}
\label{est:compactness_theta}
\int_0^{T-\tau}
\int_{\Omega}
|{\vartheta}^{\varepsilon}(t+\tau) - {\vartheta}^{\varepsilon}(t)|^2
\,{\rm d}x{\rm d}t
\leq
C \tau.
\end{equation}
Hence we conclude that
(using the estimates \eqref{est:energy_theta}, \eqref{est:compactness_theta} and \cite[Lemma 1.9]{AltLuckhaus1983})
\begin{equation*}%\label{}
{\vartheta}^{\varepsilon}  \rightarrow  {\vartheta}
\qquad  \textrm{ almost everywhere on } {\Omega_T}
\; \textmd{(along  a selected subsequence)}.
\end{equation*}
Further, from \eqref{eq:defn_weak_03} and \eqref{eq:defn_weak_04},
using
\eqref{cond:intr_perm},
\eqref{cond:rel_perm_I},
\eqref{cond:rel_perm_II},
\eqref{assum:bound_f},
\eqref{max_principle}, \eqref{est:energy_u} and \eqref{est1_theta_a}
we have
\begin{equation}\label{est_49}
\| \partial_t
{b}({x/\varepsilon},{p}^{\varepsilon},{r}^{\varepsilon})
\|_{L^{2}(0,T;{{{{W}^{1,2}(\Omega)}'}})}  \leq C
\end{equation}
and
\begin{equation}\label{est_50}
\| \partial_t
\left[
c_{w} {b}({x/\varepsilon},{p}^{\varepsilon},{r}^{\varepsilon}) \vartheta^{\varepsilon}
+
\sigma({x/\varepsilon},{r}^{\varepsilon}) \vartheta^{\varepsilon}
\right]
\|_{L^{2}(0,T;{{{{W}^{1,2}(\Omega)}'}})}  \leq C.
\end{equation}
Finally, we present uniform estimates for ${r}^{\varepsilon}$.
From \eqref{eq:memory_weak_form} we have
\begin{equation}\label{eq:a_priori_r_01}
{r}^{\varepsilon}(x,t+\tau)
-
{r}^{\varepsilon}(x,t)
=
\int_{t}^{t+\tau}
f({p}^{\varepsilon}(x,s),\vartheta^{\varepsilon}(x,s),{r}^{\varepsilon}(x,s))
\,
{\rm d}s.
\end{equation}
Integrating \eqref{eq:a_priori_r_01} over $\Omega \times (0,T-\tau)$
and
using \eqref{assum:bound_f} we arrive at the estimate
\begin{equation}\label{eq:a_priori_r_02}
\int_0^{T-\tau}
\int_{\Omega}
|{r}^{\varepsilon}(x,t+\tau)
-
{r}^{\varepsilon}(x,t)
|^2
\,{\rm d}x{\rm d}t
\leq
C_1 \tau^{2}
\leq
C_2 \tau.
\end{equation}
Further,
\begin{equation}\label{eq:a_priori_r_03}
|{r}^{\varepsilon}(x,t)|
\leq
\int_0^t
|f({p}^{\varepsilon}(x,s),\vartheta^{\varepsilon}(x,s),{r}^{\varepsilon}(x,s))|
\,
{\rm d}s
\leq
{C}T
\quad
\textmd{ almost everywhere in }
\Omega_{T},
\end{equation}
where ${C}$ is independent of $\varepsilon$.
%
%
%Hence, we conclude
%\begin{equation}\label{conv:r_01}
%{r}^{\varepsilon}  \rightharpoonup  {r}
%\qquad
%\textrm{weakly star in } L^{\infty}(0,T;L^{\infty}(\Omega)).
%\end{equation}
From \eqref{eq:memory_weak_form} we have, using \eqref{est:energy_u}, \eqref{est1_theta_a}, (iv)
and \cite[Proposition~1.28]{Roubicek2005},
\begin{equation*}%\label{}
\|{r}^{\varepsilon}(t)\|^2_{{W}^{1,2}(\Omega)}
\leq
C
\int_0^t
\left(
\|{p}^{\varepsilon}(t)\|^2_{{W}^{1,2}(\Omega)}
+
\|{\vartheta}^{\varepsilon}(t)\|^2_{{W}^{1,2}(\Omega)}
+
\|{r}^{\varepsilon}(t)\|^2_{{W}^{1,2}(\Omega)}
\right)
{\rm d}s
\quad
\textmd{   for all } t \in [0,T].
\end{equation*}
Hence applying the Gronwall argument
we get the a priori estimate
\begin{equation}\label{est1_r_a1}
\|{r}^{\varepsilon}(t)\|_{{W}^{1,2}(\Omega)} \leq C
\quad
\textmd{   for all } t \in [0,T].
\end{equation}
%Using \cite[Lemma~1.9]{AltLuckhaus1983} we deduce
%\begin{equation}
%\label{conv:r_02}
%{r}^{\varepsilon}  \rightarrow  {r}
%\qquad  \textrm{ almost everywhere on } {\Omega_T}.
%\end{equation}

%%%%%%%%%%%%%%%%%%%%%%%%%%%%%%%%%%%%%%%%%%%%%%%%%%%%%%%%%%%%%%%%%%%%%
\subsection{Passage to the limit for $\varepsilon \rightarrow 0$}
\label{subsec:passage_limit}
As a consequence of the preceding a priori estimates,
in particular,
\eqref{max_principle},
\eqref{bound_p},
\eqref{est:energy_u},
\eqref{est:energy_theta},
\eqref{est1_theta_a},
\eqref{est:compactness_theta},
\eqref{est_49},
\eqref{est_50},
\eqref{eq:a_priori_r_02},
\eqref{eq:a_priori_r_03}
and
\eqref{est1_r_a1}
and using
\cite[Lemma 1.9]{AltLuckhaus1983} and \cite[Lemma 3]{FiloKacur1995},
there exist functions
${p},{\vartheta},{r} \in L^2(0,T;W^{1,2}(\Omega)) \cap  L^{\infty}({\Omega_T})$,
$\beta,\omega \in L^2({\Omega_T})$ and the functionals
$\delta,\gamma \in L^2(0,T;{{{{W}^{1,2}(\Omega)}'}})$,
such that, along a selected
subsequence denoted again by $\left\{\varepsilon\right\}$, we have
\begin{align}
{p}^{\varepsilon} & \rightharpoonup  {p}
&&
\textrm{weakly in } L^2(0,T;W^{1,2}(\Omega)),
\label{conv:weak_u}
\\
{p}^{\varepsilon} &\rightharpoonup {p}
&&
\textrm{weakly star in } L^{\infty}({\Omega_T}),
\\
{p}^{\varepsilon} &\rightharpoonup {p}
&&
\textrm{weakly star in } L^{\infty}({\partial\Omega_T}),
\\
{p}^{\varepsilon} & \rightarrow {p}
&&
\textrm{almost everywhere in } {\Omega_T},
\label{conv:ae_p}
\\
{p}^{\varepsilon} & \rightarrow {p}
&&
\textrm{almost everywhere in } {\partial\Omega_T},
\\
{r}^{\varepsilon} & \rightharpoonup  {r}
&&
\textrm{weakly star in } L^{\infty}({\Omega_T}),
\label{conv:weak_r}
\\
{r}^{\varepsilon} & \rightarrow {r}
&&
\textrm{almost everywhere in } {\Omega_T},
\label{conv:ae_r}
\\
{b}({x/\varepsilon},{p}^{\varepsilon},{r}^{\varepsilon}) & \rightharpoonup  {\beta}
&&
\textrm{weakly in } L^2({\Omega_T}),
\label{}
\\
\partial_t {b}({x/\varepsilon},{p}^{\varepsilon},{r}^{\varepsilon})   &   \rightharpoonup    {\delta}
&&
\textrm{weakly in } L^2(0,T;{{{{W}^{1,2}(\Omega)}'}})
\label{conv:weak_der_01}
\end{align}
and
\begin{align}
{\vartheta}^{\varepsilon} & \rightharpoonup {\vartheta}
&&
\textrm{weakly in } L^2(0,T;W^{1,2}(\Omega)),
\label{}
\\
{\vartheta}^{\varepsilon} & \rightarrow {\vartheta}
&&
\textrm{almost everywhere in } {\Omega_T},
\label{conv:ae_theta}
\\
{\vartheta}^{\varepsilon} &\rightharpoonup {\vartheta}
&&
\textrm{weakly in } L^p({\Omega_T}), \; 1<p<+\infty,
\label{}
\\
{\vartheta}^{\varepsilon} &\rightharpoonup {\vartheta}
&&
\textrm{weakly star in } L^{\infty}({\Omega_T}),
\label{}
\\
c_{w} {b}({x/\varepsilon},{p}^{\varepsilon},{r}^{\varepsilon}) {\vartheta}^{\varepsilon}
+
\sigma({x/\varepsilon},{r}^{\varepsilon})   {\vartheta}^{\varepsilon}
& \rightharpoonup
{\omega}
&&\textrm{weakly in } L^2({\Omega_T}),
\label{}
\\
\partial_t
\left[
c_{w} {b}({x/\varepsilon},{p}^{\varepsilon},{r}^{\varepsilon})  {\vartheta}^{\varepsilon}
+
\sigma({x/\varepsilon},{r}^{\varepsilon})  {\vartheta}^{\varepsilon} \right]
&
\rightharpoonup
{\gamma}
&&
\textrm{weakly in } L^{2}(0,T;{(W^{1,2}(\Omega))'}).
\label{limit_der_heat}
\end{align}

Before we identify the limits ${p}$, ${\vartheta}$ and ${r}$,
we present some auxiliary convergence results (see Lemma~\ref{lem:weak_conv_01} and Lemma~\ref{lem:cinv:sup01}),
recall the definition and
review basic properties of the two-scale convergence
(see
Definition~\ref{def:two_scale_conv},
Remark~\ref{rem:two_scale_conv},
Lemma~\ref{cio_lemma_i},
Lemma~\ref{cio_lemma_ii},
Theorem~\ref{thm:two_scale_compactness}
and
Theorem~\ref{thm:two_scale_aux}).
We refer the reader to \cite{Allaire1992,Nguetseng} for a more thorough discussion.
The main result of this section is stated in Theorem~\ref{theorem:hom_limit_system} below.

\begin{definition}[\cite{Allaire1992}]\label{def:two_scale_conv}
Let ${v}^{\varepsilon}$ be a sequence in ${L}^{2}(\Omega)$.
Such a sequence  ${v}^{\varepsilon}$ is said to (weakly) two-scale converge
to a limit ${v}^{0} \in {L}^{2}(\Omega \times \mathcal{Y})$
(with respect to the scale $\left\{\varepsilon \right\}$)
\begin{equation}\label{def:two-scale_conv}
\int_{\Omega}
{v}^{\varepsilon}(x)
\varphi\left( x , \frac{x}{\varepsilon} \right)
\,{\rm d}x
\rightarrow
\int_{\Omega}
\int_{\mathcal{Y}}
{v}^{0}(x,y)
\varphi\left( x , y \right)
\,{\rm d}y {\rm d}x
\end{equation}
as $\varepsilon \rightarrow 0$
for all $\varphi \in {L}^{2}(\Omega; C_{per}({\mathcal{Y}}))$.
\end{definition}
\begin{remark}\label{rem:two_scale_conv}
Note that \eqref{def:two-scale_conv} holds also for any $\varphi$
of the form
$\varphi(x,y) = \varphi_{1}(y)  \varphi_{2}(x,y)$
with $\varphi_{1} \in {L}^{\infty}(\mathcal{Y})$
and
$\varphi_{2} \in {L}^{2}_{per}({\mathcal{Y}}; C(\overline{\Omega}))$,
see~\cite[Remark~9.4]{cioranescu}.
\end{remark}
\begin{lemma}\cite[Lemma 9.1 (i)]{cioranescu}
\label{cio_lemma_i}
Let
${v} \in L^p(\Omega; C_{per}({\mathcal{Y}}))$
with
$1 \leq p < +\infty$.
Then the function
${v}(\cdot,{\cdot/\varepsilon}) \in {L}^{p}(\Omega)$
with
$$
\| {v}(\cdot,{\cdot/\varepsilon}) \|_{L^p(\Omega)}
\leq
\| {v}(\cdot,\cdot) \|_{L^p(\Omega; C_{per}({\mathcal{Y}}))}
$$
and  ${v}(\cdot,{\cdot/\varepsilon})$ converges weakly in $L^p(\Omega)$ to
$$
\int_{\mathcal{Y}} {v}(\cdot,y) {dy}.
$$
\end{lemma}
\begin{lemma}\cite[Lemma 9.1 (ii)]{cioranescu}
\label{cio_lemma_ii}
Suppose that
$\varphi(x,y) = \varphi_{1}(x)  \varphi_{2}(y)$
with $\varphi_{1} \in {L}^{s}(\Omega)$
and
$\varphi_{2} \in {L}^{q}(\Omega)$ with
$1 \leq s,q < +\infty$ and such that
$$
\frac{1}{s}
+
\frac{1}{q}
=
\frac{1}{p}.
$$
Then the function
${\varphi}(\cdot,{\cdot/\varepsilon}) \in {L}^{p}(\Omega)$
and
$$
\varphi\left( \cdot , \frac{\cdot}{\varepsilon} \right)
\rightharpoonup
 \varphi_{1}(\cdot)
\int_{\mathcal{Y}}
\varphi_{2}(y)
\,{\rm d}y
\;
\textmd{ weakly in }
\;
{L}^{p}(\Omega).
$$
\end{lemma}
\begin{theorem}[\cite{Allaire1992}]\label{thm:two_scale_compactness}
If $v^{\varepsilon}$ is a bounded sequence in ${L}^{2}(\Omega)$
then there exists a function ${v}^{0} \in {L}^{2}(\Omega \times \mathcal{Y})$
and a  subsequence of  $v^{\varepsilon}$
which two-scale converges to ${v}^{0}$.
Moreover, this two-scale convergent subsequence converges weakly in  ${L}^{2}(\Omega)$ to
$$
{v}(x) =
\int_{\mathcal{Y}}
{v}^{0}(x,y)
\,{\rm d}y.
$$
\end{theorem}
If the sequence $v^{\varepsilon}$ is bounded in $W^{1,2}(\Omega)$
we have the following
\begin{theorem}\cite[Chapter 2]{Allaire1992}\label{thm:two_scale_aux}
Let $v^{\varepsilon}$ be a bounded sequence in $W^{1,2}(\Omega)$
that converges weakly to a limit $v$ in $W^{1,2}(\Omega)$.
Then, $v^{\varepsilon}$ two-scale converges to $v$, and
there exists a function ${v_1} \in L^{2}(\Omega; W^{1,2}_{per}({\mathcal{Y}}))$
such that,
up to a subsequence,
\begin{equation*}%\label{conv_22}
\nabla v^{\varepsilon}  \rightarrow  \nabla_x v + \nabla_y {v_1}
\quad
\textrm{in the two-scale sense}.
\end{equation*}
\end{theorem}
\begin{lemma}\cite[Lemma 4.2]{LiSun2010}
\label{lem:weak_conv_01}
Let $(\mathcal{M},\mu)$ be a $\sigma$-finite measure space
and $f_n$ and $g_n \in L^1(\mathcal{\mathcal{M}})$
be two sequences of functions, and let $f,g,h \in L^1(\mathcal{M})$ such that
(as $n \rightarrow +\infty$)
\begin{align*}
f_n & \rightarrow  f
&&
\textrm{almost everywhere in } \mathcal{M},
%\label{}
\\
g_n & \rightharpoonup g
&&
\textrm{weakly in } L^1(\mathcal{M}),
%\label{}
\\
f_n g_n
& \rightharpoonup
 h
&&
\textrm{weakly in } L^1(\mathcal{M}).
%\label{}
\end{align*}
Then $h=fg$ almost everywhere in $\mathcal{M}$.
\end{lemma}

Notice that the oscillations in the model are only due
to the spatial variable $x$ (through the characteristic functions $\chi^{\varepsilon}_{a}$ and $\chi^{\varepsilon}_{c}$).
In the homogenization procedure, the time variable
$t$ plays the role of a parameter.

\begin{lemma}\label{lem:cinv:sup01}
There exists a subsequence of $\left\{\varepsilon\right\}$,
still denoted by $\left\{\varepsilon\right\}$, such that
\begin{align}
\varrho_{w}\chi_{c}({x/\varepsilon}) \phi_{c}\left( {r}^{\varepsilon} \right)  \mathcal{S}({p}^{\varepsilon})
&
\rightharpoonup
\varrho_{w} \chi_{c}^* \phi_{c}(r) \mathcal{S}({p})
&&
\textrm{weakly in } L^2({\Omega_T}),
\label{conv:101}
\\
%\partial_t
%\left[
%\varrho_{w} \chi_{c}({x/\varepsilon})
%\phi_{c}\left( {r}^{\varepsilon} \right)  \mathcal{S}({p}^{\varepsilon})
%\right]
%&
%\rightharpoonup
%\partial_t
%\left[
%\varrho_{w} \chi_{c}^* \phi_{c}(r) \mathcal{S}({p})
%\right]
%&&
%\textrm{weakly in } L^2(0,T;{{{{W}^{1,2}(\Omega)}'}}),
%\label{conv:102}
%\\
\varrho_{w} \chi_{a}({x/\varepsilon})
\phi_{a}  \mathcal{S}({p}^{\varepsilon})
&
\rightharpoonup
\varrho_{w} \chi_{a}^*
\phi_{a}\mathcal{S}({p})
&&
\textrm{weakly in } L^2({\Omega_T}),
\label{conv:101b}
\\
{b}({x/\varepsilon},{p}^{\varepsilon},{r}^{\varepsilon})
&
\rightharpoonup
{b}^{*}({p},{r})
&&
\textrm{weakly in } L^2({\Omega_T}),
\label{conv:101c}
\\
%\partial_t
%\left[
%\varrho_{w} \chi_{a}({x/\varepsilon})
%\phi_{a}\mathcal{S}({p}^{\varepsilon})
%\right]
%&
%\rightharpoonup
%\partial_t
%\left[
%\varrho_{w} \chi_{a}^*
%\phi_{a}  \mathcal{S}({p})
%\right]
%&&
%\textrm{weakly in } L^2(0,T;{{{{W}^{1,2}(\Omega)}'}}),
%\label{conv:102b}
%\\
\partial_t
\left[
{b}({x/\varepsilon},{p}^{\varepsilon},{r}^{\varepsilon})
\right]
&
\rightharpoonup
\partial_t
\left[
{b}^{*}({p},{r})
\right]
&&
\textrm{weakly in } L^2(0,T;{{{{W}^{1,2}(\Omega)}'}}),
\label{conv:102c}
\\
c_{w} {b}({x/\varepsilon},{p}^{\varepsilon},{r}^{\varepsilon}) \vartheta^{\varepsilon}
+
\sigma({x/\varepsilon},{r}^{\varepsilon}) \vartheta^{\varepsilon}
&
\rightharpoonup
c_{w} {b}^{*}({p},{r})\vartheta
+
\sigma^*(r) \vartheta
&&\textrm{weakly in } L^2({\Omega_T})
\label{conv:103}
\end{align}
and
\begin{multline}\label{conv:104}
\partial_t  \left[
c_{w} {b}({x/\varepsilon},{p}^{\varepsilon},{r}^{\varepsilon}) \vartheta^{\varepsilon}
+
\sigma({x/\varepsilon},{r}^{\varepsilon}) \vartheta^{\varepsilon}
\right]
\\
\rightharpoonup
\partial_t \left[
c_{w} {b}^{*}({p},{r})\vartheta
+
\sigma^*(r) \vartheta
\right]
\quad
\textrm{ weakly in } L^2(0,T;{{(W^{1,2}(\Omega))'}}).
\end{multline}
\end{lemma}

\begin{proof}
From \eqref{conv:ae_p} and \eqref{conv:ae_r}
we have
$\phi_{c}({r}^{\varepsilon})  \mathcal{S}({p}^{\varepsilon})
\rightarrow
\phi_{c}(r) \mathcal{S}({p})$
almost everywhere in $\Omega_T$.
Using Lemma~\ref{cio_lemma_ii} (setting $\varphi_{1}(x) \equiv 1$)
we also have
$\chi_{c}({x/\varepsilon}) \rightharpoonup  \chi_{c}^*$ weakly in  $L^2({\Omega_T})$.
Taking into account \eqref{con11a} and \eqref{bound_phi},
\eqref{conv:101} can be shown using Lemma~\ref{lem:weak_conv_01}.
Similar arguments apply to \eqref{conv:101b}, \eqref{conv:101c} and \eqref{conv:103}.
Now, \eqref{conv:102c} and \eqref{conv:104} follows from \eqref{est_49} and \eqref{est_50},
recall also \eqref{conv:weak_der_01} and \eqref{limit_der_heat}.
%
%Limits \eqref{conv:101} and \eqref{conv:102} are proven in \cite[Corollary~4.3]{Nandakumaran2001}.
%The limit \eqref{conv:103} follows from Lemma \ref{lem:weak_conv_01},
%\eqref{conv:101} and the estimate
%\begin{equation*}
%\Big\|
%c_{w} {b}({x/\varepsilon},{p}^{\varepsilon},{r}^{\varepsilon}) \vartheta^{\varepsilon}
%+
%\sigma({x/\varepsilon},{r}^{\varepsilon}) \vartheta^{\varepsilon}
%\Big\|_{L^{2}({\Omega_T})}
%\leq c,
%\end{equation*}
%where $c$ is independent of $\varepsilon$. Now, \eqref{conv:104} follows from
%\eqref{conv:103} and \eqref{limit_der_heat}.
\hfill $\square$
\end{proof}

Taking into account the convergence results \eqref{conv:weak_u}--\eqref{limit_der_heat}
together with Lemma~\ref{lem:cinv:sup01} we prove the following theorem
which is the main result of this section.
\begin{theorem}\label{theorem:hom_limit_system}
There exist a function ${p}_{1} \in L^{2}(\Omega_T; W^{1,2}_{per}(\mathcal{Y}))$
and a subsequence ${p}^{\varepsilon}$ (still denoted by ${p}^{\varepsilon}$)
such that
\begin{equation*}
\nabla {p}^{\varepsilon}
\rightarrow
\nabla_x {p} + \nabla_{y} {p}_{1}(x,y,t)
\quad
\textrm{in the two-scale sense}.
%\label{}
\end{equation*}
Similarly,
there exist a function ${{\vartheta}_1} \in L^{2}(\Omega_T; W^{1,2}_{per}({\mathcal{Y}}))$
and a subsequence ${\vartheta}^{\varepsilon}$ (still denoted by ${\vartheta}^{\varepsilon}$)
such that
\begin{equation*}
\nabla {\vartheta}^{\varepsilon}
\rightarrow
\nabla_x {\vartheta} + \nabla_{y} {{\vartheta}_1}(x,y,t)
\quad
\textrm{in the two-scale sense}.
%\label{}
\end{equation*}
Further,
the pairs $({p},{p}_{1})$ and $({\vartheta},{\vartheta}_{1})$
and the function $r$
satisfy the following two-scale homogenized coupled problem
\begin{align}
\label{weak_form_101}
&
\int_{0}^{T}
\langle
\int_{\mathcal{Y}}
\partial_t   {b}({y},{p},{r})   {\rm d}y,
{\varphi}
\rangle
{\rm d}t
\nonumber
\\
&
+
\int_{\Omega_T}
\int_{\mathcal{Y}}
{a}({y},{p},\vartheta,{r})
\left(
\nabla_x {p} + \nabla_y {p}_{1}(x,y,t)
\right)
\cdot
\left(
\nabla_x {\varphi} + \nabla_y {{\varphi}_1}(x,y,t)
\right)
\,{\rm d}y {\rm d}x {\rm d}t
\nonumber
\\
&
+
\int_{\partial\Omega_T}
\beta_{e}
{p}   \varphi
\,
{\rm d}{S} {\rm d}t
\nonumber
\\
=
&
\;
\int_{\Omega_T}
\int_{\mathcal{Y}}
\alpha_{1}  \chi_{c}({y})  f( {p},\vartheta,{r} )
\varphi
\;
{\rm d}y {\rm d}x {\rm d}t
+
\int_{\partial\Omega_T}
\beta_{e}
{p}_{\infty}   \varphi
\,
{\rm d}{S} {\rm d}t
\end{align}
and
\begin{align}\label{weak_form_102}
&
\int_{0}^{T}
\langle
\int_{\mathcal{Y}}
\partial_t
\left[
c_{w} {b}({y},{p},{r})  {\vartheta}
+
\sigma\left( {y} , {r} \right)  {\vartheta}
\right]
{\rm d}y,
\psi
\rangle
{\rm d}t
\nonumber
\\
&
+
\int_{\Omega_T}
\int_{\mathcal{Y}}
\lambda({y},{p},\vartheta,{r})
\left(
\nabla_x {\vartheta} + \nabla_y {{\vartheta}_1}(x,y,t)
\right)
\cdot
\left( \nabla_x \psi + \nabla_y {\psi_1}(x,y,t) \right)
\,{\rm d}y{\rm d}x{\rm d}t
\nonumber
\\
&
+
\int_{\Omega_T}
\int_{\mathcal{Y}}
c_{w}
{\vartheta}\,
{a}({y},{p},\vartheta,{r})
\left(
\nabla_x {p} + \nabla_y {p}_{1}(x,y,t)
\right)
\cdot
\left( \nabla_x \psi + \nabla_y {\psi_1}(x,y,t) \right)
\,{\rm d}y{\rm d}x{\rm d}t
\nonumber
\\
&
+
\int_{\partial\Omega_T}
\alpha_{e} \vartheta  \psi
\,
{\rm d}{S} {\rm d}t
+
c_{w}
\int_{\partial\Omega_T}
\beta_{e} \vartheta
(
{p}
-
{p}_{\infty}
) \psi
\,
{\rm d}{S} {\rm d}t
\nonumber
\\
=
&
\;
\int_{\Omega_T}
\int_{\mathcal{Y}}
\alpha_{2}  \chi_{c}({y})  f( {p},\vartheta,{r} )
\psi
\;
{\rm d}y {\rm d}x {\rm d}t
+
\int_{\partial\Omega_T}
\alpha_{e} \vartheta_{\infty}  \psi
\,
{\rm d}{S} {\rm d}t
\end{align}
for all test functions
${\varphi},\psi\in  C^{\infty}({\Omega}_T)$
and
${{\varphi}_1},{\psi_1} \in C_0^{\infty}(\Omega_T;C^{\infty}_{per}({\mathcal{Y}}))$
and
\begin{equation}\label{eq:memory_limit}
{r}(x,t)
=
\int_0^t
{f}({p}(x,s),\vartheta(x,s),{r}(x,s))
{\rm d}s.
\end{equation}
\end{theorem}
\begin{proof}
Let ${\varphi},\psi \in C^{\infty}({\Omega}_T)$
and ${{\varphi}_1},{\psi_1} \in C_0^{\infty}(\Omega_T;C^{\infty}_{per}({\mathcal{Y}}))$.
We take test functions as
\begin{equation}\label{test_phi}
{\varphi}^{\varepsilon} = {\varphi}(x,t) + \varepsilon {{\varphi}_1}( x,{x/\varepsilon},t)
\end{equation}
and
\begin{equation}\label{test_psi}
\psi^{\varepsilon} =  \psi(x,t) + \varepsilon {\psi_1}(x,{x/\varepsilon},t),
\end{equation}
respectively, in \eqref{eq:defn_weak_03} and \eqref{eq:defn_weak_04}.
Note that, by \eqref{conv:102c} and the strong convergence of ${\varphi}^{\varepsilon}$
to ${\varphi}$ in $L^p({\Omega_T})$ we deduce
\begin{equation}\label{conv:201}
\int_{0}^{T}
\langle
\partial_t
{b}({x/\varepsilon},{p}^{\varepsilon},{r}^{\varepsilon}),
{\varphi}^{\varepsilon}
\rangle
{\rm d}t
\rightarrow
\int_{0}^{T}
\langle
\partial_t
{b}^{*}({p},{r}),
{\varphi}
\rangle
{\rm d}t.
\end{equation}
Similarly, by \eqref{conv:104} we have
\begin{equation}\label{conv:202}
\int_{0}^{T}
\langle
\partial_t  \left[
c_{w}{b}({x/\varepsilon},{p}^{\varepsilon},{r}^{\varepsilon})  {\vartheta}^{\varepsilon}
+
\sigma\left( {x/\varepsilon}  , {r}^{\varepsilon} \right)   {\vartheta}^{\varepsilon}
\right],
\psi^{\varepsilon}
\rangle
\,{\rm d}t
\rightarrow
\int_{0}^{T}
\langle
\partial_t
\left[
c_{w}{b}^{*}({p},{r}) {\vartheta} +  \sigma^*(r)  {\vartheta}
\right],
\psi
\rangle
{\rm d}t.
\end{equation}
Further, using \eqref{con11a}--\eqref{cond:viscosity},
\eqref{conv:weak_u} and \eqref{conv:ae_theta} we deduce
\begin{equation*}%\label{conv:203}
\varrho_{w}\frac{ {k}_{c}({r}^{\varepsilon}){k}_{R}(\mathcal{S}({p}^{\varepsilon}))  }{\mu({\vartheta}^{\varepsilon})}
\nabla {p}^{\varepsilon}
\rightharpoonup
\varrho_{w}\frac{ {k}_{c}({r}){k}_{R}(\mathcal{S}({p}))  }{\mu({\vartheta})}
\nabla {p}
\qquad
\textrm{weakly in } L^{2}({\Omega_T})^{2}
\end{equation*}
and
\begin{equation*}%\label{conv:203}
\varrho_{w}\frac{ {k}_{a} {k}_{R}(\mathcal{S}({p}^{\varepsilon}))  }{\mu({\vartheta}^{\varepsilon})}
 \nabla {p}^{\varepsilon}
\rightharpoonup
\varrho_{w}\frac{ {k}_{a} {k}_{R}(\mathcal{S}({p}))  }{\mu({\vartheta})}
\nabla {p}
\qquad
\textrm{weakly in } L^{2}({\Omega_T})^{2}.
\end{equation*}
Indeed, ${\chi_{c}(y)} \in L^{\infty}({\mathcal{Y}})$
and
$\nabla {\varphi}^{\varepsilon} \in L^{2}_{per}(\mathcal{Y};C(\overline{\Omega}))$
so that $\chi_{c}({x/\varepsilon})\nabla{\varphi}^{\varepsilon}$
can be used as a test function in the two-scale convergence (see Remark~\ref{rem:two_scale_conv}),
namely,
\begin{align}\label{conv:204}
&
\varrho_{w}
\int_{\Omega_T}
\chi_{c}({x/\varepsilon})
\frac{ {k}_{c}({r}^{\varepsilon}){k}_{R}(\mathcal{S}({p}^{\varepsilon}))  }{\mu({\vartheta}^{\varepsilon})}
\nabla {p}^{\varepsilon}
\cdot
\nabla{\varphi}^{\varepsilon}
{\rm d}x {\rm d}t
=
\varrho_{w}
\int_{\Omega_T}
\frac{ {k}_{c}({r}^{\varepsilon}){k}_{R}(\mathcal{S}({p}^{\varepsilon}))  }{\mu({\vartheta}^{\varepsilon})}
\nabla {p}^{\varepsilon}
\cdot
\left[ \chi_{c}({x/\varepsilon})  \nabla{\varphi}^{\varepsilon}\right]
{\rm d}x {\rm d}t
\nonumber
\\
&
\rightarrow
\varrho_{w}
\int_{\Omega_T}
\int_{\mathcal{Y}}
\chi_{c}(y)
\frac{ {k}_{c}({r}){k}_{R}(\mathcal{S}({p}))  }{\mu({\vartheta})}
\left(
\nabla_x {p} + \nabla_y {p}_{1}(x,y,t)
\right)
\cdot
\left( \nabla_x {\varphi} + \nabla_y {{\varphi}_1}(x,y,t) \right)
{\rm d}y{\rm d}x {\rm d}t
\end{align}
and, similarly,
\begin{align}\label{conv:205}
&
\varrho_{w}
\int_{\Omega_T}
\chi_{a}({x/\varepsilon})
\frac{ {k}_{a} {k}_{R}(\mathcal{S}({p}^{\varepsilon}))  }{\mu({\vartheta}^{\varepsilon})}
\nabla {p}^{\varepsilon}
\cdot
\nabla{\varphi}^{\varepsilon}
{\rm d}x {\rm d}t
=
\varrho_{w}
\int_{\Omega_T}
\frac{ {k}_{a} {k}_{R}(\mathcal{S}({p}^{\varepsilon}))  }{\mu({\vartheta}^{\varepsilon})}
\nabla {p}^{\varepsilon}
\cdot
\left[ \chi_{a}({x/\varepsilon})  \nabla{\varphi}^{\varepsilon}\right]
{\rm d}x {\rm d}t
\nonumber
\\
&
\rightarrow
\varrho_{w}
\int_{\Omega_T}
\int_{\mathcal{Y}}
\chi_{a}(y)
\frac{ {k}_{a} {k}_{R}(\mathcal{S}({p}))  }{\mu({\vartheta})}
\left(
\nabla_x {p} + \nabla_y {p}_{1}(x,y,t)
\right)
\cdot
\left( \nabla_x {\varphi} + \nabla_y {{\varphi}_1}(x,y,t) \right)
{\rm d}y{\rm d}x {\rm d}t.
\end{align}
In the same manner we conclude that
\begin{align}\label{conv:206}
&
\int_{\Omega_T}
\lambda({x/\varepsilon},{p}^{\varepsilon},\vartheta^{\varepsilon},{r}^{\varepsilon})
\nabla {\vartheta}^{\varepsilon}
\cdot
\nabla\psi^{\varepsilon}
\,{\rm d}x {\rm d}t
\nonumber
\\
&
\rightarrow
\int_{\Omega_T}
\int_{\mathcal{Y}}
\lambda({y},{p},\vartheta,{r})
\left(
\nabla_x {\vartheta} + \nabla_y {{\vartheta}_1}(x,y,t)
\right)
\cdot
\left( \nabla_x \psi + \nabla_y {\psi_1}(x,y,t) \right)
\,{\rm d}y{\rm d}x{\rm d}t
\end{align}
and
\begin{align}\label{conv:207}
&
\int_{\Omega_T}
c_{w}
\vartheta^{\varepsilon}
{a}({x/\varepsilon},{p}^{\varepsilon},\vartheta^{\varepsilon},{r}^{\varepsilon})
\nabla {p}^{\varepsilon}
\cdot \nabla\psi^{\varepsilon}
\,{\rm d}x {\rm d}t
\nonumber
\\
&
\rightarrow
\int_{\Omega_T}
\int_{\mathcal{Y}}
c_{w}
{\vartheta}
\,
{a}({y},{p},\vartheta,{r})
\left(
\nabla_x {p} + \nabla_y {p}_{1}(x,y,t)
\right)
\cdot
\left( \nabla_x \psi + \nabla_y {\psi_1}(x,y,t) \right)
\,{\rm d}y{\rm d}x{\rm d}t
\end{align}
and finally
\begin{equation}\label{conv:208}
\int_{\Omega_T}
\chi_{c}({x/\varepsilon})
f({p}^{\varepsilon},\vartheta^{\varepsilon},{r}^{\varepsilon})
\varphi^{\varepsilon}
\;
{\rm d}x {\rm d}t
\rightarrow
\int_{\Omega_T}
\int_{\mathcal{Y}}
 \chi_{c}({y})
f( {p},\vartheta,{r} )
\varphi
\;
{\rm d}y {\rm d}x {\rm d}t.
\end{equation}
Thus, choosing \eqref{test_phi} and \eqref{test_psi}, respectively,
as test functions in \eqref{eq:defn_weak_03} and \eqref{eq:defn_weak_04}
the above convergences \eqref{conv:201}--\eqref{conv:202}
and
\eqref{conv:204}--\eqref{conv:208}
are sufficient for taking the limit $\varepsilon_j \rightarrow 0$
as $j \rightarrow +\infty$
(along a selected subsequence) to get \eqref{weak_form_101}--\eqref{weak_form_102}.
The proof of Theorem \ref{theorem:hom_limit_system} is complete.
\hfill$\square$
\end{proof}

%%%%%%%%%%%%%%%%%%%%%%%%%%%%%%%%%%%%%%%%%%%%%%%%%%%%%%%%%%%%%%%%%%%%%%%%%%%%%%%%%%%
\subsection{The homogenized system}
\label{subsec:homogenized_model}
%%%%%%%%%%%%%%%%%%%%%%%%%%%%%%%%%%%%%%%%%%%%%%%%%%%%%%%%%%%%%%%%%%%%%%%%%%%%%%%%%%%
In this section we complete the proof of Theorem~\ref{thm:main_result}.
We will eliminate the variables ${p}_{1}$ and $\vartheta_{1}$
from \eqref{weak_form_101} and \eqref{weak_form_102} to obtain the closed system
for the remaining unknowns.
This can be handled in the following way.
Setting ${{\varphi}} = 0$ in \eqref{weak_form_101} we can write
\begin{equation}\label{eq:homogenization_01}
\int_{\Omega_T}
\int_{\mathcal{Y}}
{a}({y},{p},\vartheta,{r})
\left(
\nabla_x {p} + \nabla_y {p}_{1}(x,y,t)
\right)
\cdot
\nabla_y {{\varphi}_1}(x,y,t)
\,{\rm d}y {\rm d}x {\rm d}t
=
0
\end{equation}
for all ${{\varphi}_1} \in C_0^{\infty}(\Omega_T;C^{\infty}_{per}({\mathcal{Y}}))$.
Here we determine ${p}_{1}$ (up to a constant) as
\begin{equation}\label{eq:u_1}
{p}_{1}(x,y,t) = \nabla_x {p}(x,t) \cdot \bfw(x,y,t),
\quad
\bfw(x,y,t) = (w_1,w_2),
\end{equation}
where $\bfw(x,y,t)$ can be obtained as follows.
Combining \eqref{eq:homogenization_01} and \eqref{eq:u_1} we get
\begin{equation}\label{}
\int_{\Omega_T}
\int_{\mathcal{Y}}
{a}({y},{p},\vartheta,{r})
\left(
\nabla_x {p}
+
\nabla_y \left[  \nabla_x {p}(x,t) \cdot \bfw(x,y,t)   \right]
\right)
\cdot
\nabla_y {{\varphi}_1}(x,y,t)
\,{\rm d}y {\rm d}x {\rm d}t
=
0
\end{equation}
for all ${{\varphi}_1} \in C_0^{\infty}(\Omega_T;C^{\infty}_{per}({\mathcal{Y}}))$.
The above equation may be rewritten as
\begin{equation}\label{eq:homogenization_02}
\int_{\Omega_T}
\int_{\mathcal{Y}}
{a}({y},{p},\vartheta,{r})
\nabla_x {p}
\left(
{\bfI_d} + \nabla_y \bfw(x,y,t)
\right)
\cdot
\nabla_y {{\varphi}_1}(x,y,t)
\,{\rm d}y{\rm d}x{\rm d}t
=0
\end{equation}
for all ${{\varphi}_1} \in C_0^{\infty}(\Omega_T;C^{\infty}_{per}({\mathcal{Y}}))$.
Here and in what follows, $\nabla_y\bfw$ is the matrix
$(\nabla_y\bfw)_{ij} = \partial {w}_{j}/\partial {y}_{i}$.
Integrating by parts in \eqref{eq:homogenization_02} we deduce
\begin{equation*}
-\int_{\Omega_T}
\int_{\mathcal{Y}}
{\nabla_y
\cdot
\left[
{a}({y},{p},\vartheta,{r})
\nabla_x {p}
\left(
{\bfI_d} + \nabla_y\bfw(x,y,t)
\right)
\right]}
{{\varphi}_1}(x,y,t)
\,{\rm d}y{\rm d}x{\rm d}t
=
0.
\end{equation*}
From this, it is easy to verify that
$\bfw$ satisfies the periodic cell problem
\eqref{eq:cell_problem_w}.

Now let ${{\varphi}_1}=0$  in \eqref{weak_form_101}. Then we have
\begin{align}
\label{}
&
\int_{0}^{T}
\langle
\int_{\mathcal{Y}}
\partial_t   {b}({y},{p},{r})    {\rm d}y,
{\varphi}
\rangle
{\rm d}t
\nonumber
\\
&
+
\int_{\Omega_T}
\int_{\mathcal{Y}}
{a}({y},{p},\vartheta,{r})
\left(
\nabla_x {p} + \nabla_y {p}_{1}(x,y,t)
\right)
\cdot
\nabla_x {\varphi}
\,{\rm d}y {\rm d}x {\rm d}t
\nonumber
\\
&
+
\int_{\partial\Omega_T}
\beta_{e}
{p}   \varphi
\,
{\rm d}{S} {\rm d}t
\nonumber
\\
=
&
\;
\int_{\Omega_T}
\int_{\mathcal{Y}}
\alpha_{1}  \chi_{c}({y})  f( {p},\vartheta,{r} )
\varphi
\;
{\rm d}y {\rm d}x {\rm d}t
+
\int_{\partial\Omega_T}
\beta_{e}
{p}_{\infty}   \varphi
\,
{\rm d}{S} {\rm d}t
\end{align}
and taking into account \eqref{eq:u_1} we can write
\begin{align}
\label{eq:weak_form_u_proof}
&
\int_{0}^{T}
\langle
\partial_t
\int_{\mathcal{Y}}
{b}({y},{p},{r})
 {\rm d}y,
{\varphi}
\rangle
{\rm d}t
\nonumber
\\
&
+
\int_{\Omega_T}
\nabla_x {p} \;
\left(
\int_{\mathcal{Y}}
{a}({y},{p},\vartheta,{r})
\left( {\bfI_d} + \nabla_y \bfw(x,y,t) \right)
{\rm d}y
\right)
\cdot
\nabla_x {\varphi}
\, {\rm d}x {\rm d}t
\nonumber
\\
&
+
\int_{\partial\Omega_T}
\beta_{e}
{p}   \varphi
\,
{\rm d}{S} {\rm d}t
\nonumber
\\
&
=
\alpha_{1}
\int_{\mathcal{Y}}
\chi_{c}({y})
{\rm d}y
\int_{\Omega_T}
f( {p},\vartheta,{r} )
\varphi
\;
{\rm d}x {\rm d}t
+
\int_{\partial\Omega_T}
\beta_{e}
{p}_{\infty}   \varphi
\,
{\rm d}{S} {\rm d}t
\end{align}
for all ${\varphi} \in C^{\infty}({\Omega}_T)$.
Note that \eqref{eq:weak_form_u_proof} represents the weak form
of the problem \eqref{eq:hom_01} and \eqref{eq:hom_03}.

Let $\psi=0$ in \eqref{weak_form_102}, we get
\begin{align}\label{eq:201}
&
\int_{\Omega_T}
\int_{\mathcal{Y}}
{\lambda}({y},{p},\vartheta,{r})
\left(
\nabla_x {\vartheta} + \nabla_y {{\vartheta}_1}(x,y,t)
\right)
\cdot
\nabla_y {\psi_1}(x,y,t)
\,{\rm d}y{\rm d}x{\rm d}t
\nonumber
\\
&
+
\int_{\Omega_T}
\int_{\mathcal{Y}}
c_{w}
{\vartheta}\,
{a}({y},{p},\vartheta,{r})
\left(
\nabla_x {p} + \nabla_y {p}_{1}(x,y,t)
\right)
\cdot
\nabla_y {\psi_1}(x,y,t)
\,{\rm d}y{\rm d}x{\rm d}t
\nonumber
\\
&
=
0
\end{align}
for all ${\psi_1} \in C_0^{\infty}(\Omega_T;C^{\infty}_{per}({\mathcal{Y}}))$.
Here, ${\vartheta}_1$ can be determined as (up to a constant)
\begin{equation}\label{eq:Upsilon}
{{\vartheta}_1}(x,y,t) = \nabla_x {\vartheta}(x,t) \cdot \bfv(x,y,t),
\quad
\bfv = ({v}_{1},{v}_{2}),
\end{equation}
which yields, substituting \eqref{eq:u_1} and \eqref{eq:Upsilon} into \eqref{eq:201},
\begin{align}\label{eq:301}
&
\int_{\Omega_T}
\int_{\mathcal{Y}}
{\lambda}({y},{p},\vartheta,{r})
\nabla_x {\vartheta}\left[
{\bfI_d}
+
\nabla_y
\bfv(x,y,t)
\right]
\cdot
\nabla_y {\psi_1}(x,y,t)
\,{\rm d}y{\rm d}x{\rm d}t
\nonumber
\\
&
+
\int_{\Omega_T}
c_{w}
{\vartheta}\,
\int_{\mathcal{Y}}
{a}({y},{p},\vartheta,{r})
\nabla_x {p}
\left[
{\bfI_d} + \nabla_y \bfw(x,y,t)
\right]
\cdot
\nabla_y {\psi_1}(x,y,t)
\,{\rm d}y
{\rm d}x {\rm d}t
\nonumber
\\
&
=
0
\end{align}
for all ${\psi_1} \in C_0^{\infty}(\Omega_T;C^{\infty}_{per}({\mathcal{Y}}))$.
Here, $\nabla_y \bfv$ is the matrix
$(\nabla_y \bfv)_{ij} = \partial {v}_{j}/\partial {y}_{i}$.
In view of \eqref{eq:homogenization_02},
the second integral on the left hand side in \eqref{eq:301} vanishes.
Hence,
${v}_{i}$ are obtained
as the unique solutions of the periodic cell problems
\eqref{eq:cell_problem_v}.

Now let ${\psi_1}=0$ in \eqref{weak_form_102}. This leads to
\begin{align*}%\label{}
&
\int_{0}^{T}
\langle
\int_{\mathcal{Y}}
\partial_t
\left[
c_{w} {b}({y},{p},{r})   {\vartheta}
+
\sigma\left( {y} , {r} \right)  {\vartheta}
\right]
{\rm d}y,
\psi
\rangle
{\rm d}t
\nonumber
\\
&
+
\int_{\Omega_T}
\int_{\mathcal{Y}}
{\lambda}({y},{p},\vartheta,{r})
\left(
\nabla_x {\vartheta} + \nabla_y {{\vartheta}_1}(x,y,t)
\right)
\cdot
\nabla_x \psi
\,{\rm d}y{\rm d}x{\rm d}t
\nonumber
\\
&
+
\int_{\Omega_T}
\int_{\mathcal{Y}}
c_{w}
{\vartheta}\,
{a}({y},{p},\vartheta,{r})
\left(
\nabla_x {p} + \nabla_y {p}_{1}(x,y,t)
\right)
\cdot
\nabla_x \psi
\,{\rm d}y{\rm d}x{\rm d}t
\nonumber
\\
&
+
\int_{\partial\Omega_T}
\alpha_{e} \vartheta  \psi
\,
{\rm d}{S} {\rm d}t
+
c_{w}
\int_{\partial\Omega_T}
\beta_{e} \vartheta
(
{p}
-
{p}_{\infty}
) \psi
\,
{\rm d}{S} {\rm d}t
\nonumber
\\
=
&
\;
\int_{\Omega_T}
\int_{\mathcal{Y}}
\alpha_{2}
 \chi_{c}({y})
 f( {p},\vartheta,{r} )
\psi
\;
{\rm d}y {\rm d}x {\rm d}t
+
\int_{\partial\Omega_T}
\alpha_{e} \vartheta_{\infty}  \psi
\,
{\rm d}{S} {\rm d}t
\end{align*}
for all $\psi \in C^{\infty}({\Omega}_T)$
and, using \eqref{eq:u_1} and \eqref{eq:Upsilon}, we can write
\begin{align}\label{eq:weak_form_theta_proof}
&
\int_{0}^{T}
\langle
\partial_t
\left[
c_{w}
\int_{\mathcal{Y}} {b}({y},{p},{r})     {\rm d}y \,   {\vartheta}
+
\int_{\mathcal{Y}} \sigma\left( {y} , {r} \right) {\rm d}y \, {\vartheta}
\right],
\psi
\rangle
{\rm d}t
\nonumber
\\
&
+
\int_{\Omega_T}
\nabla_x {\vartheta}
\left(
\int_{\mathcal{Y}}
{\lambda}({y},{p},\vartheta,{r})
\left(
{\bfI_d} + \nabla_y \bfv(x,y,t)
\right)
{\rm d}y
\right)
\cdot
\nabla_x \psi
\,{\rm d}x{\rm d}t
\nonumber
\\
&
+
\int_{\Omega_T}
c_{w}
{\vartheta}
\,
\nabla_x {p}
\left(
\int_{\mathcal{Y}}
{a}({y},{p},\vartheta,{r})
\left(
{\bfI_d}  + \nabla_y \bfw(x,y,t)
\right)
{\rm d}y
\right)
\cdot
\nabla_x \psi
\,{\rm d}x {\rm d}t
\nonumber
\\
&
+
\int_{\partial\Omega_T}
\alpha_{e} \vartheta  \psi
\,
{\rm d}{S} {\rm d}t
+
c_{w}
\int_{\partial\Omega_T}
\beta_{e} \vartheta
(
{p}
-
{p}_{\infty}
) \psi
\,
{\rm d}{S} {\rm d}t
\nonumber
\\
=
&
\;
\alpha_{2}
\int_{\mathcal{Y}}
 \chi_{c}({y})
{\rm d}y
\int_{\Omega_T}f( {p},\vartheta,{r} )
\psi
\;
{\rm d}x {\rm d}t
+
\int_{\partial\Omega_T}
\alpha_{e} \vartheta_{\infty}  \psi
\,
{\rm d}{S} {\rm d}t
.
\end{align}
We conclude that \eqref{eq:weak_form_theta_proof} represents the weak formulation
of the problem~\eqref{eq:hom_02} and \eqref{eq:hom_04}.
This completes the main result of the paper.

%%%%%%%%%%%%%%%%%%%%%%%%%%%%%%%%%%%%%%%%%%%%%%%%%%%%%%%%%%%%%%%%%%%%%%%%%%%%

\appendix

\section{The existence of the weak solution to \eqref{eq1a}--\eqref{eq1g}}
\label{proof_1}

\subsection{Approximations}\label{sec:approximations}
Applying the method of discretization in time, we divide the interval
$[0,T]$ into $n$ subintervals of lengths  ${h}:= T/n$ (a time step),
replace the time derivatives
by the corresponding difference quotients
and the integral in \eqref{eq:memory_weak_form} by a sum.
In this way,
we approximate the problem \eqref{eq1a}--\eqref{eq1g}
by a semi-implicit time discretization scheme
and re-formulate the problem in a weak sense.

Let us consider
$p^0_{n} := p_{0}$,
$\vartheta^{0}_{n} := \vartheta_{0}$
and
$r^0_{n} := {0}$ a.e. on $\Omega$.
We now define,
in each time step $i=1,\dots,n$,
a threesome
$[{p}^{i}_{n},\vartheta^{i}_{n},{r}^{i}_{n}]$
as a solution of the
following recurrence steady problem:
for a given threesome
$[p^{i-1}_{n},\vartheta^{i-1}_{n},r^{i-1}_{n}]$,
$i=1,2,\dots,n$,
${p}^{i-1}_{n} \in  L^{\infty}(\Omega)$,
$\vartheta^{i-1}_{n} \in  W^{1,2}(\Omega) \cap L^{\infty}(\Omega)$
and
${r}^{i-1}_{n} \in  W^{1,2}(\Omega) \cap L^{\infty}(\Omega)$,
find  $[p^{i}_{n},\vartheta^{i}_{n},r^{i}_{n}]$,
such that
${p}^{i}_{n} \in  {W}^{1,2}(\Omega) \cap L^{\infty}(\Omega)$,
$\vartheta^{i}_{n} \in  {W}^{1,2}(\Omega) \cap L^{\infty}(\Omega)$,
${r}^{i}_{n} \in  {W}^{1,2}(\Omega) \cap L^{\infty}(\Omega)$
and
\begin{align}
\label{approximate_problem_01}
&
\varrho_{w}
\int_{\Omega}
\chi^{\varepsilon}_{c}
\frac{ \phi_{c}({r}^{i}_{n})\mathcal{S}({p}^{i}_{n}) - \phi_{c}({r}^{i-1}_{n})\mathcal{S}({p}^{i-1}_{n}) }{h}
\zeta
{\,{\rm d}x}
+
\varrho_{w}
\int_{\Omega}
\chi^{\varepsilon}_{a}
\phi_{a}
\frac{ \mathcal{S}({p}^{i}_{n}) - \mathcal{S}({p}^{i-1}_{n}) }{h}
\zeta
{\,{\rm d}x}
\nonumber
\\
&
+
\varrho_{w}
\int_{\Omega}
\left[
\chi^{\varepsilon}_{c}
\frac{{k}_{c}({r}^{i-1}_{n})}{\mu({\vartheta}^{i-1}_{n})}
{k}_{R}(\mathcal{S}({p}^{i}_{n}))
+
\chi^{\varepsilon}_{a}
\frac{ {k}_{a} }{\mu({\vartheta}^{i-1}_{n})}
{k}_{R}(\mathcal{S}({p}^{i}_{n}))
\right]
\nabla {p}_{n}^{i}
\cdot\nabla\zeta
{\,{\rm d}x}
\nonumber
\\
&
+
\beta_{e}
\int_{\partial\Omega}
\left(
{p}_{n}^{i}
-
{p}_{\infty}
\right)
\zeta
{\,{\rm d}S}
\nonumber
\\
=
&
\;
\alpha_1
\int_{\Omega}
\chi^{\varepsilon}_{c}
f({{p}_{n}^{i}},\vartheta_{n}^{i-1},{r}_{n}^{i-1})
\zeta
{\,{\rm d}x}
\end{align}
for any $\zeta \in {{W}^{1,2}(\Omega)}$;
\begin{align}
\label{approximate_problem_03}
&
c_{w}
\int_{\Omega}
\frac{
{b}^{\varepsilon}({x},{p}^{i}_{n},{r}^{i}_{n})  \vartheta^{i}_{n}
-
{b}^{\varepsilon}({x},{p}^{i-1}_{n},{r}^{i-1}_{n}) \vartheta_{n}^{i-1} }{h}
\psi
{\,{\rm d}x}
\nonumber
\\
&
+
\int_{\Omega}
\frac{ \sigma^{\varepsilon}(x,r_{n}^{i})\vartheta_{n}^{i} - \sigma^{\varepsilon}(x,r_{n}^{i-1})\vartheta_{n}^{i-1} }{h}  \psi
{\,{\rm d}x}
\nonumber
\\
&
+
\int_{\Omega}
\lambda^{\varepsilon}({x},{{p}_{n}^{i-1}},\vartheta_{n}^{i-1},{r}_{n}^{i-1})
\nabla \vartheta_{n}^{i}
\cdot
\nabla\psi
{\,{\rm d}x}
\nonumber
\\
&
+
c_{w}
\int_{\Omega}
\vartheta_{n}^{i}
{a}^{\varepsilon}({x},{{p}_{n}^{i}},\vartheta_{n}^{i-1},{r}_{n}^{i-1})
\nabla {p}_{n}^{i}
\cdot \nabla\psi
{\,{\rm d}x}
\nonumber
\\
&
+
\alpha_{e}
\int_{\partial\Omega}
\left(
\vartheta_{n}^{i}
-
\vartheta_{\infty}
\right)
\psi
{\,{\rm d}S}
+
c_{w}
\int_{\partial\Omega_T}
\beta_{e} \vartheta_{n}^{i}
(
{p}_{n}^{i}
-
{p}_{\infty}
)
\psi
\,
{\rm d}{S} {\rm d}t
\nonumber
\\
=
&
\;
\alpha_2
\int_{\Omega}
\chi^{\varepsilon}_{c}(x)
f({{p}_{n}^{i}},\vartheta_{n}^{i-1},{r}_{n}^{i-1})
\psi
{\,{\rm d}x}
\end{align}
for any $\psi \in {{W}^{1,2}(\Omega)}$
%\end{problem}
and
\begin{align}
&
r_{n}^{i}
=
{h} \sum_{j=1}^{i}
f({p}_{n}^{j},\vartheta_{n}^{j-1},{r}_{n}^{j-1}),
\quad
i=1,\dots,n,
\label{approximate_problem_04a}
\\
&
r_{n}^{0}(x) = 0.
\label{approximate_problem_04b}
\end{align}

%--------------------------------------------------------------------------------
\begin{theorem}[Existence of the solution to
\eqref{approximate_problem_01}--\eqref{approximate_problem_04b}]
\label{thm:aprox}
Let
${p}^{i-1}_{n} \in  L^{\infty}(\Omega)$,
$\vartheta^{i-1}_{n} \in W^{1,2}(\Omega) \cap L^{\infty}(\Omega)$
and
${r}^{i-1}_{n} \in  W^{1,2}(\Omega) \cap L^{\infty}(\Omega)$
be given and the Assumptions {\rm (i)--(v)} be satisfied.
Then there exists $[{p}^{i}_{n},\vartheta^{i}_{n},{r}^{i}_{n}]$,
such that
${p}^{i}_{n} \in  W^{1,s}(\Omega)$,
$\vartheta^{i}_{n} \in  W^{1,s}(\Omega)$ with some $s>2$,
and
${r}^{i}_{n} \in  W^{1,2}(\Omega) \cap L^{\infty}(\Omega)$,
satisfying
\eqref{approximate_problem_01}--\eqref{approximate_problem_04b}.
\end{theorem}
Note that by Theorem~\ref{thm:aprox} and
the embedding theorems  ${W}^{1,s}(\Omega) \hookrightarrow W^{1,2}(\Omega)$
and $W^{1,s}(\Omega) \hookrightarrow L^{\infty}(\Omega)$
(with some $s>2$ and the two-dimensional domain $\Omega$ with Lipschitz boundary)
we are able to solve \eqref{approximate_problem_01}--\eqref{approximate_problem_04b}
recursively for
$[p^{i}_{n},\vartheta^{i}_{n},r^{i}_{n}]$
by the already known
$[p^{i-1}_{n},\vartheta^{i-1}_{n},r^{i-1}_{n}]$
from the preceding time step,
such that we obtain
\begin{equation}\label{regularity_stationary}
\left\{ \quad
\begin{split}
&   {p}^{i}_{n} \in  W^{1,2}(\Omega) \cap L^{\infty}(\Omega),
\\
&   {\vartheta}^{i}_{n} \in  W^{1,2}(\Omega) \cap L^{\infty}(\Omega),
\\
&   {r}^{i}_{n} \in  W^{1,2}(\Omega) \cap L^{\infty}(\Omega)
\end{split}
\right.
\qquad \qquad \textmd{for all} \; i=1,\dots,n.
\end{equation}
\begin{theorem}\label{thm:max_principle_pressure}
Let ${r}^{i}_{n},{r}^{i-1}_{n},\vartheta^{i-1}_{n},{p}^{i-1}_{n} \in L^{1}(\Omega)$
be given and let
${p}^{i}_{n} \in  {W}^{1,2}(\Omega) \cap {L}^{\infty}(\Omega)$
solve \eqref{approximate_problem_01}.
Then
\begin{equation}\label{est:p_uniform_bound}
0   \geq   {p}_{n}^{i} \geq {p}_{\infty} \; \textmd{ almost everywhere in } \Omega \textmd{ and } \partial\Omega
\textmd{ and for all} \; i=1,2,\dots,n.
\end{equation}
\end{theorem}
\begin{proof}[Proof of Theorem~\ref{thm:max_principle_pressure}]
Let us set (for $i=1,2,\dots,n$)
\begin{equation*}
(S({p}_{n}^{i})-S({p}_{\infty}))_{-}
\equiv
\left\{
\begin{array}{ll}
S({p}_{n}^{i})-S({p}_{\infty}) ,         &  {p}_{n}^{i} < {p}_{\infty} ,
\\
0 ,                   & {p}_{n}^{i} \geq {p}_{\infty}.
\end{array} \right.
\end{equation*}
Setting $\zeta = (S({p}_{n}^{i})-S({p}_{\infty}))_{-}$
as a test function in
\eqref{approximate_problem_01}
we arrive at the estimate
\begin{align}
\label{est:901}
&
\frac{\varrho_{w}}{2h}
\int_{\Omega}
(\chi^{\varepsilon}_{c}\phi_{c}({r}^{i}_{n}) + \chi^{\varepsilon}_{a}\phi_{a})
|(\mathcal{S}({p}_{n}^{i})-\mathcal{S}({p}_{\infty}))_{-}|^2
{\,{\rm d}x}
\nonumber
\\
&
-
\frac{\varrho_{w}}{2h}
\int_{\Omega}
 (\chi^{\varepsilon}_{c}\phi_{c}({r}^{i-1}_{n}) + \chi^{\varepsilon}_{a}\phi_{a})
|(\mathcal{S}({p}_{n}^{i-1})-\mathcal{S}({p}_{\infty}))_{-}|^2
{\,{\rm d}x}
\nonumber
\\
&
+
\int_{\Omega}
a^{\varepsilon}({x},{{p}_{n}^{i}},\vartheta_{n}^{i-1},{r}_{n}^{i-1})
\frac{1}{\mathcal{S}'({p}_{n}^{i})}
|\nabla (\mathcal{S}({p}_{n}^{i}) - \mathcal{S}({p}_{\infty}))_{-}|^2
{\,{\rm d}x}
\nonumber
\\
&
+
\beta_{e}
\int_{\partial\Omega}
\left(
{p}_{n}^{i}
-
{p}_{\infty}
\right)
(\mathcal{S}({p}_{n}^{i})-\mathcal{S}({p}_{\infty}))_{-}
{\,{\rm d}S}
\nonumber
\\
\leq
&
\;
\alpha_1
\int_{\Omega}
\chi^{\varepsilon}_{c}
f({{p}_{n}^{i}}, \vartheta_{n}^{i-1}, {r}_{n}^{i-1})
(\mathcal{S}({p}_{n}^{i})-\mathcal{S}({p}_{\infty}))_{-}
{\,{\rm d}x}
\nonumber
\\
&
-
\frac{\varrho_{w}}{2h}
\int_{\Omega}
\chi^{\varepsilon}_{c}
\left[
\phi_{c}({r}^{i}_{n}) - \phi_{c}({r}^{i-1}_{n})
\right]
(\mathcal{S}({p}_{n}^{i})+\mathcal{S}({p}_{\infty}))
(\mathcal{S}({p}_{n}^{i})-\mathcal{S}({p}_{\infty}))_{-}
{\,{\rm d}x}.
\end{align}
Using the Lipschitz continuity of $\phi$ with respect to $r$  (recall \eqref{lipschitz_phi})
and using \eqref{approximate_problem_04a}
we can write
\begin{align}\label{est:902}
| \phi_{c}({r}^{i}_{n}) - \phi_{c}({r}^{i-1}_{n})  |
&
\leq
C_{\phi}
|r_{n}^{i} - r_{n}^{i-1}|
\nonumber
\\
&
\leq
{h}
C_{\phi}
|f({{p}_{n}^{i}}, \vartheta_{n}^{i-1}, {r}_{n}^{i-1})|.
\end{align}
Upon addition \eqref{est:901} for $i=1,2,\dots,j$ and taking into account
\eqref{est:902},
we can write
\begin{align}
\label{est:901b}
&
\frac{\varrho_{w}}{2h}
\int_{\Omega}
(\chi^{\varepsilon}_{c}\phi_{c}({r}^{j}_{n}) + \chi^{\varepsilon}_{a}\phi_{a})
|(\mathcal{S}({p}_{n}^{j})-\mathcal{S}({p}_{\infty}))_{-}|^2
{\,{\rm d}x}
\nonumber
\\
&
+
\sum_{i=1}^{j}
\int_{\Omega}
a^{\varepsilon}({x},{{p}_{n}^{i}},\vartheta_{n}^{i-1},{r}_{n}^{i-1})
\frac{1}{\mathcal{S}'({p}_{n}^{i})}
|\nabla (\mathcal{S}({p}_{n}^{i}) - \mathcal{S}({p}_{\infty}))_{-}|^2
{\,{\rm d}x}
\nonumber
\\
&
+
\sum_{i=1}^{j}
\int_{\partial\Omega}
\beta_{e}
\left(
{p}_{n}^{i}
-
{p}_{\infty}
\right)
(\mathcal{S}({p}_{n}^{i})-\mathcal{S}({p}_{\infty}))_{-}
{\,{\rm d}S}
\nonumber
\\
\leq
&
\;
\sum_{i=1}^{j}
\int_{\Omega}
(  \varrho_{w} C_{\phi}S_s  + |\alpha_1|)
\chi^{\varepsilon}_{c}
|f({{p}_{n}^{i}}, \vartheta_{n}^{i-1}, {r}_{n}^{i-1})
(\mathcal{S}({p}_{n}^{i}) - \mathcal{S}({p}_{\infty}))_{-}|
{\,{\rm d}x}
\nonumber
\\
&
-
\frac{\varrho_{w}}{2h}
\int_{\Omega}
 (\chi^{\varepsilon}_{c}\phi_{c}({r}^{0}_{n}) + \chi^{\varepsilon}_{a}\phi_{a})
|(\mathcal{S}({p}_{n}^{0})-\mathcal{S}({p}_{\infty}))_{-}|^2
{\,{\rm d}x}
\end{align}
and in view of \eqref{con11a},
\eqref{assum:zero_hydration_f} and \eqref{cond:lower_bound_pressure}
we obtain
\begin{align}
\label{est:901c}
&
\frac{{c}_{1}}{2h}
\int_{\Omega}
|(\mathcal{S}({p}_{n}^{j})-\mathcal{S}({p}_{\infty}))_{-}|^2
{\,{\rm d}x}
\nonumber
\\
&
+
\beta_{e}
{S}_{L}
\sum_{i=1}^{j}
\int_{\partial\Omega}
|(\mathcal{S}({p}_{n}^{j})-\mathcal{S}({p}_{\infty}))_{-}|^2
{\,{\rm d}S}
\nonumber
\\
\leq
&
\;
\sum_{i=1}^{j}
\int_{\Omega}
( \varrho_{w} C_{\phi}S_s + |\alpha_1|)
\chi^{\varepsilon}_{c}
|f({{p}_{n}^{i}}, \vartheta_{n}^{i-1}, {r}_{n}^{i-1})
(\mathcal{S}({p}_{n}^{i}) - \mathcal{S}({p}_{\infty}))_{-}|
{\,{\rm d}x}
\nonumber
\\
&
-
\frac{\varrho_{w}}{2h}
\int_{\Omega}
 (\chi^{\varepsilon}_{c}\phi_{c}({r}^{0}_{n}) + \chi^{\varepsilon}_{a}\phi_{a})
|(\mathcal{S}({p}_{n}^{0})-\mathcal{S}({p}_{\infty}))_{-}|^2
{\,{\rm d}x}
\nonumber
\\
=
&
\;
0.
\end{align}
Now, let us set
\begin{equation*}
(\mathcal{S}({p}_{n}^{i})-\mathcal{S}(0))_{+}
\equiv
\left\{
\begin{array}{ll}
\mathcal{S}({p}_{n}^{i})-\mathcal{S}(0) ,         &  {p}_{n}^{i} > 0 ,
\\
0 ,                   & {p}_{n}^{i} \leq  0,
\end{array} \right.
\qquad
i=1,2,\dots,n.
\end{equation*}
Setting $\zeta = (S({p}_{n}^{i})-S(0))_{+}$
as a test function in
\eqref{approximate_problem_01}
we arrive at the estimate
\begin{align}
\label{est:901k}
&
\frac{\varrho_{w}}{2h}
\int_{\Omega}
(\chi^{\varepsilon}_{c}\phi_{c}({r}^{i}_{n}) + \chi^{\varepsilon}_{a}\phi_{a})
|(\mathcal{S}({p}_{n}^{i})-\mathcal{S}(0))_{+}|^2
{\,{\rm d}x}
\nonumber
\\
&
-
\frac{\varrho_{w}}{2h}
\int_{\Omega}
 (\chi^{\varepsilon}_{c}\phi_{c}({r}^{i-1}_{n}) + \chi^{\varepsilon}_{a}\phi_{a})
|(\mathcal{S}({p}_{n}^{i})-\mathcal{S}(0))_{+}|^2
{\,{\rm d}x}
\nonumber
\\
&
+
\int_{\Omega}
a^{\varepsilon}({x},{{p}_{n}^{i}},\vartheta_{n}^{i-1},{r}_{n}^{i-1})
\frac{1}{\mathcal{S}'({p}_{n}^{i})}
|\nabla (\mathcal{S}({p}_{n}^{i})-\mathcal{S}(0))_{+}|^2
{\,{\rm d}x}
\nonumber
\\
&
+
\beta_{e}
\int_{\partial\Omega}
\left(
{p}_{n}^{i}
-
{p}_{\infty}
\right)
(\mathcal{S}({p}_{n}^{i})-\mathcal{S}(0))_{+}
{\,{\rm d}S}
\nonumber
\\
\leq
&
\;
\alpha_1
\int_{\Omega}
\chi^{\varepsilon}_{c}
f({{p}_{n}^{i}}, \vartheta_{n}^{i-1}, {r}_{n}^{i-1})
(\mathcal{S}({p}_{n}^{i})-\mathcal{S}(0))_{+}
{\,{\rm d}x}
\nonumber
\\
&
-
\frac{\varrho_{w}}{2h}
\int_{\Omega}
\chi^{\varepsilon}_{c}
\left[
\phi_{c}({r}^{i}_{n}) - \phi_{c}({r}^{i-1}_{n})
\right]
(\mathcal{S}({p}_{n}^{i})+\mathcal{S}(0))
(\mathcal{S}({p}_{n}^{i})-\mathcal{S}(0))_{+}
{\,{\rm d}x}.
\end{align}
Upon addition \eqref{est:901k} for $i=1,2,\dots,j$ and taking into account
\eqref{est:902},
we deduce
\begin{align}
\label{est:901bk}
&
\frac{\varrho_{w}}{2h}
\int_{\Omega}
(\chi^{\varepsilon}_{c}\phi_{c}({r}^{j}_{n}) + \chi^{\varepsilon}_{a}\phi_{a})
|(\mathcal{S}({p}_{n}^{i})-\mathcal{S}(0))_{+}|^2
{\,{\rm d}x}
\nonumber
\\
&
+
\sum_{i=1}^{j}
\int_{\Omega}
a^{\varepsilon}({x},{{p}_{n}^{i}},\vartheta_{n}^{i-1},{r}_{n}^{i-1})
\frac{1}{\mathcal{S}'({p}_{n}^{i})}
|\nabla (\mathcal{S}({p}_{n}^{i})-\mathcal{S}(0))_{+}|^2
{\,{\rm d}x}
\nonumber
\\
&
+
\sum_{i=1}^{j}
\int_{\partial\Omega}
\beta_{e}
\left(
{p}_{n}^{i}
-
{p}_{\infty}
\right)
(\mathcal{S}({p}_{n}^{i})-\mathcal{S}(0))_{+}
{\,{\rm d}S}
\nonumber
\\
\leq
&
\;
\sum_{i=1}^{j}
\int_{\Omega}
(  \varrho_{w} C_{\phi}S_s  + \alpha_1)
\chi^{\varepsilon}_{c}
f({{p}_{n}^{i}}, \vartheta_{n}^{i-1}, {r}_{n}^{i-1})
(\mathcal{S}({p}_{n}^{i})-\mathcal{S}(0))_{+}
{\,{\rm d}x}
\nonumber
\\
&
-
\frac{\varrho_{w}}{2h}
\int_{\Omega}
 (\chi^{\varepsilon}_{c}\phi_{c}({r}^{0}_{n}) + \chi^{\varepsilon}_{a}\phi_{a})
|(\mathcal{S}({p}_{n}^{i})-\mathcal{S}(0))_{+}|^2
{\,{\rm d}x}
\end{align}
and in view of \eqref{con11a}, \eqref{lipschitz_phi},
\eqref{assum:bound_f} and \eqref{material_constants_condition_drying} we arrive at the estimate
\begin{align}
\label{est:901ck}
&
\frac{{c}_{1}}{2h}
\int_{\Omega}
|(\mathcal{S}({p}_{n}^{i})-\mathcal{S}(0))_{+}|^2
{\,{\rm d}x}
\nonumber
\\
&
+
\beta_{e}
{S}_{L}
\sum_{i=1}^{j}
\int_{\partial\Omega}
|(\mathcal{S}({p}_{n}^{i})-\mathcal{S}(0))_{+}|^2
{\,{\rm d}S}
\nonumber
\\
\leq
&
\;
\sum_{i=1}^{j}
\int_{\Omega}
( \varrho_{w} C_{\phi}S_s + \alpha_1)
\chi^{\varepsilon}_{c}
f({{p}_{n}^{i}}, \vartheta_{n}^{i-1}, {r}_{n}^{i-1})
(\mathcal{S}({p}_{n}^{i})-\mathcal{S}(0))_{+}
{\,{\rm d}x}
\nonumber
\\
&
-
\frac{\varrho_{w}}{2h}
\int_{\Omega}
 (\chi^{\varepsilon}_{c}\phi_{c}({r}^{0}_{n}) + \chi^{\varepsilon}_{a}\phi_{a})
|(\mathcal{S}({p}_{n}^{i})-\mathcal{S}(0))_{+}|^2
{\,{\rm d}x}
+
\sum_{i=1}^{j}
\int_{\partial\Omega}
\beta_{e}
{p}_{\infty}
(\mathcal{S}({p}_{n}^{i})-\mathcal{S}(0))_{+}
{\,{\rm d}S}
\nonumber
\\
\leq
&
\;
0.
\end{align}
From \eqref{est:901c} and \eqref{est:901ck} we have \eqref{est:p_uniform_bound}.
\hfill  $\square$
\end{proof}

Now we are ready to prove Theorem~\ref{thm:aprox}.
\begin{proof}[Proof of Theorem~\ref{thm:aprox}]
We begin by proving the existence of ${p}_{n}^{i}  \in {{W}^{1,2}(\Omega)}$,
being the solution to problem \eqref{approximate_problem_01}.
Due to Theorem~\ref{thm:max_principle_pressure}, we may consider the truncated
function $\widetilde{k}_r$ defined by
\begin{equation*}
\widetilde{k}_r (\xi)
\equiv
\left\{
\begin{array}{ll}
{k}_{R} (\mathcal{S}(\xi)) ,         &  \xi > {p}_{\infty} ,
\\
{k}_{R} (\mathcal{S}({p}_{\infty})) ,                   & \xi \leq {p}_{\infty},
\end{array} \right.
\end{equation*}
where ${p}_{\infty}$ is taken from \eqref{est:p_uniform_bound}.
Note that \eqref{est:p_uniform_bound} also remains valid if
in \eqref{approximate_problem_01}
we replace ${k}_{R} \circ \mathcal{S}$ by $\widetilde{k}_r$.
Recall that $k_{R}$ is positive and strictly increasing on $[0,S_s]$
and $\mathcal{S}$ is positive and strictly increasing on $\mathbb{R}$. Hence
$\widetilde{k}_r$ is the increasing function such that
\begin{equation}\label{bound_rel_perm}
0 < K_0 \leq \widetilde{k}_r(\xi) \leq K_1
\qquad
\forall \xi \in \mathbb{R}
\end{equation}
with appropriate chosen constants $K_0$ and $K_1$,
say
$K_0 = {k}_{R}  (\mathcal{S}({p}_{\infty}))$ and $K_1=  {k}_{R}  (S_s)$.
Hence, problem \eqref{approximate_problem_01} takes the form
\begin{align}
\label{proof:approximate_problem_01}
&
\varrho_{w}
\int_{\Omega}
\left[
\chi^{\varepsilon}_{c}
\frac{{k}_{c}({r}^{i-1}_{n})}{\mu({\vartheta}^{i-1}_{n})}
+
\chi^{\varepsilon}_{a}
\frac{ {k}_{a} }{\mu({\vartheta}^{i-1}_{n})}
\right]
\widetilde{k}_{r}(\mathcal{S}({p}^{i}_{n}))
\nabla {p}_{n}^{i}
\cdot\nabla\zeta
{\,{\rm d}x}
\nonumber
\\
&
+
\beta_{e}
\int_{\partial\Omega}
\left(
{p}_{n}^{i}
-
{p}_{\infty}
\right)
\zeta
{\,{\rm d}S}
\nonumber
\\
&
+
\frac{\varrho_{w}}{h}
\int_{\Omega}
\left(
\chi^{\varepsilon}_{c}\phi_{c}({r}^{i}_{n})
+
\chi^{\varepsilon}_{a}
\phi_{a}
\right)
\mathcal{S}({p}^{i}_{n})
\zeta
{\,{\rm d}x}
\nonumber
\\
&
-
\alpha_1
\int_{\Omega}
\chi^{\varepsilon}_{c}
f({{p}_{n}^{i}},\vartheta_{n}^{i-1},{r}_{n}^{i-1})
\zeta
{\,{\rm d}x}
\nonumber
\\
=
&
\;
\frac{\varrho_{w}}{h}
\int_{\Omega}
\left(
\chi^{\varepsilon}_{c}\phi_{c}({r}^{i-1}_{n})
+
\chi^{\varepsilon}_{a}
\phi_{a}
\right)
\mathcal{S}({p}^{i-1}_{n})
\zeta
{\,{\rm d}x}
\end{align}
for any $\zeta \in {{W}^{1,2}(\Omega)}$.
Note that
the unknown $r_{n}^{i}$ in the second line of \eqref{proof:approximate_problem_01}
can be easily eliminated using
the equation \eqref{approximate_problem_04a}, which can be rewritten as
\begin{equation}\label{eq:mamory_mod_01}
r_{n}^{i}
=
r_{n}^{i-1}
+
{h}
f({p}_{n}^{i},\vartheta_{n}^{i-1},{r}_{n}^{i-1}).
\end{equation}
We now
define the so called Kirchhoff transformation, which employs the primitive function
$\kappa: \mathbb{R}\rightarrow\mathbb{R}$, defined by
\begin{displaymath}%
\kappa(\xi) = \int\limits_{0}^{\xi} \widetilde{k}_r(s)
{\rm d}s.
\end{displaymath}
It is worth noting that \eqref{bound_rel_perm} implies $\kappa$ to
be continuous and increasing, and one-to-one with $\kappa^{-1}$
Lipschitz continuous.
Hence,
with the notation $u(x) = \kappa({p}_{n}^{i}(x))$,
 problem \eqref{proof:approximate_problem_01}
can be rewritten in terms of a new variable $u$ as
\begin{equation}
\label{proof_exist:approximate_problem_01}
\int_{\Omega}
A({x})
\nabla {u}
\cdot\nabla\zeta
{\,{\rm d}x}
+
\beta_{e}
\int_{\partial\Omega}
\kappa^{-1}(u)
\zeta
{\,{\rm d}S}
+
\int_{\Omega}
B({x},{u})
\zeta
{\,{\rm d}x}
=
\int_{\Omega}
g(x)
\zeta
{\,{\rm d}x}
+
\beta_{e}
\int_{\partial\Omega}
{p}_{\infty}
\zeta
{\,{\rm d}S}
\end{equation}
for any $\zeta \in {{W}^{1,2}(\Omega)}$,
where we denote briefly
\begin{align*}
A({x})
=
&
\varrho_{w}
\left[
\chi^{\varepsilon}_{c}
\frac{{k}_{c}({r}^{i-1}_{n})}{\mu({\vartheta}^{i-1}_{n})}
+
\chi^{\varepsilon}_{a}
\frac{ {k}_{a} }{\mu({\vartheta}^{i-1}_{n})}
\right],
\nonumber
\\
B({x},{u})
=
&
\frac{\varrho_{w}}{h}
\left(
\chi^{\varepsilon}_{c}\phi_{c}(r_{n}^{i-1}(x)
+
{h}
f(\kappa^{-1}(u),\vartheta_{n}^{i-1}(x),{r}_{n}^{i-1}(x)))
+
\chi^{\varepsilon}_{a}
\phi_{a}
\right)
\mathcal{S}(\kappa^{-1}(u)),
\nonumber
\\
&
-
\alpha_1
\chi^{\varepsilon}_{c}
f(\kappa^{-1}(u),\vartheta_{n}^{i-1},{r}_{n}^{i-1})
\nonumber
\\
g({x})
=
&
\frac{\varrho_{w}}{h}
\left(
\chi^{\varepsilon}_{c}\phi_{c}({r}^{i-1}_{n})
+
\chi^{\varepsilon}_{a}
\phi_{a}
\right)
\mathcal{S}({p}^{i-1}_{n}).
\end{align*}
Note that $g \in L^{\infty}(\Omega)$ and
\begin{align*}
& 0<A_1<A({\cdot})<A_2<+\infty \quad ( A_1, A_2 = {\rm const})
&& \textmd{a.e. in } \Omega,
\nonumber
\\
&
|B({\cdot},{\xi})| \leq C   && \forall \xi \in \mathbb{R} \textmd{ and a.e. in } \Omega.
\nonumber
\end{align*}
The existence of $u \in {W}^{1,2}(\Omega)$,
the solution of problem \eqref{proof_exist:approximate_problem_01},
follows from \cite[Chapter~2]{Roubicek2005}.
With ${u} \in {W}^{1,2}(\Omega)$ in hand,
we now assume the problem
\begin{equation*}
%\label{proof_exist:approximate_problem_05}
\int_{\Omega}
A({x})
\nabla {u}
\cdot\nabla\zeta
{\,{\rm d}x}
=
\langle \mu , \zeta \rangle_{W^{1,2}(\Omega)',W^{1,2}(\Omega)}
\end{equation*}
for any $\zeta \in {{W}^{1,2}(\Omega)}$,
where the functional $\mu$ is defined by the equation
\begin{displaymath}
\langle \mu , \zeta \rangle_{W^{1,2}(\Omega)',W^{1,2}(\Omega)}
=
-
\beta_{e}
\int_{\partial\Omega}
\kappa^{-1}(u)
\zeta
{\,{\rm d}S}
+
\int_{\Omega}
g(x)
\zeta
{\,{\rm d}x}
-
\int_{\Omega}
B({x},{u})
\zeta
{\,{\rm d}x}
+
\beta_{e}
\int_{\partial\Omega}
{p}_{\infty}
\zeta
{\,{\rm d}S}
\end{displaymath}
for all $\zeta \in {{W}^{1,2}(\Omega)}$.
It is not difficult to show that  $\mu \in W^{1,q'}(\Omega)'$, $q'=q/(q-1)$, with some $q>2$.
The same arguments as in the proof of \cite[Theorem~3]{gallouet} yields  $u \in {W}^{1,q}(\Omega)$.
This (together with the embedding theorem) gives the regularity ${u} \in L^{\infty}(\Omega)$.
Note that, in view of \eqref{bound_rel_perm},
the Kirchhoff transformation preserves $L^{\infty}$ space
for the problem. We now set ${{p}_{n}^{i}}(x) := \kappa^{-1} (u(x))$ a.e. in $\Omega$
to get the representation
\begin{equation*}%\label{kirch_reg1}
\nabla {{p}_{n}^{i}}
=
\frac{1}{\widetilde{k}_r( \kappa^{-1}({u}))}\nabla {u},
\quad \textmd{ i.e. }
\quad
\widetilde{k}_r(  {{p}_{n}^{i}}  )\nabla {{p}_{n}^{i}}
=
\nabla {u}
\end{equation*}
and hence
\begin{equation*}%\label{kirch_reg2}
{{p}_{n}^{i}}  \in {{W}^{1,2}(\Omega)} \cap L^{\infty}(\Omega)
\quad
{ \rm iff }
\quad
u \in {{W}^{1,2}(\Omega)} \cap L^{\infty}(\Omega).
\end{equation*}
We now conclude that ${{p}_{n}^{i}}$ solves \eqref{approximate_problem_01}.

With ${{p}_{n}^{i}}  \in {{W}^{1,2}(\Omega)} \cap L^{\infty}(\Omega)$ in hand,
we rewrite the equation \eqref{approximate_problem_01} in the form
(transferring the lower-order terms to the right hand side)
\begin{align*}
%\label{approximate_problem_01}
&
\varrho_{w}
\int_{\Omega}
\left[
\chi^{\varepsilon}_{c}
\frac{{k}_{c}({r}^{i-1}_{n})}{\mu({\vartheta}^{i-1}_{n})}
{k}_{R}(\mathcal{S}({p}^{i}_{n}))
+
\chi^{\varepsilon}_{a}
\frac{ {k}_{a} }{\mu({\vartheta}^{i-1}_{n})}
{k}_{R}(\mathcal{S}({p}^{i}_{n}))
\right]
\nabla {p}_{n}^{i}
\cdot\nabla\zeta
{\,{\rm d}x}
\nonumber
\\
&
+
\beta_{e}
\int_{\partial\Omega}
\left(
{p}_{n}^{i}
-
{p}_{\infty}
\right)
\zeta
{\,{\rm d}S}
\nonumber
\\
=
&
\;
\alpha_1
\int_{\Omega}
\chi^{\varepsilon}_{c}
f({{p}_{n}^{i}},\vartheta_{n}^{i-1},{r}_{n}^{i-1})
\zeta
{\,{\rm d}x}
\nonumber
\\
&
-
\varrho_{w}
\int_{\Omega}
\chi^{\varepsilon}_{c}
\frac{ \phi_{c}({r}^{i}_{n})\mathcal{S}({p}^{i}_{n}) - \phi_{c}({r}^{i-1}_{n})\mathcal{S}({p}^{i-1}_{n}) }{h}
\zeta
{\,{\rm d}x}
-
\varrho_{w}
\int_{\Omega}
\chi^{\varepsilon}_{a}
\phi_{a}
\frac{ \mathcal{S}({p}^{i}_{n}) - \mathcal{S}({p}^{i-1}_{n}) }{h}
\zeta
{\,{\rm d}x}
\end{align*}
for any $\zeta \in {{W}^{1,2}(\Omega)}$.

In view of Assumptions (i), (iii) and (iv),
all integrals on the right hand side
make sense for any $\zeta \in {W}^{1,{q}'}(\Omega)$, ${q}'={q}/({q}-1)$ with some ${q}>2$.
Now we are able to apply \cite[Theorem~3]{gallouet} to obtain
${p}^{i}_{n} \in  {W}^{1,s}(\Omega)$ with some $s>2$.
Now with
$\vartheta^{i-1}_{n} \in  W^{1,2}(\Omega)$,
${r}^{i-1}_{n} \in  W^{1,2}(\Omega)$
and
${p}^{i}_{n} \in  W^{1,s}(\Omega)$
(with some $s>2$) in hand,
one obtains ${r}^{i}_{n}$
directly from \eqref{eq:mamory_mod_01}.
Since $f$ is supposed to be Lipschitz continuous,
we easily deduce ${r}^{i}_{n} \in W^{1,2}(\Omega)$ (c.f. \cite[Proposition~1.28]{Roubicek2005}).
Moreover, from \eqref{assum:bound_f} we have ${r}^{i}_{n} \in L^{\infty}(\Omega)$.
The existence of $\vartheta^{i}_{n} \in  {W}^{1,2}(\Omega)$,
the solution to problem \eqref{approximate_problem_03},
can be proven in the same way as \cite[Theorem~6.5]{BenesKrupicka2016}.
In particular, with ${p}^{i}_{n} \in  W^{1,s}(\Omega)$, $s>2$,
and $r_{n}^{i} \in L^{\infty}(\Omega)$, given by \eqref{approximate_problem_04a},
in hand, \eqref{approximate_problem_03}
represents the semilinear equation which can be solved by the approach in \cite[Chapter~2.4]{Roubicek2005}.
Analysis similar to the above yields
$\vartheta^{i}_{n} \in  {W}^{1,s}(\Omega)$ with some ${s}>2$.
By embedding theorem we have $\vartheta^{i}_{n} \in {L}^{\infty}(\Omega)$.

 \hfill $\square$
\end{proof}

%%%%%%%%%%%%%%%%%%%%%%%%%%%%%%%%%%%%%%%%%%%%%%%%%%%%%%%%%%%%%%%%%%%%%%%%%%%%%%%%%%%%%%%
\subsection{Temporal interpolants and uniform estimates}
By means of the sequences
${p}^{i}_{n},\vartheta^{i}_{n},{r}^{i}_{n}$
constructed
in Section~\ref{sec:approximations}, we define
the piecewise constant interpolants
$\bar{\varphi}_{n}(t) = \varphi^{i}_{n}$ for $t \in ((i - 1){h}, i{h}]$
and, in addition, we extend $\bar{\varphi}_{n}$ for $t\leq 0$ by
$
\bar{\varphi}_{n}(t) = \varphi_{0}$ for $t \in (-{h}, 0]$.
Here, $\varphi^{i}_{n}$ stands for ${p}^{i}_{n},\vartheta^{i}_{n}$ or ${r}^{i}_{n}$.

For a function $\varphi$ we often use the simplified notation
$\varphi := \varphi(t)$, $\varphi_{h}(t) := \varphi(t-{h})$,
$\partial_t^{-{h}}\varphi(t) := \frac{\varphi(t) - \varphi(t-{h})}{h}$,
$\partial_t^{h}\varphi(t) := \frac{\varphi(t+{h}) - \varphi(t)}{h}$.
Then,
following \eqref{approximate_problem_01}--\eqref{approximate_problem_03},
the piecewise constant time interpolants
$\bar{p}_{n} \in L^{\infty}(0,T;{W}^{1,s}(\Omega))$
and
$\bar{\vartheta}_{n} \in L^{\infty}(0,T;{W}^{1,s}(\Omega))$
(with some $s>2$)
satisfy the equations
\begin{align}
\label{eq:1001}
&
\varrho_{w}
\int_{\Omega}
\partial_t^{-{h}}
[
\left(
\chi^{\varepsilon}_{c} \phi_{c}(\bar{r}_{n}(t))+\chi^{\varepsilon}_{a} \phi_{a}
 \right)\mathcal{S}(\bar{p}_{n}(t))
]
\zeta
{\,{\rm d}x}
\nonumber
\\
&
+
\varrho_{w}
\int_{\Omega}
\left[
\chi^{\varepsilon}_{c}
\frac{{k}_{c}(\bar{r}_{n}(t-{h}))}{\mu(\bar{\vartheta}_{n}(t-{h}))}
+
\chi^{\varepsilon}_{a}
\frac{ {k}_{a} }{\mu(\bar{\vartheta}_{n}(t-{h}))}
\right]
{k}_{R}(\mathcal{S}(\bar{p}_{n}(t)))
\nabla \bar{p}_{n}(t)
\cdot\nabla\zeta
{\,{\rm d}x}
\nonumber
\\
&
+
\beta_{e}
\int_{\partial\Omega}
\left(
\bar{p}_{n}(t)
-
{p}_{\infty}
\right)
\zeta
{\,{\rm d}S}
\nonumber
\\
=
&
\int_{\Omega}
\alpha_1
\chi^{\varepsilon}_{c}
f(\bar{p}_{n}(t),\bar{\vartheta}_{n}(t-{h})),\bar{r}_{n}(t-{h})))
\zeta
{\,{\rm d}x}
\end{align}
for any $\zeta \in {{W}^{1,2}(\Omega)}$
and
\begin{align}\label{eq:1003}
&
c_{w} \varrho_{w}
\int_{\Omega}
\partial_t^{-{h}}
\left[
\left(
\chi^{\varepsilon}_{c} \phi_{c}(\bar{r}_{n}(t))+\chi^{\varepsilon}_{a} \phi_{a}
\right)
\mathcal{S}(\bar{p}_{n}(t))
\bar{\vartheta}_{n}(t)
\right]
\psi
{\,{\rm d}x}
\nonumber
\\
&
+
\int_{\Omega}
\partial_t^{-{h}}
\left[
\sigma^{\varepsilon}({x},\bar{r}_{n}(t)) \bar{\vartheta}_{n}(t)
\right] \psi
{\,{\rm d}x}
\nonumber
\\
&
+
\int_{\Omega}
\lambda^{\varepsilon}(x,\bar{p}_{n}(t-{h}),\bar{\vartheta}_{n}(t-{h}),\bar{r}_{n}(t-{h})) \nabla \bar{\vartheta}_{n}(t) \cdot \nabla\psi
{\,{\rm d}x}
\nonumber
\\
&
+
c_{w}
\int_{\Omega}
\bar{\vartheta}_{n}(t)
\left[
\chi^{\varepsilon}_{c}
\frac{{k}_{c}(\bar{r}_{n}(t-{h}))}{\mu(\bar{\vartheta}_{n}(t-{h}))}
+
\chi^{\varepsilon}_{a}
\frac{ {k}_{a} }{\mu(\bar{\vartheta}_{n}(t-{h}))}
\right]
{k}_{R}(\mathcal{S}(\bar{p}_{n}(t)))
\nabla \bar{p}_{n}(t)
\cdot \nabla\psi
{\,{\rm d}x}
\nonumber
\\
&
+
\alpha_{e}
\int_{\partial\Omega}
\left(
\bar{\vartheta}_{n}(t)
-
\vartheta_{\infty}
\right)
\psi
{\,{\rm d}S}
+
c_{w}
\int_{\partial\Omega}
\beta_{e}
\bar{\vartheta}_{n}(t)
(
\bar{p}_{n}(t)
-
{p}_{\infty}
)
\psi
\,
{\rm d}{S}
\nonumber
\\
=
&
\int_{\Omega}
\alpha_2
\chi^{\varepsilon}_{c}(x)
f(\bar{p}_{n}(t),\bar{\vartheta}_{n}(t-{h})),\bar{r}_{n}(t-{h})))
\psi
{\,{\rm d}x}
\end{align}
for any $\psi \in {{W}^{1,2}(\Omega)}$.
Finally, from \eqref{approximate_problem_04a} and \eqref{approximate_problem_04b} we have
\begin{equation}\label{eq:1004a}
\bar{r}_{n}(t)
=
\int_{0}^{t}
f(\bar{p}_{n}(s),\bar{\vartheta}_{n}(s-{h})),\bar{r}_{n}(s-{h})))
{\,{\rm d}s}
\end{equation}
for all $t \in [0,T]$.
To be able to say something about the limits of the sequences
$\left\{ \bar{p}_{n} \right\}$,
$\left\{ \bar{\vartheta}_{n} \right\}$ %,
and
$\left\{ \bar{r}_{n} \right\}$,
we now present some apriori estimates for solutions of
the problem \eqref{eq:1001}--\eqref{eq:1004a}.

Analysis similar to that in \cite[Sections 4.2 and 4.3]{BenesPazanin2018} shows
that
($\Theta$ is defined by \eqref{est:press_02})
\begin{align}
\sup_{0 \leq t \leq T}
\int_{\Omega}
\Theta(\bar{p}_{n}(t))
{{\rm d}x}
+
\int_0^T \|\bar{p}_{n}(t)\|^2_{{W}^{1,2}(\Omega)} {\rm d}t
&\leq C,
\label{est:801}
\\
%\int_0^T \|\bar{c}_{n}(t)\|^2_{{W}^{1,2}(\Omega)} {\rm d}t
%&\leq C,
%\label{est:802}
%\\
\int_0^T \|\bar{\vartheta}_{n}(t)\|^2_{{W}^{1,2}(\Omega)} {\rm d}t
&\leq C,
\label{est:803}
\\
\int_0^T \|\bar{r}_{n}(t)\|^2_{{W}^{1,2}(\Omega)} {\rm d}t
&\leq C,
\label{est:803b}
\\
\|\bar{r}_{n}\|_{L^{\infty}({\Omega_T})}
&\leq C.
\label{est:805}
\end{align}
Let $\ell$ be an odd integer.
Using $\zeta = c_{w} [\ell/(\ell+1)] (\bar{\vartheta}_{n})^{\ell+1}$
as a test function in \eqref{eq:1001}
and
$\psi = (\bar{\vartheta}_{n})^{\ell}$ in \eqref{eq:1003}
and combining both equations (subtracting \eqref{eq:1001}
from \eqref{eq:1003}) we have (with the notation \eqref{notation_form_b})
\begin{align}\label{app:eq_bound_theta_00}
&
\int_{\Omega}
[\bar{\vartheta}_{n}(t)]^{\ell+1}
\left[
c_{w}  b\left({x/\varepsilon},\bar{p}_{n}(t),\bar{r}_{n}(t)\right)
+
\sigma({x/\varepsilon},\bar{r}_{n}(t))
\right]
{\rm d}x
\nonumber
\\
&
+
\int_{\Omega_t}
\ell (\bar{\vartheta}_{n})^{\ell-1}
{\lambda}({x/\varepsilon},\bar{p}_{n},\bar{\vartheta}_{n},\bar{r}_{n})
|\nabla\bar{\vartheta}_{n}|^2
{\rm d}x{\rm d}s
\nonumber
\\
&
+
\int_{\partial\Omega_{t}}
\alpha_{e} (\bar{\vartheta}_{n})^{\ell+1}
\,
{\rm d}{S} {\rm d}t
+
c_{w}
\frac{1}{\ell+1}
\int_{\partial\Omega_{t}}
\beta_{e}
 (\bar{\vartheta}_{n})^{\ell+1}
(
\bar{p}_{n}
-
{p}_{\infty}
)
\,
{\rm d}{S} {\rm d}t
\nonumber
\\
\leq
&
\;
\int_{\Omega}
{\vartheta}_0^{\ell+1}
\left[
c_{w} b\left({x/\varepsilon},{p}_{0},{0}\right)
+
\sigma({x/\varepsilon},{0})
\right]
{\rm d}{x}
+
\int_{\partial\Omega_{t}}
\alpha_{e} \vartheta_{\infty} \;  (\bar{\vartheta}_{n})^{\ell}
\,
{\rm d}{S} {\rm d}t
\nonumber
\\
&
+
\int_{\Omega_t}
\left(
\alpha_{2} (\bar{\vartheta}_{n})^{\ell}
-
\alpha_{1} c_{w} [\ell/(\ell+1)] (\bar{\vartheta}_{n})^{\ell+1}
\right)
\chi_{c}({x/\varepsilon})
f(\bar{p}_{n},\bar{\vartheta}_{n},\bar{r}_{n})
{\rm d}x{\rm d}s.
\end{align}
Applying the Young's inequality we have
\begin{equation}
\int_{\partial\Omega_{t}}
\alpha_{e} \vartheta_{\infty} \;  (\bar{\vartheta}_{n})^{\ell}
\,
{\rm d}{S} {\rm d}t
\leq
\alpha_{e}
\frac{1}{\ell+1}
\int_{\partial\Omega_{t}}
(\vartheta_{\infty})^{\ell+1}
\,
{\rm d}{S} {\rm d}t
+
\alpha_{e}
\frac{\ell}{\ell+1}
\int_{\partial\Omega_{t}}
 (\bar{\vartheta}_{n})^{\ell+1}
\,
{\rm d}{S} {\rm d}t.
\end{equation}
Further,
using \eqref{con11a}, \eqref{bound_phi}, \eqref{assum:bound_f} and  \eqref{est:p_uniform_bound},
the inequality \eqref{app:eq_bound_theta_00} can be simplified (recall that $\ell$ is the odd integer)
\begin{align}\label{app:eq_bound_theta_00b}
&
c_1
\int_{\Omega}
[\bar{\vartheta}_{n}(t)]^{\ell+1}
{\rm d}x
\nonumber
\\
\leq
&
\;
\int_{\Omega}
[\bar{\vartheta}_{n}(t)]^{\ell+1}
\left[
c_{w}  b\left({x/\varepsilon},\bar{p}_{n}(t),\bar{r}_{n}(t)\right)
+
\sigma({x/\varepsilon},\bar{r}_{n}(t))
\right]
{\rm d}x
\nonumber
\\
&
+
\int_{\Omega_t}
\ell (\bar{\vartheta}_{n})^{\ell-1}
{\lambda}({x/\varepsilon},\bar{p}_{n},\bar{\vartheta}_{n},\bar{r}_{n})
|\nabla\bar{\vartheta}_{n}|^2
{\rm d}x{\rm d}s
\nonumber
\\
&
+
\frac{1}{\ell+1}
\int_{\partial\Omega_{t}}
\alpha_{e} (\bar{\vartheta}_{n})^{\ell+1}
\,
{\rm d}{S} {\rm d}t
\nonumber
\\
\leq
&
\;
\int_{\Omega}
{\vartheta}_0^{\ell+1}
\left[
c_{w} b\left({x/\varepsilon},{p}_{0},{0}\right)
+
\sigma({x/\varepsilon},{0})
\right]
{\rm d}{x}
+
\alpha_{e}
\frac{1}{\ell+1}
\int_{\partial\Omega_{t}}
(\vartheta_{\infty})^{\ell+1}
\,
{\rm d}{S} {\rm d}t
\nonumber
\\
&
+
\int_{\Omega_t}
\left[
\alpha_{2}
\left(
\frac{1}{1+\ell}
+
\frac{\ell}{1+\ell}
(\bar{\vartheta}_{n})^{\ell+1}
\right)
-
\alpha_{1} c_{w} [\ell/(\ell+1)] (\bar{\vartheta}_{n})^{\ell+1}
\right]
\chi_{c}({x/\varepsilon})
f(\bar{p}_{n},\bar{\vartheta}_{n},\bar{r}_{n})
{\rm d}x{\rm d}s
\nonumber
\\
\leq
&
\;
c_2
\|
{\vartheta}_0
\|_{L^{\ell+1}(\Omega)}^{\ell+1}
+
c_3
\frac{[\vartheta_{\infty}]^{\ell+1}}{\ell+1}
+
c_4
\frac{1}{\ell+1}
+
c_5
\int_{0}^{t}
\int_{\Omega}
[\bar{\vartheta}_{n}]^{\ell+1}
{\rm d}x{\rm d}s.
\end{align}
Finally,  the application of
Gronwall's inequality \cite[Chapter 1.6]{Roubicek2005}) yields
\begin{eqnarray}
\|  \bar{\vartheta}_{n}  \|_{L^{\infty}(0,T;L^{\ell+1}(\Omega))} \leq C,
\label{est:unform_bound_theta_01}
\\
\|  \bar{\vartheta}_{n}  \|_{L^{\ell+1}(0,T;L^{\ell+1}(\partial\Omega))} \leq C,
\label{est:unform_bound_theta_boundary_01}
\end{eqnarray}
where the constant $C$ is independent of $\ell$, $\tau$ and $\varepsilon$.
Now, let $\ell \rightarrow +\infty$ in \eqref{est:unform_bound_theta_01}, we get
\begin{equation}\label{est:unform_bound_theta_02}
\|  \bar{\vartheta}_{n}  \|_{L^{\infty}({\Omega_T})} \leq C,
\end{equation}
where the constant $C$ is independent of $\tau$ and $\varepsilon$.
Moreover, \eqref{est:p_uniform_bound} further implies
\begin{eqnarray}
\|  \bar{p}_{n}  \|_{L^{\infty}({\Omega_T})} &\leq& C,
\label{est:unform_bound_pressure_02a}
\\
\|  \bar{p}_{n}  \|_{L^{\infty}({\partial\Omega_T})} &\leq& C.
\label{est:unform_bound_pressure_02b}
\end{eqnarray}
Finally,
the following estimates can be obtained
in much the same way as \cite[eqs (109), (111) and (112)]{BenesPazanin2018}
\begin{align}
&
\int_0^{T-k{h}}
\left[
{S}(\bar{p}_{n}(t+k{h}))
-
{S}(\bar{p}_{n}(t))
\right]
\left(
\bar{p}_{n}(t+k{h})
-
\bar{p}_{n}(t)
\right)
{\rm d}t
\leq C k {h},
\label{est:806}
\\
&
\int_0^{T-k{h}}
|\bar{\vartheta}_{n}(t+k{h}) - \bar{\vartheta}_{n}(t)|^2
{\rm d}t
\leq {C} k {h},
\label{est:808}
\\
&
\int_0^{T-k{h}}
|\bar{r}_{n}(t+k{h}) - \bar{r}_{n}(t)|^2
{\rm d}t
\leq {C} k {h}.
\label{est:809}
\end{align}
Finally,
from \eqref{eq:1001} and \eqref{eq:1003}, using
\eqref{est:801},
\eqref{est:803},
\eqref{est:unform_bound_theta_02},
\eqref{con11a}--\eqref{con12d}, \eqref{bound_phi} and \eqref{assum:bound_f},
we get
\begin{equation}\label{}
\|
\varrho_{w}
\partial_t^{-{h}}
[
\left( \chi^{\varepsilon}_{c} \phi_{c}(\bar{r}_{n}(t))+\chi^{\varepsilon}_{a} \phi_{a}  \right)
\mathcal{S}(\bar{p}_{n}(t))
]
\|_{L^{2}(0,T;{{{{W}^{1,2}(\Omega)}'}})}  \leq C.
\end{equation}
and
\begin{equation}\label{x}
\|
\partial_t^{-{h}}
\left[
c_{w} \varrho_{w}
\left(
\chi^{\varepsilon}_{c} \phi_{c}(\bar{r}_{n}(t))+\chi^{\varepsilon}_{a} \phi_{a}
\right)
\mathcal{S}(\bar{p}_{n}(t))
\bar{\vartheta}_{n}(t)
+
\sigma^{\varepsilon}({x},\bar{r}_{n}(t)) \bar{\vartheta}_{n}(t)
\right]
\|_{L^{2}(0,T;{{{{W}^{1,2}(\Omega)}'}})}  \leq C.
\end{equation}

%------------------------------------------------------------------------------
\subsection{Passage to the limit}
\label{subsec:limit}
The a-priori estimates
\eqref{est:801}--\eqref{est:805},
\eqref{est:unform_bound_theta_02} and \eqref{est:806}--\eqref{est:809}
allow us to conclude that there exist
\begin{align*}
&
{p} \in L^2(0,T;{W}^{1,2}(\Omega)) \cap L^{\infty}({\Omega_T}),
\\
&
{\vartheta} \in L^2(0,T;{W}^{1,2}(\Omega)) \cap L^{\infty}({\Omega_T}),
\\
&
{r} \in L^2(0,T;{W}^{1,2}(\Omega)) \cap L^{\infty}({\Omega_T}),
\end{align*}
such that,
letting $n \rightarrow +\infty$ (along a selected subsequence),
\begin{align*}
\bar{p}_{n} & \rightharpoonup  {p}
&&
\textrm{weakly in } L^2(0,T;{W}^{1,2}(\Omega)),
\\
\bar{\vartheta}_{n} & \rightharpoonup  \vartheta
&&
\textrm{weakly in } L^2(0,T;{W}^{1,2}(\Omega)),
%\label{conv_34}
\\
\bar{p}_{n}  & \rightharpoonup  {p}
&&
\textrm{weakly star in } L^{\infty}({\Omega_{T}}),
\\
\bar{p}_{n}  & \rightharpoonup  {p}
&&
\textrm{weakly star in } L^{\infty}({\partial\Omega_{T}}),
\\
\bar{\vartheta}_{n}  & \rightharpoonup  \vartheta
&&
\textrm{weakly star in } L^{\infty}({\Omega_{T}}),
\\
\bar{r}_{n} & \rightharpoonup  {r}
&&
\textrm{weakly in } L^2(0,T;W^{1,2}(\Omega)),
\\
\bar{r}_{n}  & \rightharpoonup  {r}
&&
\textrm{weakly star in } L^{\infty}({\Omega_T}),
\\
\bar{p}_{n} & \rightarrow  p
&&
\textrm{almost everywhere on } \Omega_T,
\\
\bar{\vartheta}_{n} & \rightarrow \vartheta
&&
\textrm{almost everywhere on } \Omega_T,
\\
\bar{r}_{n} & \rightarrow {r}
&&
\textrm{almost everywhere on } \Omega_T,
\\
\bar{\vartheta}_{n} & \rightharpoonup \vartheta
&&
\textrm{weakly in } L^{\ell}(\partial\Omega_{T}) \; (\forall  1 \leq  \ell < +\infty).
\end{align*}
Finally, \cite[Lemma 3]{FiloKacur1995} yields
\begin{eqnarray*}
\bar{\vartheta}_{n}  &\rightarrow& \vartheta
\qquad
\textrm{almost everywhere on } \partial\Omega_{T},
\label{conv_501}
\\
\bar{p}_{n}  &\rightarrow& {p}
\qquad
\textrm{almost everywhere on } \partial\Omega_{T}.
\label{conv_502}
\end{eqnarray*}
The above established convergences
are sufficient for taking the limit $n \rightarrow \infty$
in \eqref{eq:1001}--\eqref{eq:1004a}
(along a selected subsequence) to get the weak solution
of the system  \eqref{eq1a}--\eqref{eq1g}
in the sense of Definition~\ref{def:weak_form}.
\hfill $\square$
%%%%%%%%%%%%%%%%%%%%%%%%%%%%%%%%%%%%%%%%%%%%%%%%%%%%%%%%%%%%%%%%%%%%%%%%%%%%%%%%%%%%%

\subsection*{Acknowledgment}
The first author of this work has been supported by the project GA\v{C}R~16-20008S.

%%%%%%%%%%%%%%%%%%%%%%%%%%%%%%%%%%%%%%%%%%%%%%%%%%%%%%%%%%%%%%%%%%%%%%%%%%%%%%%%%%%%%

% ------------------------------------------------------------------------

\begin{thebibliography}{1}

%{\bibitem{Abdulle}
%A.~Abdulle, Y.~Bai,
%Reduced-order modelling numerical homogenization.
%Phil. Trans. R. Soc. A
%372 (2014) 20130388.
%}

\bibitem{AdamsFournier1992}
A.~Adams, J.F.~Fournier,
Sobolev spaces,
Pure and Applied Mathematics 140, Academic Press, 2003.

\bibitem{Allaire1992}
G.~Allaire,
Homogenization of two-scale convergence,
SIAM J. Math. Anal.
{23} (1992) 1482--1518.

\bibitem{AltLuckhaus1983}
H.W.~Alt, S.~Luckhaus,
Quasilinear elliptic-parabolic differential equations,
Math. Z.
{183} (1983) 311--341.

\bibitem{amaziane2010a}
B.~Amaziane, S.~Antontsev, L.~Pankratov, A.~Piatnitski,
Homogenization of immiscible compressible two-phase flow in porous
media: Application to gas migration in a nuclear waste repository,
Multiscale Model. Sim.
8 (2010) 2023--2047.


\bibitem{amaziane2013a}
B.~Amaziane, M.~Jurak, A.~Vrba\v{s}ki,
Homogenization results for a coupled system modelling immiscible
compressible two-phase flow in porous media by the concept of global pressure,
Appl. Anal.
92 (2013) 1417--1433.

%
%
%\bibitem{Aubin1963}
%{J.P.~Aubin},
%{Un th\'{e}or\`{e}me de compacit\'{e}},
%C. R. Acad. Sci. Paris
%{256} (1963) 5042--5044.
%
%
%\bibitem{BaKa1996}% book
%Z.P.~Ba\v{z}ant,  M.F.~Kaplan,
%Concrete at high temperatures: material properties and mathematical models,
%Concrete design and construction series, Longman, 1996.
%
%
%{
%\bibitem{bazant1972}
%Z.P.~Ba\v{z}ant, L.J.~Najjar,
%\emph{Nonlinear water diffusion in nonsaturated concrete},
%Materials and Structures,
%5 (1972) 3--20.
%}

\bibitem{bear}%book
J.~Bear,
Dynamics of fluids in porous media,
Amer. Elsevier, New York, 1979.


\bibitem{BenesKrupicka2015}
M.~Bene\v{s}, L.~Krupi\v{c}ka,
On doubly nonlinear elliptic-parabolic systems arising in coupled transport phenomena in unsaturated porous media,
Nonlinear Differ. Equ. Appl.
22 (2015) 121--141.

%\bibitem{BenesZeman}% article
%{M.~Bene\v{s}, J.~Zeman},
%{Some properties of strong solutions to nonlinear heat and moisture transport in
%multi-layer porous structures},
%Nonlin. Anal. RWA
%{13} (2012) 1562--1580.

%\bibitem{beneskrupicka2016}
%M.~Bene\v{s}, L.~Krupi\v{c}ka,
%Weak solutions of coupled dual porosity flows in fractured rock mass
%and structured porous media,
%J. Math. Anal. Appl.
%433 (2016) 543--565.


\bibitem{BenesKrupicka2016}
M.~Bene\v{s}, L.~Krupi\v{c}ka,
{Weak solutions of coupled dual porosity flows in fractured rock mass and structured porous media},
J. Math. Anal. Appl.
433 (2016) 543--565.


%
%\bibitem{BenesPazanin2017}
%M.~Bene\v{s}, I.~Pa\v{z}anin,
%\emph{On existence, regularity and uniqueness of thermally coupled
%incompressible flows in a system of three dimensional pipes},
%Nonlinear Analysis,
%149 (2017) 56--80.

\bibitem{BenesPazanin2017b}
M.~Bene\v{s}, I.~Pa\v{z}anin,
{Homogenization of degenerate coupled fluid flows and heat
transport through porous media},
J. Math. Anal. Appl.
446 (2017) 165--192.


\bibitem{BenesPazanin2018}
M.~Bene\v{s}, I.~Pa\v{z}anin;
On degenerate coupled transport processes in porous media with memory
phenomena.
Z Angew Math Mech.
98 (2018) 919--944.


\bibitem{BenesZeman}% article
{M.~Bene\v{s}, J.~Zeman},
{Some properties of strong solutions to nonlinear heat and moisture transport in
multi-layer porous structures},
Nonlinear Anal-Real.
{13} (2012) 1562--1580.


%
%\bibitem{brooks1966}
%R.H.~Brooks and A.T.~Corey,
%Properties of porous media affecting fluid flow,
%J. of the Irrigation and Drainage Division Proc. ASCE,
%92 (1966) 61--88.

\bibitem{cao2013}% article
H.~Cao, X.~Yue,
Homogenization of a Nonlinear Degenerate Parabolic Differential Equation,
Acta. Math. Sin.
29 (2013) 1429--1436.

\bibitem{cao2014}% article
H.~Cao, X.~Yue,
Homogenization of Richard's equation of van Genuchten-Mualem model,
J. Math. Anal. Appl.
412 (2014) 391--400.


%
%\bibitem{cao2015}% article
%H. Cao, T. Yu, and X. Yue,
%Fully discrete IPDG–HMM for multiscale Richards equation of unsaturated flow in porous media,
%J. Comput. Appl. Math.
%290 (2015) 352--369.
%
%
%\bibitem{Cervera1999}% article
%M.~Cervera, J.~Olivier, T.~Prato,
%\emph{A thermo-chemo-mechanical model for concrete. I: hydration and aging},
%Journal of Engineering Mechanics (ASCE),
%125 (1999) 1018--1027.
%
%
%\bibitem{Cervera2000}% article
%M.~Cervera, J.~Olivier, T.~Prato,
%\emph{Simulation of Construction of RCC Dams. I: Temperature and Aging},
%Journal of Structural Engineering,
%126 (2000) 1053--1061.
%
%
%\bibitem{Cervera2002}% article
%M.~Cervera, R.~Faria, J.~Olivier, T.~Prato,
%\emph{Numerical modelling of concrete curing, regarding hydration and temperature phenomena},
%Computers and Structures,
%80 (2002) 1511--1521.

%
%{\bibitem{Chen}
%Z.~Chen, W.~Deng, H.~Ye,
%Upscaling of a class of nonlinear parabolic equations for the flow transport in heterogeneous porous media,
%Comm. Math. Sci.
%3 (2005) 493--515.
%}



\bibitem{cioranescu}
D.~Cioranescu, P.~Donato,
An Introduction to Homogenization,
Oxford University Press, 1999.

\bibitem{Clark}
G.W.~Clark and R.E.~Showalter,
Two-scale convergence of a model for flow in a
partially fissured medium,
Electron. J. Diff. Eq.
1999 (1999) 1--20.






%
%\bibitem{Cook1999}% article
%R.A.~Cook, K.C.~Hover,
%\emph{Mercury porosimetry of hardened cement pastes},
%Cement and Concrete Research,
%29 (1999) 933--943.
%
%
%\bibitem{Couture}
%F.~Couture, W.~Jomaa, J.R.~Puiggali,
%Relative permeability relations: A key factor for a drying model,
%Transport Porous Med.
%23 (1996) 303--335.
%
%\bibitem{Deangelis1998}
%M.L.~Deangelis, E.F.~Wood,
%A detailed model to simulate heat and moisture transport in a frozen soil,
%IAHS-Aish P.
%248 (1998).
%
%
%
%
%
%
%
%\bibitem{degond}
%P.~Degond, S.~G\'{e}nieys, A.~J\"{u}ngel,
%\emph{A system of parabolic equations in nonequilibrium thermodynamics
%including thermal and electrical effects},
%Journal de Math\'{e}matiques Pures et Appliqu\'{e}es,
%76 (1997) 991--1015.
%
%
%\bibitem{Doktor1985}
%A.~Doktor,
%\emph{On the solution of the heat equation with nonlinear unbounded memory},
%Applications of Mathematics,
%30 (1985) 461--474.
%
%
%
%
%
%
%
%
%
%
%
%\bibitem{Fatima2010}% article
%T.~Fatima, N. Arab, E.P.~Zemskov, A.~Muntean,
%Homogenization of a reaction--diffusion system modeling sulfate corrosion of concrete in locally periodic perforated domains,
%J. Eng. Math.
%69 (2011) 261--276.
%
%
%
%
%
%
%
%\bibitem{Filo1987}
%J.~Filo,
%\emph{On solutions of a perturbed fast diffusion equation},
%Aplikace matematiky,
%32 (1987)  364--380.



\bibitem{FiloKacur1995}
J. Filo, J. Ka\v{c}ur,
Local existence of general nonlinear parabolic systems,
Nonlinear Anal.
24 (1995) 1597--1618.


{\bibitem{gallouet}
T.~Gallou\"{e}t, A. Monier,
On the regularity of solutions to elliptic equations,
Rendiconti Mat.
19 (1999) 471--488.
}


%
%{
%\bibitem{Gawin1999}
%D.~Gawin, C.~Majorana and B.A.~Schrefler,
%\emph{Numerical analysis of hygro-thermal behaviour and damage of concrete at high temperature},
%Mechanics of Cohesive-frictional Materials,
%4 (1999) 37--74.
%}
%
%\bibitem{Gawin2006a}
%D.~Gawin, F.~Pesavento and B.A.~Schrefler,
%\emph{Hygro-thermo-chemo-mechanical modelling of concrete at early
%ages and beyond. Part I: Hydration and
%hygro-thermal phenomena},
%International Journal For Numerical Methods In Engineering,
%67 (2006) 299--331.
%
%\bibitem{Gawin2011a}
%D.~Gawin, F.~Pesavento, B.~Schrefler,
%\emph{What physical
%phenomena can be neglected when modelling concrete at high temperature? {A}
%comparative study. {P}art 1: Physical phenomena and mathematical model},
%International Journal of Solids and Structures,
%48 (2011) 1927--1944.
%
%\bibitem{Gawin2011b}
%D.~Gawin, F.~Pesavento, B.~Schrefler,
%\emph{What physical
%phenomena can be neglected when modelling concrete at high temperature? {A}
%comparative study. {P}art 2: Comparison between models.}
%International Journal of Solids and Structures,
%48 ~(2011)  1945--1961.
%
%
%\bibitem{Genuchten1980}
%M.~van Genuchten,
%A closed form equation for predicting the hydraulic
%conductivity of unsaturated soil,
%Soil Sci. Soc. Am. J.
%44 (1980) 892--898.
%
%\bibitem{gerkegenuchten}
%H.~Gerke, M.~Van~Genuchten,
%A dual-porosity model for simulating the
%preferential movement of water and solutes in structured porous media,
%Water Resour. Res. 29 (1993) 305--319.
%
%\bibitem{gerkegenuchten_2}
%H.~Gerke, M.~Van~Genuchten,
%Evaluation of the first order transfer term for
%variably saturated dual porosity flow models,
%Water Resour. Res. 29 (1993) 1225--1238.
%
%
%
%\bibitem{Gilbarg}
%D.~Gilbarg,  N.S.~Trudinger,
%Elliptic Partial Differential Equations of Second Order,
%Springer-Verlag Berlin Heidelberg, 2001.
%
%
%
%%
%%\bibitem{hassanizadeh1979a}
%%S.M. Hassanizadeh, W.G. Gray,
%%General conservation equations for multi-phase systems: 1. Averaging procedure,
%%Adv. Water Resour. 2 (1979) 131--144.
%%
%%\bibitem{hassanizadeh1979b}
%%S.M. Hassanizadeh, W.G. Gray,
%%General conservation equations for multi-phase systems: 2. Mass, momenta, energy and entropy equations,
%%Adv. Water Resour. 2 (1979) 191--203.
%%
%%\bibitem{hassanizadeh1980}
%%S.M. Hassanizadeh, W.G. Gray,
%%General conservation equations for multi-phase systems: 3. Constitutive theory for porous media flow,
%%Adv. Water Resour. 3 (1980) 25--40.
%
%
%
%
%
%
%
%\bibitem{GiaquintaModica1987}
%M.~Giaquinta, G.~Modica;,
%\emph{Local existence for quasilinear
%parabolic systems under nonlinear boundary conditions},
%Annali di Matematica Pura ed Applicata,
%{149} (1987)  41--59.
%
%
%
%
%
%\bibitem{GiTru}% book
%Gilbarg,~D., Trudinger,~N.S.:
%\textit{Elliptic Partial Differential Equations of Second Order.}
%Springer (2001)
%
%
%
%\bibitem{groger1989}
%K.~Gr\"{o}ger,
%\emph{A $W^{1,p}$-estimate for solutions to mixed boundary value
%problems for second order elliptic differential equations},
%Mathematische Annalen,
%283 (1989) 679--687.
%
%
%
%
%
%
%
%{
%
%\bibitem{Halamickova1995}
%P.~Halamickova, R.J.~Detwiler, D.P.~Bentz, E.J.~Garboczi,
%\emph{Water Permeability And Chloride Ion Diffusion
%In Portland Cement Mortars:
%Relationship To Sand Content And Critical Pore Diameter},
%Cement and Concrete Research,
%25 (1995) 790--802.
%}
%
%\bibitem{harris2017}
%P.A.~Harris, E.N.M.~Cirillo, A.~Muntean,
%\emph{Weak solutions to Allen-Cahn-like equations modelling consolidation of porous media},
%IMA Journal of Applied Mathematics
%82 (2017) 224--250.
%
%
%
%
%
%
%{\bibitem{heuser}%book
%P.~Heuser,
%Homogenization of Quasilinear Elliptic–Parabolic Equations with Respect to Measures,
%Ruprecht-Karls-Universität, Heidelberg, 2012.
%}
%
%
%\bibitem{hornung}%book
%U.~Hornung,
%Homogenization and porous media,
%Springer-Verlag New York, 1997.
%
%
%
%
%\bibitem{IshidaMaekawaKishi2007}
%T.~Ishida, K.~Maekawa, T.~Kishi,
%\emph{Enhanced modeling of moisture equilibrium and transport in cementitious
%materials under arbitrary temperature and relative humidity history},
%Cement and Concrete Research,
%37 (2007) 565--578.




\bibitem{jian2000}
H.~Jian,
On the homogenization of degenerate parabolic equations,
Acta Math. Appl. Sin.-E
16 (2000) 100--110.




\bibitem{jungel2000}
A.~J\"{u}ngel,
\emph{Regularity and uniqueness of solutions to a parabolic system in nonequilibrium thermodynamics},
Nonlinear Analysis,
41 (2000)  669--688.









\bibitem{Kacur1990a}
J.~Ka\v{c}ur,
On a solution of degenerate elliptic-parabolic systems in Orlicz-Sobolev spaces. I,
Math. Z.
203 (1990) 153--171.


\bibitem{Kacur1999}
J.~Ka\v{c}ur,
\emph{Solution to strongly nonlinear parabolic problems by a linear approximation scheme},
IMA Journal of Numerical Analysis,
19 (1999)  119--145.


\bibitem{Kacur1999b}
J.~Ka\v{c}ur,
\emph{Solution of convection-diffusion problems with the memory terms},
in Applied Mathematical Analysis, A. Sequiera,
H. Beirao de Veiga, and J.H. Videman, Eds., Kluwer Academic, Plenum Publ., New York (1999)
199--212.




%
%
%\bibitem{Kacur2001}
%J.~Ka\v{c}ur,
%\emph{Solution of Degenerate Convection-Diffusion Problems by the Method of Characteristics},
%SIAM Journal on Numerical Analysis,
%39 (2001) 858--879.
%
%\bibitem{KacurKeer2003}
%J.~Ka\v{c}ur, R.~Van~Keer,
%Solution of degenerate parabolic variational inequalities with convection,
%ESAIM-Math. Model. Num.
%37 (2003) 417--431.

\bibitem{Khoa2016}
V.A.~Khoa, A.~Muntean,
Asymptotic analysis of a semi-linear elliptic system in perforated domains: Well-posedness and correctors for the homogenization limit,
J. Math. Anal. Appl.
439 (2016) 271--295.


\bibitem{Krehel2014}
O.~Krehel, T.~Aiki, A.~Muntean,
Homogenization of a thermo-diffusion system with Smoluchowski interactions,
Netw. Heterog. Media
9 (2014) 739--762.



\bibitem{Krehel2015}
O.~Krehel, A.~Muntean, P.~Knabner,
Multiscale modeling of colloidal dynamics in porous media including aggregation and deposition,
Adv. Water Resour.
86 (2015) 209--216.

\bibitem{KufFucJoh1977}% book
{A.~Kufner,  O.~John, S.~Fu\v{c}\'{i}k},
{Function spaces},
Academia, 1997.














%\bibitem{Liu2013}
%Z.~Liu, X.~Yu,
%\emph{Multiscale Chemo-Thermo-Hydro-Mechanical Modeling of Early Stage Hydration and
%Shrinkage of Cement Compounds},
%Journal of Materials in Civil Engineering,
%25 (2013)  1239--1247.
%
%\bibitem{LadUr}% book
%O.A.~Ladyzhenskaya, N.N.~Uraltseva,
%Linear and Quasilinear Elliptic Equations,
%Academic Press New York and London, 1968.
%
%
%
%\bibitem{LadSolUr}
%Ladyzhenskaya,~O.A., Solonnikov,~V.A., Uraltseva,~N.N.:
%\textit{Linear and quasilinear equations of parabolic type.}
%Translations of Mathematical Monographs 23,
%American Mathematical Society, Providence, R.I. (1967)




\bibitem{LiSun2010}
B.~Li, W.~Sun,
Global existence of weak solution for nonisothermal multicomponent flow
in porous textile media,
SIAM J. Math. Anal.
42 (2010) 3076--3102.

%
%\bibitem{LiSunWang2010}
%B.~Li, W.~Sun, Y.~Wang,
%\emph{Global existence of weak solution to
%the heat and moisture transport system in fibrous porous media},
%Journal of Differential Equations,
%249 (2010)  2618--2642.
%
%\bibitem{LiSun2012}
%B.~Li, W.~Sun,
%\emph{Global weak solution for a heat and sweat transport
%system in three-dimensional fibrous porous media with condensation/evaporation
%and absorption},
%SIAM Journal on Mathematical Analysis,
%{44} (2012)  1448--1473.
%
%
%{\bibitem{liu}
%K.F.~Liu, Y.H.~Wu, Y.C.~Hsu,
%Homogenization theory applied to unsaturated solid-liquid mixture,
%J. Mech.
%28 (2012) 329--335.
%}
%
%
%\bibitem{Liu2013}
%Z.~Liu, X.~Yu;
%\emph{Multiscale Chemo-Thermo-Hydro-Mechanical Modeling of Early Stage Hydration and
%Shrinkage of Cement Compounds},
%Journal of Materials in Civil Engineering,
%25 (2013), 1239--1247.
%
%
%\bibitem{Maekawa3}
%K.~Maekawa, T.~Ishida, T.~Kishi,
%\emph{Multi-scale modeling of concrete performance},
%Journal of Advanced Concrete Technology,
%1 (2003) 91--126.
%
%
%\bibitem{Maekawa4}
%K.~Maekawa, R.~Chaube, T.~Kishi,
%\emph{Modelling of concrete performance : hydration, microstructure formation, and mass transport}.
%London ; New York : E \& FN Spon, 1999.



%\bibitem{majorana_salomoni_schrefler_1998}
%C.E.~Majorana, V.~Salomoni, B.A.~Schrefler,
%Hygrothermal and mechanical model of concrete at high temperature,
%Mater. Struct.
%31 (1998) 378--386.

%\bibitem{MazRoss2010}
%V.G.~Mazya, J.~Rossmann,
%Elliptic equations in polyhedral domains,
%Mathematical Surveys and Monographs Vol. 162,
%2010.
%
%\bibitem{merz}
%{ W.~Merz and P.~Rybka},
%Strong solution to the Richards equation in the unsaturated zone,
%J. Math. Anal. Appl. \textbf{371} (2010), 741--749.


%
%{\bibitem{mikelic}
%A.~Mikeli\'{c}, C.~Rosier,
%Modeling solute transport through unsaturated porous media using homogenization I,
%Comput. Appl. Math.
%23 (2004) 195--211.
%}



\bibitem{Nandakumaran2001}
A.K.~Nandakumaran, M.~Rajesh,
Homogenization of a nonlinear degenerate parabolic differential equation,
Electron. J. Diff. Eq.
49 (2001) 1--19.

%\bibitem{Nandakumaran2002a}
%A.K.~Nandakumaran and M.~Rajesh,
%Homogenization of a parabolic equation in perforated domain with Neumann boundary condition,
%Proceedings of the Indian Academy of Sciences - Mathematical Sciences,
%112 (2002) 195--207.
%
%\bibitem{Nandakumaran2002b}
%A.K.~Nandakumaran and M.~Rajesh,
%Homogenization of a parabolic equation in perforated domain with Dirichlet boundary condition,
%Proceedings of the Indian Academy of Sciences - Mathematical Sciences,
%112 (2002) 425--439.
%
%
%\bibitem{necas1967}
%J.~Ne\v{c}as,
%\emph{Les methodes directes en theorie des equations elliptiques},
%Academia, Prague 1967.
%
%
%\bibitem{Necas1983}
%J.~Ne\v{c}as,
%Introduction to the Theory of Nonlinear Elliptic Equations,
%Teubner-Texte zur Mathematik, Leipzig, 1983.
%
%
%{\bibitem{Cirpka}
%I.~Neuweiler, O.A.~Cirpka,
%Homogenization of Richards equation in permeability fields with different connectivities,
%Water Resour. Res.
%41 (2005) W02009.
%}

\bibitem{Nguetseng}
G.~Nguetseng, 	
A general convergence result for a functional related to the theory of homogenization.
SIAM J. Math. Anal.
20 (1989) 608--623.

%\bibitem{Nield}% book
%D.A.~Nield, A.~Bejan,
%Convection in Porous Media.
%Springer, 2013.
%
%\bibitem{Ning1990c}
%S.~Ning,
%\emph{Mathematical problems on the fluid-solute-heat flow through porous media,
%I. Unsaturated case},
%Acta Mathematicae Applicatae Sinica,
%6 (1990) 135--144.
%
%\bibitem{Ning1990}
%S.~Ning,
%\emph{Mathematical problems on the fluid-solute-heat flow through porous media,
%II. Partially saturated case},
%Acta Mathematicae Applicatae Sinica,
%6 (1990) 145--157.
%
%\bibitem{Ning1990b}
%S.~Ning,
%\emph{An elliptic-parabolic system arising from the fluid-solute-heat flow
%through saturated porous media},
%Acta Mathematicae Applicatae Sinica,
%6 (1990) 224--237.
%
%\bibitem{Ning1993}
%S.~Ning,
%\emph{Multidimensional degenerate diffusion problem with evolutionary boundary condition:
%existence, uniqueness, and approximation},
%In: Flow in Porous Media,
%Volume 114 of the series ISNM International Series of Numerical Mathematics,
%(1993) 165--178.
%
%\bibitem{Ozbolt}
%J.~O\v{z}bolt, G.~Balabani\'{c}, G.~Peri\v{s}ki\'{c}, M.~Ku\v{s}ter,
%\emph{Modelling the effect of damage on transport processes in concrete},
%Construction and Building Materials,
%24 (2010)  1638--1648.



\bibitem{otto}
F.~Otto,
${L}^{1}$-Contraction and Uniqueness for Quasilinear Elliptic–Parabolic Equations,
J. Differ. Equations
131 (1996) 20-38.



{\bibitem{Peszynska}
M.~Peszy\'{n}ska, R.~Showalter, S.-Y.~Yi,
Homogenization of a pseudoparabolic system,
Appl. Anal.
88 (2009) 1265--1282.
}

\bibitem{Pinder}
G.F.~Pinder, W.G.~Gray,
Essentials Of Multiphase Flow And Transport In Porous Media,
John Wiley \& Sons, 2008.

%
%\bibitem{Pluschke1988}
%V.~Pluschke;
%\emph{Local solution of parabolic equations with strongly increasing nonlinearity
%by the Rothe method},
%Czechoslovak Mathematical Journal,
%38 (1988) 642--654.
%
%\bibitem{Pluschke1992}
%V.~Pluschke,
%\emph{Rothe's Method for Semilinear Parabolic Problems with Degeneration},
%Mathematische Nachrichten,
%156 (1992) 283--295.
%
%\bibitem{Pluschke1997}
%V.~Pluschke,
%\emph{Rothe's Method for Degenerate Quasilinear Parabolic Equations},
%In: Zuzana Dosla and J. Kuben and Jaromir Vosmansky (eds.): Proceedings of Equadiff 9,
%Conference on Differential Equations and Their Applications, Brno, August 25-29, 1997,
%[Part~3] Papers. Masaryk University, Brno, 1998. CD-ROM,
%pp. 247--254.
%
%
%{
%
%\bibitem{Powers}
%G.~Powers, T.L.~Brownyard,
%\emph{Studies of the physical properties of hardened Portland cement paste.
%Portland Cement Association},
%Res. Bull.
%22 (1948).
%
%}
%
%
%
%
%
%
%
%
%\bibitem{Rektorys}
%K.~Rektorys,
%\emph{The method of discretization in time and partial differential equations},
%Reidel Co, Dodrecht, Holland 1982.
%
%{\bibitem{Renard}
%Ph.~Renard, G.~de Marsily,
%Calculating equivalent permeability: A review,
%Adv. Water Resour.
%20 (1997) 253--278.
%}

\bibitem{Roubicek2005}
T.~Roub\'{i}\v{c}ek,
Nonlinear Partial Differential Equations with Applications,
Birkh\"{a}user, 2005.

%
%\bibitem{Scanlon1994}
%B.R.~Scanlon, P.C.D.~Milly,
%Water and heat fluxes in desert soils: 2. Numerical simulations,
%Water Resour. Res.
%30 (1994) 721--733.
%
%
%
%\bibitem{Song2201}
%H.-W.~Song, H.-J.~Cho, S.-S.~Park, K.-J.~Byun, K.~Maekawa,
%\emph{Early-age cracking resistance evaluation of concrete structures},
%Concrete Science and Engineering,
% 3  (2001)  63--72.
%
%
%
%
%{\bibitem{Sviercoski}
%R.F.~Sviercoski, A.W.~Warrick, C.L.~Winter,
%Two-scale analytical homogenization of Richards' equation for flows through block inclusions,
%Water Resour. Res.
%45 (2009) W05403.
%
%
%\bibitem{sykora2012}
%J.~S\'{y}kora, T.~Krej\v{c}\'{i}, J.~Kruis, M.~\v{S}ejnoha.
%Computational homogenization of non-stationary transport processes in masonry structures,
%J. Comput. Appl. Math.
%236 (2012) 4745--4755.
%
%\bibitem{sykora2013}
%J.~S\'{y}kora, M.~\v{S}ejnoha, J.~\v{S}ejnoha.
%Homogenization of coupled heat and moisture transport in masonry
%  structures including interfaces,
%Appl. Math. Comput.
%219 (2013) 7275--7285.
%}
%
%\bibitem{Thomas1995}
%H.R.~Thomas, M.R.~Sansom,
%Fully coupled analysis of heat, moisture and air transfer in unsaturated soil,
%J. Engng. Mech.
%121 (1995) 392--405.
%
%
%{
%
%\bibitem{ulm1995}
%F.J.~Ulm, O.~Coussy,
%\emph{Modeling of thermo-chemo-mechanical couplings of concrete at early ages},
%Journal of Engineering Mechanics (ASCE),
%121 (1995) 785--794.
%
%\bibitem{ulm1996}
%F.J.~Ulm, O.~Coussy,
%\emph{Strength growth as chemo-plastic hardening in early age concrete},
%Journal of Engineering Mechanics (ASCE),
%122 (1996) 1123--1132.
%
%\bibitem{ulm1998}
%F.J.~Ulm, O.~Coussy,
%\emph{Couplings in early-age concrete: From material modeling to structural design},
%International Journal of Solids and Structures,
%35 (1998) 4295--4311.
%
%}
%
%
%
%\bibitem{Vala2002}
%J.~Vala,
%\emph{On a system of equations of evolution with a
%non-symmetrical parabolic part occuring in the analysis of moisture
%and heat transfer in porous media},
%Applications of Mathematics,
%{47} (2002) 187--214.
%
%
%\bibitem{Weidemaier1991}
%P.~Weidemaier,
%\emph{Local existence for parabolic problems with
%fully nonlinear boundary condition; An $L_p$-approach},
%Annali di Matematica Pura ed Applicata,
%{160} (1991)  207--222.
%
%
%\bibitem{Wu2006}
%Z.~Wu, J.~Yin, Ch.~Wang,
%\emph{Elliptic and Parabolic Equations},
%World Scientific, 2006.



\end{thebibliography}
\end{document}